\documentclass{cocv}
\usepackage{color}
\usepackage{amsmath}
\usepackage{amssymb}
\usepackage{algorithmic}
\usepackage{algorithm,algorithmic}
\usepackage{graphicx}   % µ¼ÈëͼƬ
\usepackage{epstopdf}
\usepackage{graphics}
\usepackage{subfigure}
\graphicspath{{eps/}}
\usepackage{lineno,hyperref}
\usepackage{tikz}
\usetikzlibrary{shapes,arrows,chains,backgrounds,fit,decorations.pathreplacing}
\usepackage[europeanresistors,americaninductors]{circuitikz}
\usetikzlibrary{shapes,arrows}

\newenvironment{pf2}{{\it Proof  of Theorem~\Rref{original main result}: \enspace}}{\hfill $\square$\par}

\newenvironment{pf5}{{\it Proof  of Theorem~\Rref{solution exsistence} : \enspace}}{\hfill $\square$\par}

  \allowdisplaybreaks[4]

\begin{document}
	%%-----------------------------
	%%      the top matter
	%%-----------------------------
	\title{Input-to-state stabilization  of an   ODE cascaded with a parabolic equation  involving Dirichlet-Robin boundary disturbances}%\thanks{This work is supported in part by the NSFC under grant 11901482 and in part by the NSERC under grant RGPIN-2018-04571.}% At most 5 thanks
	\author{Yongchun Bi}\address{State Key Laboratory of Rail
		Transit Vehicle System, Southwest Jiaotong University,
		Chengdu 611756, Sichuan, China}\sameaddress{,2}
	\author{Jun Zheng}\address{School of Mathematics, Southwest Jiaotong University,
		Chengdu 611756, Sichuan, China}\sameaddress{,3}
	\author{Guchuan Zhu}\address{Department of Electrical Engineering, Polytechnique Montr\'{e}al, P. O. Box 6079, Station Centre-Ville, Montreal, QC, Cannada H3T 1J4}\thanks{{Email addresses: yongchun.bi@my.swjtu.edu.cn $(\mathrm{Yongchun}$ $\mathrm{Bi}^{1,2})$, zhengjun2014@aliyun.com $(\mathrm{Jun}$ $\mathrm{Zheng}^{2,3})$, guchuan.zhu@polymtl.ca $(\mathrm{Guchuan}$ $\mathrm{Zhu}^3)$}}\thanks{{Corresponding author: zhengjun2014@aliyun.com}}

	\
	%
%	\date{January 2025}
	%
	\begin{abstract}
		%±¾ÎÄ¿¼ÂǾßÓÐDirichletÈŶ¯µÄcascaded ODE with parabolic PDEµÄÊäÈë״̬Õò¶¨ÎÊÌâ, ÆäÖÐPDEµÄ×ó¶ËλÒƽøÈëODE£¬±ß½ç·´À¡¿ØÖÆ×÷ÓÃÓÚÅ×Îï·½³Ì²¿·Ö£¬²¢Í¨¹ý±ß½çÊäÈë¼ä½ÓÓ°Ïì΢·Ö·½³Ì¶¯Á¦Ñ§. A boundary feedback control law is constructed by backstepping. Input-to-state stability of the closed loop is achieved by generalized Lyapunov method under the control law.¡¡ÀûÓÃËã×Ó°ëȺÀíÂ۵õ½±Õ»·ÏµÍ³½âµÄ´æÔÚÐÔ. ×îºó, ͨ¹ýÊýֵģÄâÑéÖ¤Á˱߽ç¿ØÖÆ·½·¨µÄÓÐЧÐÔ¡£
		This paper focuses on the input-to-state stabilization problem for an ordinary differential equation (ODE)  cascaded by  parabolic partial differential equation (PDE) in the presence of Dirichlet-Robin boundary  disturbances, as well as in-domain disturbances.  For the cascaded system with a Dirichlet pointwise interconnection, the ODE takes the value of a Robin boundary condition at the ODE-PDE interface as its direct input, and the PDE is driven by a Dirichlet boundary input at the opposite end.
		%For the cascaded system,  the PDE subsystem satisfies a Robin boundary condition at the ODE-PDE interface,   whose value   acts directly as the input to the ODE,  while the PDE is driven by a Dirichlet boundary input at the opposite end.
		We first employ the backstepping method to design a boundary controller and to decouple the cascaded system. This decoupling facilitates independent stability analysis of the PDE and ODE systems sequentially. Then, to address the challenges posed by Dirichlet boundary disturbances to the application of the classical Lyapunov method, we utilize the generalized Lyapunov method to establish the ISS in the max-norm for the cascaded system involving Dirichlet boundary disturbances and two other types of   disturbances.
		%DirichletÐÍ»¥ÁªÔ¼ÊøPDE״̬the left endpoint value of the PDE solution is fed into the ODE while the value of the ODE solution is   fed into the Dirichlet boundary controller, which is designed by using   the backstepping method.
		%   Taking into account the complexity of the considered system,  we decouple the target system when using  backstepping. This decoupling facilitates independent stability analyses of the PDE and ODE systems sequentially.
		%in the presence of Dirichlet-Robin boundary disturbances and in-domain disturbances.
		% Regarding stability analysis of the target PDE system, to address the challenges posed by Dirichlet boundary disturbances to the application of the classical Lyapunov method, we utilize the generalized Lyapunov method to establish the ISS in the max-norm for the target system involving Dirichlet boundary disturbances and two other types of   disturbances.
		The obtained result  indicates that even in the presence of   different types of disturbances, ISS analysis can still be conducted within the framework of Lyapunov stability theory.  %Notably, the presence of endpoint term of PDE's solution in the ODE system requires establishing a stability estimate for the PDE system in the sense of the max-norm. This step is crucial for addressing the endpoint term  in the ODE system and subsequently establishing the  ISS in the max-norm for the original cascaded system.
		For the {well-posedness} of the target system, it is conducted by using the technique of lifting and the semigroup method. Finally, numerical simulations are conducted to illustrate the effectiveness of the proposed  control scheme and ISS properties for a cascaded system with different disturbances. %\textcolor{red}{(to be simplified)}

		%For the coupled  system, the left-end displacement of the PDE enters the ODE. The boundary feedback control acts on the parabolic equation part and indirectly influences the dynamics of the differential equation through the boundary input.   A boundary feedback control law is constructed by backstepping and the well-posedness of the solutions of the closed-loop system in the space of continuous functions is obtained by the operator semigroup theory in this paper. Subsequently, the input-to-state stability (ISS) of the closed-loop system is achieved via the generalized Lyapunov method under the control law.  Finally, the effectiveness of the boundary control method is verified through numerical simulations.
	\end{abstract}
	%
	%\begin{resume} ... \end{resume}
	%
	\maketitle

\tableofcontents
	%%-----------------------------
	%%      your text
	%%-----------------------------

	\section{Introduction}
	The model of cascaded  ODE-PDE systems can be applied to a wide range of engineering problems, such as flexible robotic arms \cite{Cheng2024INJRC}, road traffic \cite{Hasan2016}, and metal rolling processes \cite{Diagne2017TAC}, etc.
	Over the past decades, the stabilization for such systems has been extensively studied in the   literature; see, e.g., \cite{Antonio2010,Cai2019,Diagne2017TAC,Guo2025,Krstic2008scl,krstic2009,Krstic2009TAC,Krstic2009SCL,Nikdel2021,wang2025TAC}. It should be noted that  in practical applications, system  stability is often compromised by complex environmental conditions, such as modeling errors, systems' uncertainties, and external disturbances. To address this,  the input-to-state stability (ISS) theory---first introduced by Sontag for finite-dimensional nonlinear systems \cite{sontag1990further} and  later extended to infinite-dimensional systems---provides a crucial framework for characterizing the influence of external inputs on stability, holding both theoretical and practical significance  {\cite{Mironchenko2023book, Mironchenko2020, Karafyllis2018book}}.
	Consequently,  within the framework of ISS, stabilization for cascaded ODE-PDE systems  has attracted considerable   attention in recent years; see, e.g., \cite{zhang2020,zhang2025IMA,abdallah2025,Zhang2021,Zhang2019}.
	
%	Notably, in \cite{Zhang2019} and \cite{zhang2020}, the authors  designed boundary  backstepping controllers for bidirectional pointwise interconnection systems in which an ODE is coupled with a heat and a wave equation, respectively,   and in the presence of  disturbances   acting solely within the ODE and established  the ISS  in the $L^2$-norm for the closed-loop system by using the Lyapunov method.
   Notably,  when external disturbances  appear   only in the ODE plants,   the authors  of   \cite{Zhang2019} and \cite{zhang2020} considered     an ODE cascaded with the heat equation  and  the wave equation,    respectively, having the following forms:
  %	\begin{subequations}
  		\begin{align*}
  			\dot{X}(t)=&A X(t)+B u(0, t)+B_1 d(t),  t>0, \\
  			u_t(z,t)=&u_{zz}(z,t)+c u(z,t),  0<z<1, t>0, \\
u_z(0, t)=&C X(t),  t>0, \\
  u_z(1, t)=&U(t),  t>0,
  		\end{align*}
 % 		\end{subequations}
  and
%\begin{subequations}
	\begin{align*}
		\dot{X}(t)=&A X(t)+B u(0, t)+B_1 d(t),  t>0, \\
		u_{t t}(z,t)=&u_{zz}(z,t),  0<z<1, t>0, \\
		u_z(0, t)=&C X(t),  t>0, \\
		u_z(1, t)=&U(t),  t>0,
  \end{align*}
%    		\end{subequations}
   where  $d(t) \in \mathbb{R}$ represents  external disturbances, $u(0, t) \in \mathbb{R}$ and $X(t) \in \mathbb{R}^{n \times 1}$ are taken as inputs to  the ODEs and the Neumman boundary terms of the  PDEs, respectively,
   $U(t) \in \mathbb{R}$ is the control input,  $A \in \mathbb{R}^{n \times n}$, $B, B_1 \in \mathbb{R}^{n \times 1}, C \in \mathbb{R}^{1 \times n}$ are constant matrices, and $c>0$ is a constant. By using the method of backstepping, the authors  of \cite{Zhang2019} and \cite{zhang2020} designed     boundary  feedback controllers to ensure the ISS    in the $L^2$-norm for the cascaded systems in closed loop.

  	% \textcolor{blue}{In \cite{zhang2020}, the authors  consider the ODE-wave feedback-connection system with bidirectional interconnection at a single point:}
%  	$$
%  	\begin{cases}\textcolor{blue}{\dot{X}(t)=A X(t)+B u(0, t)+B_1 d(t),} & t>0, \\ \textcolor{blue}{u_{t t}(z,t)=u_{zz}(z,t),} & 0<z<1, t>0, \\ \textcolor{blue}{u_z(0, t)=C X(t),} & t>0, \\ \textcolor{blue}{u_z(1, t)=U(t),} & t>0,\end{cases}
%  	$$
%  	\textcolor{blue}{where $X(t) \in \mathbb{R}^{n \times 1}$ and $u(z,t) \in \mathbb{R}$ are the states of ODE and wave PDE, respectively, $U(t) \in \mathbb{R}$ is the control input, $A \in \mathbb{R}^{n \times n}, B, B_1 \in \mathbb{R}^{n \times 1}, C \in \mathbb{R}^{1 \times n}, c>0$, and $d(t) \in \mathbb{R}$ is a bounded disturbance satisfying $|d(t)| \leq M_1$ where $M_1$ is a positive constant and $\dot{d}(t)$ is a piecewise continuous function on $\mathbb{R}^{+}$.}
%  	

  	% In \cite{Zhang2021},  the authors consider  an ODE cascaded by a reaction-diffusion equation with disturbances in both the ODE and PDE domains, as well as   Neumann boundary disturbances,  and design   boundary feedback controllers    by combining backstepping   with sliding mode control, with ISS  established    using the Lyapunov method  in the $L^2$-norm.
  	  When external disturbances appear  in both the ODEs    and the PDEs,     the authors  of \cite{Zhang2021} investigated   an ODE  cascaded with  a parabolic equation under the following form:
  	 \begin{subequations} \label{ODE-PDE-1}
  \begin{align}
  	\dot{X}(t) =& A X(t) + B u(0,t) + B_1 d_1(t),  t>0, \\
  	u_t(z,t) =& \varepsilon u_{zz}(z,t) + \phi(z) u(z,t) + \lambda(z) u(0,t)+ \int_0^z l(z,y) u(y,t) \, dy + d_2(z,t),  0<z<1, t>0, \\
  	u_z(0,t) =& d_3(t), t>0,\\
  	u_z(1,t) =& U(t) + d_4(t), t>0,
  		%\textcolor{blue}{X(0) = X_0, \quad u(x,0) = u_0(x).}
  		\end{align}
  		  		\end{subequations}
  where $d_1(t), d_2(z,t), d_3(t), d_4(t)\in\mathbb{R} $ represents external disturbances,  $u(0, t) \in \mathbb{R}$  is taken as an input  to  the ODE while  $X(t) \in \mathbb{R}^{n \times 1}$ is not an input of the  PDE,  $A \in \mathbb{R}^{n \times n}$, $B, B_1 \in \mathbb{R}^{n \times 1} $ are constant matrices, $\varepsilon > 0$ is a constant, $\phi(z) ,\lambda(z), l(z,y)$ are  space-varying functions. Under the assumption that  $|d_i(t)|+\|d_2(\cdot,t)\|_{L^2(0,1)} \leq M $ with a positive constant $M $ for   $i=1, 3,4$,  the authors designed boundary feedback control law $U(t)$
by using   backstepping   and the sliding mode control with the aid of the parameter $M$, and established   the local ISS  estimate  in the $L^2$-norm for the closed-loop system  by using the Lyapunov method.
  %
%
%
%
%  and
%  	 \textcolor{blue}{where \( x \in (0,1) \), \( X(t) \in \mathbb{R}^{n \times 1} \) and \( u(z,t) \in \mathbb{R} \) are the states of ODE and reaction-diffusion equation, respectively.
%  	 \( A \in \mathbb{R}^{n \times n} \), \( B, B_1 \in \mathbb{R}^{n \times 1} \), \((A,B)\) is stabilizable and \( \varepsilon > 0 \) is a constant.
%  	 \(\phi, \lambda \in C^1([0,1])\) are the spatially varying functions, and \( l \in C^1([0,1] \times [0,1]) \).
%  	 \( U(t) \in \mathbb{R} \) is the control input.
%  	 The disturbances are bounded, i.e., \( |d_i(t)| < M \, (i=1,3,4) \) and \( \|d_2(\cdot,t)\|_{L^2(0,1)} < M \) for some \( M > 0 \) and all \( t \geq 0 \).
%  	 Moreover, \( d_i \in H^{1}_{\text{loc}}([0,\infty)) \) for \( i=1,3,4 \), \( d_2 \in H^{1}_{\text{loc}}([0,\infty); L^2(0,1)) \), and
%  	 \[
%  	 H^{1}_{\text{loc}}([0,\infty)) = \left\{ v: [0,\infty) \to \mathbb{R} \mid v \in H^1(V), \text{ for each } V \subset \overline V \subset [0,\infty) \text{ and } \overline V \text{ is compact} \right\}.
%  	 \]
%  	 Let \( L^2(0,1) \) be the Hilbert space with the usual inner product \(\langle\cdot,\cdot\rangle\) and norm \(\|\cdot\|\).}
  	  	    %In \cite{abdallah2025}, the authors design an output feedback controller by backstepping and achieve ISS by Lyapunov method  for an  ODE cascaded by heat PDE subject to in-domain disturbances in both subsystems and the Neumann boundary disturbance  in the $L^2$-norm.

Furthermore, the authors  of \cite{zhang2025IMA} extended the application of the combined approach of backstepping and sliding mode control        to the input-to-sate stabilization problem for   the cascaded system \eqref{ODE-PDE-1} when  $\varepsilon =1$, $\phi(z)$ becomes a constant, $\lambda(z)= d_2(z,t)=0$, and $d_3(t)$ is replaced with $CX(t)$, namely, $X(t)  $ is  also  input into  the Neumman boundary term  of the  PDE. Analogous to \cite{Zhang2021}, under the assumption that all   disturbances are bounded, the authors of  \cite{zhang2025IMA}  designed a boundary feedback controller to  ensure   the local ISS    in the $L^2$-norm for the closed-loop system.

%
%an ODE cased with a parabolic PDE under the following form
%\begin{align*}
%  \begin{cases}\textcolor{blue}{\dot{X}(t) = A X(t) + B u(0,t) + B_1 d_1(t), } & t>0, \\ \textcolor{blue}{\mu_t(z,t)=\mu_{zz}(z,t)+\varepsilon \mu(z,t)+\int_0^x \lambda(x, y) \mu(y, t) \mathrm{d} y,} & 0<z<1, t>0, \\ \textcolor{blue}{\mu_z(0, t)=C Y(t),} & t>0, \\ \textcolor{blue}{\mu_z(1, t)=U_c(t)+d_2(t),} & t>0,\end{cases}
%\end{align*}
%    \textcolor{blue}{where $Y(t) \in \mathbb{R}^{n \times 1}$ and $\mu(z,t) \in \mathbb{R}$ are the states of ODE and PIDE, respectively, and $U_c(t) \in \mathbb{R}$ is the control input to the entire system. $A \in \mathbb{R}^{n \times n}, B, B_1 \in \mathbb{R}^{n \times 1}, C \in \mathbb{R}^{1 \times n}$, we assume  $(A, B)$  is stabiliable and $(A,C)$ is observable. Parameter $\varepsilon>0$ is a reaction coefficient, $\lambda(x, y)$ is the spatial weighting kernel function and the integral term denotes the spatial memory. The external disturbances $d_1(t), d_2(t) \in \mathbb{R}$ are bounded: $\left|d_1(t)\right| \leq M_1$ and $\left|d_2(t)\right| \leq M_2$, where $M_1>0$ and $M_2>0$ are constants. Let $L^2(0,1)$ be the Hilbert space with the usual inner product $\langle\cdot\rangle$ and norm $\|\cdot\|$.}

  It worth noting that in the literature, the disturbances considered in the cascaded systems  are    distributed either in the domain or on the {Neumann} boundary of the domain. Thus, the effect of Dirichlet boundary disturbances, which impose a  significant challenge on the  Lyapunov  stability analysis\cite{Mironchenko2020,Karafyllis2018book,Zheng2024}, on the system stability {has} not been investigated yet.  In addition,     in \cite{zhang2020,zhang2025IMA,Zhang2021,Zhang2019}, all disturbances   are assumed to be globally bounded, which is a very strong condition,  and the ISS estimates are given  in the $L^2$-norm while ISS characterizations in other norms are less provided.
   %This is because Dirichlet boundary disturbances, arising from inhomogeneous boundary conditions, present challenges for stability analysis when applying the classical Lyapunov method.
   %the controllers designed in \cite{Zhang2021,zhang2025IMA} by using the sliding mode control  are straightforward to design and offer  a rapid dynamic response; however, they  may suffer form chattering  due to imperfections in switching devices and delays \cite{Khalil2001non}. Therefore, how to design a continuous controller for cascaded ODE-PDE system directly to avoid it becomes a problem worth exploring.
%
  A more important point is that although sliding mode control can be effectively used in controller design for input-to-state stabilizing cascaded ODE-PDE systems as presented in \cite{Zhang2021,zhang2025IMA}, its inherent high-frequency switching characteristic may lead to chattering in the system  or weaken the {robustness} of the system within the predefined boundary layer \cite{Khalil2001non},  adversely affecting stability and equipment lifespan in practical applications. Therefore, it  remains a highly challenging task  to design continuous boundary feedback controllers for ODE-PDE systems cascaded via  boundary terms to ensure that the closed loops   are ISS under various disturbances while  avoiding chattering phenomena  and eliminating the steady-state error.
 %Therefore, it remains a significant challenge  how to avoid the chattering  phenomenon,  designing a continuous controller for input-to-state stabilizing     ODE-PDE systems cascaded via boundary terms  remains a significant challenge in the presence of various external disturbances.

In this paper, we investigate the input-to-state stabilization for an ODE cascaded with a parabolic PDE  under various disturbances,  including  internal disturbances of the ODE and in-domain disturbances,  Robin  boundary disturbances, as well as Dirichlet boundary disturbances of the PDE. In the considered setup, the value of the PDE solution at the left  boundary   is fed into the ODE, while the actuator   is placed at the right Dirichlet boundary. For controller design, we exclusively adopt the backstepping method to devise a continuous boundary feedback control law, thereby avoiding the use of sliding mode control. For stability analysis---particularly to handle the PDE boundary term appearing in the ODE and the Dirichlet boundary disturbances in the PDE when applying the Lyapunov method---we first decouple the cascaded system via backstepping,  and subsequently, employ {the generalized} Lyapunov functionals introduced  in \cite{Zheng2024,Zheng2025} to establish the ISS  in the $\sup$-norm for the target PDE system, thereby ultimately arriving at the  ISS   for the original cascaded ODE-PDE system subject to different disturbances under the framework of Lyapunov stability theory. %Note that when analyzing the ISS for a system using the generalized Lyapunov method, the   critical step is to construct the generalized ISS-Lyapunov functional (GISS-LF) for the system.
 {It is worth mentioning that in the well-posedness analysis,     we take a spatially continuous function  space  as the state space  to ensure that ISS in the  $\sup$-norm is well-defined. This, in comparison with the  most of existing literature  typically  considering the problem in the Hilbert spaces, brings   more complexities when using the analytic semigroup theory of linear operators.}

Overall, the main contributions of this paper are threefold:
\begin{enumerate}
	\item[(i)] For the ODE   cascaded by a parabolic PDE  with different types of disturbances,  a continuous state feedback controller is designed to ensure the ISS of the closed-loop system by using backstepping while avoiding the use of  sliding mode control and the boundedness parameter of disturbances.
	\item[(ii)] The ISS  in the $\sup$-norm  of the    cascaded   system  is assessed under the framework of Lyapunov stability theory. In particular, we show how to use the  generalized Lyapunov functionals to prove the ISS for   parabolic PDEs with general space-time-varying coefficients and Dirichlet-Robin boundary disturbances.
\item[(iii)]  The well-posedness of the target PDE system   is proved by utilizing the technique of lifting and  the analytic semigroup theory due to the appearance of $\sup$-norm of the solution.
\end{enumerate}

In the remainder of the paper, some basic notations employed in this paper are introduced first. In Section~\Rref{sec:Problem setting}, the problem formulation is presented. In Section~\Rref{sec:controller design}, we demonstrate how to design a continuous state feedback controller by using backstepping rather than the  sliding mode control.  In Section~\Rref{well-posedness}, the     semigroup method is utilized to prove the well-posedness of the solution in the  continuous functionals  spaces. In Section~\Rref{sec:input-to-state stability}, we show how to establish the ISS for parabolic  PDEs with  general space-time-varying coefficients and Dirichlet boundary disturbances by employing the generalized Lyapunov functionals. Furthermore, we utilize this method to prove the main result on the ISS of the cascaded  ODE-PDE  in closed loop. Numerical simulations are carried out in Section~\Rref{sec:numerical results} to illustrate the effectiveness of the proposed control scheme. Some conclusions are presented in Section~\Rref{conclusion}.

\textbf{Notation}  {${\mathbb{N}_0}:=\{0,1,2,...\}$, ${\mathbb{N}}:=\mathbb{N}_0\setminus\{0\}$, $\mathbb{R}:=(-\infty,+\infty)$, $\mathbb{R}_{\geq 0}:=[0,+\infty)$, $\mathbb{R}_{> 0}:=(0,+\infty)$, and $\mathbb{R}_{\leq 0}:=(-\infty,0]$.}
 Let $Q_\infty$$:=(0,1)\times {\mathbb{R}_{>0}}$ and $\overline{Q}_\infty:=[0,1]\times {\mathbb{R}_{\geq 0}}$.
For $T\in\mathbb{R}_{>0} $, let $Q_T := (0,1) \times  (0,T)$ and $\overline{Q}_T$$:=$$[0,1]$$\times$$[0,T]$.

For $N\in\mathbb{N}$, $\mathbb{R}^N$  denotes the $N$-dimensional Euclidean space with the standard  norm $|\cdot|$. For a domain $\Omega$ (either open or closed) in $\mathbb{R}^1$ or $\mathbb{R}^2$,
let $C\left({\Omega};\mathbb{R}^N\right):=\{v: {\Omega} \rightarrow \mathbb{R}^N|~$$v$ is continuous on $\Omega$\} with the norm $\|v\|_{C\left({\Omega};\mathbb{R}^N\right)}:=\sup_{z\in\Omega}|v(z)|$. % l{For simplicity, we write $C\left({\Omega} \right):=C\left({\Omega};\mathbb{R}^N\right)$ when $N=1$.}
For $i\in\mathbb{N}$, $C^i\left({\Omega};\mathbb{R}^N\right):=\{v: {\Omega} \rightarrow\mathbb{R}^N\Big|\ v$ has continuous derivatives up to order $i$ on ${\Omega}$\}. Let $ C\left(\mathbb{R}_{\geq 0}; Y\right):=\{v: {\mathbb{R}_{\geq 0}} \rightarrow Y|~$$v$ is continuous on ${\mathbb{R}_{\geq 0}}$\} with the norm $\|f\|_{C\left(\mathbb{R}_{\geq 0} ; Y\right)}:=\sup_{z\in \mathbb{R}_{\geq 0}}\|f(z)\|_Y$, where  $Y$ is a normed linear space.

% with $\|g\|_{C\left([0,T];\mathbb{R}^n\right)}=\max_{t \in[0, T]}|g(t)|$. %Denoted by $C\left(\mathbb{R}_{\geq 0} ; \mathbb{R}_{\geq 0}\right)$ the set of all continuous functions $g: \mathbb{R}_{\geq 0} \rightarrow \mathbb{R}_{\geq 0}$.
%For a normed linear space $Y$, $ C\left(\mathbb{R}_{\geq 0} ; Y\right)$ denotes the set of all continuous functionals $g: \mathbb{R}_{\geq 0} \rightarrow Y$ with the norm $\|f\|_{C\left(\mathbb{R}_{\geq 0} ; Y\right)}:=\sup_{x\in \mathbb{R}_{\geq 0}}\|f(x)\|_Y$.

%with $\|g\|_{ C\left(\mathbb{R}_{\geq 0} ; Y\right)}:=\sup _{s \in \mathbb{R}_{\geq 0}}\|g(s)\|_Y<+\infty$.
%For $f: X \rightarrow \mathbb{R}$, where $X \subset \overline{Q}_{\infty}$, the notation $f[t]$ {(or $f[{z}]$)} denotes the profile at a certain $t \in \mathbb{R}_{\geq 0}$ (or ${z} \in[0,1]$), that is, $f[t]({z})=f({z}, t)$ (or $f[{z}](t)=f({z}, t)$).
%Let the state space  $\mathcal{N}={C\left([0,T];\mathbb{R}^n\right)}\times C(\overline{Q}_T)$ with the norm $\|(f_1,f_2)\|_\mathcal{N}=\max_{t\in[0,T]}|f_2|+\max_{(z,{t})\in \overline Q_T}|f_1|$.
%Áî$L^{\infty}((0,T);L^\infty(0,1))$ $:=$$\{v: (0,1)\times(0,T) \rightarrow\mathbf{R}|~v(\cdot,t)\in L^\infty(0,1), v(\cdot,x)\in L^p(0,T)\}$ ²¢¸³Óè·¶Êý
% \begin{aligned}\|v\|_{L^{p}((0,T);L^1(0,1))} :=
	%\left(\int_0^T\!\!\!\left(\int_0^1|v(z,t)|\text{d}x\right)^p\!\text{d}t\right)^\frac{1}{p}. \end{aligned}$$
$\mathbb{R}^{N\times M}$ is the linear space of $N\times M$ matrices. When $M=N$, $\mathbb{R}^{N\times N}$ is a Banach algebra. The superscript $\top$ represents matrix transposition. For a matrix $Q$, the notation $Q>0$ denotes that $Q$ is symmetric and positive definite. The symbol $\textbf{1}_N$ and $\textbf{0}_N$ represents an $N$-dimensional column vector with all elements being 1 and 0, respectively.
%{The notation $e^{Dz}$ for any matrix $D$ and $z\in[0,1]$ refers to the matrix exponential.}

For normed linear spaces $X$ and $Z$, let $L(X, Z)$ be the space of bounded linear operators $P: X \rightarrow Z$. Let $L(X):=L(X, X)$ with the norm$
\|P\|:=\|P\|_{L(X)}:=\sup \left\{\|P x\|_X \mid x \in X,\|x\|_X \leq 1\right\}$.
For a given linear operator $\mathcal{A}$, $D(\mathcal{A})$ and $\rho(\mathcal{A})$  denote the domain of $\mathcal{A}$ and the resolvent set of $\mathcal{A}$, respectively.

Let $\mathcal {K}$$:=$$\{\gamma:\mathbb{R}_{\geq 0}$$\rightarrow$$\mathbb{R}_{\geq 0}|$$\gamma(0)${$=$}{$0$}, $\gamma$ is continuous, strictly increasing\}, $\mathcal {L}$$:=$$\{\gamma:$$\mathbb{R}_{\geq 0}$$\rightarrow$$\mathbb{R}_{\geq 0}|$ $\gamma$ is continuous, strictly decreasing, {$\lim_{s\rightarrow\infty}\gamma(s)$}{$=$}{$0$}\}, $\mathcal {KL} := \{\beta:$$\mathbb{R}_{\geq 0}$$\times$$\mathbb{R}_{\geq 0}$ $\rightarrow$$\mathbb{R}_{\geq 0}|$${\beta}$ is continuous, $\beta(\cdot,t)$$\in$$\mathcal {K}$, $\forall t \in \mathbb{R}_{\geq 0}$; $\beta(s,\cdot) \in \mathcal {L},\forall  s \in {\mathbb{R}_{> 0}}\}$.

\section{Problem Formulation}\label{sec:Problem setting}
In this paper, we consider the stabilization problem of  an ODE  cascaded by a  $1$-D linear parabolic PDE with different disturbances under the following form:
\begin{subequations}\label{original system}
	\begin{align}
		\dot{{X}}(t)=&A {X}(t)+B u(0, t)+{D}(t), t\in \mathbb{R}_{>0}, \\
		u_t(z, t)=&u_{z z}(z, t)+b(z) u(z, t)+c(z) u(0, t)+f(z, t), (z,t)\in Q_\infty, \label{original equation}\\
		u_z(0, t)=&q u(0, t)+ d_0(t), t\in \mathbb{R}_{>0}, \label{original boundary 0}\\
		u(1, t)=&U(t)+d_1(t), t\in \mathbb{R}_{>0}, \label{original boundary 1} \\
		X(0)=&X_0, u(z, 0)=u_0(z), z\in(0,1), \label{original initial value}
	\end{align}
\end{subequations}
where  $X(t)\in\mathbb{R}^{N}$ and $u(z,t)\in \mathbb{R}$ are the states of the ODE  and  the parabolic PDE, respectively, $A\in\mathbb{R}^{N\times N}$, $B\in\mathbb{R}^{N }$ are constant matrices with the pair $(A, B)$ being controllable,  $b(z),c(z)$ are space-varying coefficients,  $q$ is a constant, ${D(t)\in\mathbb{R}^{N}}$ represents internal disturbances of the ODE plant, $f(z, t)$ represent in-domain disturbances of the PDE plant while $d_0(t)$ and $d_1(t)$ represent Robin and Dirichlet boundary disturbances, respectively,  $X_0\in\mathbb{R}^{N },u_0(z)\in\mathbb{R} $ are the initial {data}, and  $U(t)$ is the control input that needs to be designed. The cascaded system \eqref{original system} is depicted in Fig. \ref{fig:cascade}.

 For system \eqref{original system}, the parabolic PDE  may be unstable in open loop when the reaction coefficient $b(z)$ is positive and sufficiently large. Moreover, the ODE  may also be driven unstable through two complementary mechanisms: either the matrix $A$ has eigenvalues with positive real parts, or the unstable dynamics of the PDE   are imported into the ODE via the coupling interface, e.g., $u(0,t)$.

The aim of the work  is   to design a  continuous state feedback controller by directly using  backstepping rather than the sliding mode control to ensure the robustness of  the cascaded ODE-PDE system  in closed loop and to prove the ISS in the $\sup$-norm   under the framework of Lyapunov stability theory when  different disturbances appear.
%\begin{enumerate}
%	\item[(i)] We design a continuous state feedback controller based on backstepping to stabilize the cascaded ODE-PDE system~\eqref{original system}, and the well-posedness of the solution to the cascaded ODE-PDE  system~\eqref{original system} is proved by using operator semigroup theory in the space of continuous functions.
%	\item[(ii)] To address Dirichlet boundary disturbances and variable coefficients, the generalized Lyapunov method is employed for stability analysis in the max-norm.
%\end{enumerate}

\begin{figure}[ht]
	\centering
	\begin{tikzpicture}[
		node distance = 2.2cm and 1.8cm,
		block/.style  = {draw, rectangle, minimum width=2.4cm, minimum height=1.2cm,
			align=center, rounded corners},
			 noframe/.style={rectangle, minimum width=1cm, minimum height=1.2cm, align=center},
		arr/.style    = {-stealth, thick}
		]
		% ½Úµã
		\node[block,minimum width=3.4cm, minimum height=1.2cm] (act) {parabolic  PDE \\ (actuator)};
		\node[block, right=of act] (plant) {ODE\\(plant)};
		\node[noframe, left=of act] (control) {$U(t)$};
			% u(z,t) ±êÇ©ÔÚ PDE ÉÏ·½
		\node[above=0.05cm of act] {$u(z,t)$};
		\node[above=0.05cm of plant] {$X(t)$};
		% Íⲿ¶Ë×Ó
		%\coordinate[above left=1cm and 1.5cm of act] (inU);
		%\coordinate[right=1.5cm of plant] (outX);
		% ¼ýÍ·
		\draw[arr] (control) -- node[above] {$u(1,t)$} (act.west);
		\draw[arr] (act.east) -- node[above] {$u(0,t)$} (plant.west);
		\coordinate (fixed point) at ([xshift=1.5cm]plant.east);
	\draw[arr] (plant) -- node[above] {} (fixed point);
		%\draw[arr] (plant.east) -- node[above] {$X(t)$} (outX);
		% PDE ÄÚ²¿ÐźÅ
	%	\draw[arr] ([yshift=3mm]act.west) -- ++(-1.2,0) node[left] {$u(z,t)$};
	%	\draw[arr] ([yshift=-3mm]act.west) -- ++(-1.2,0) node[left] {$u_z(D,t)$};
	\end{tikzpicture}
	\caption{Signal flow of the ODE-PDE cascade}
	\label{fig:cascade}
\end{figure}
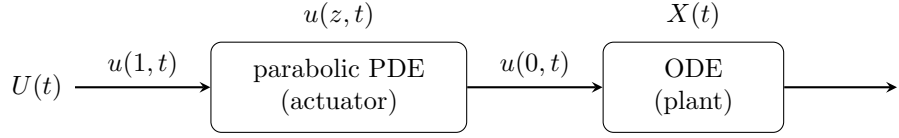

Throughout this paper, we always assume that
\begin{align}
%	f\in& \textcolor{blue}{C^{\alpha,\frac{\alpha}{2}}}(\overline{Q}_{\infty};\mathbb{R}),\quad \quad  D  \in C(\mathbb{R}_{\geq 0};\mathbb{R}^N),\quad d_0,d_1\in {\textcolor{blue}{C^{1,\alpha}}}(\mathbb{R}_{\ge 0};\mathbb{R}),\label{solution-exsitence}\\
	q\in\mathbb{R}_{>0},\quad b,c,u_0\in C((0,1);\mathbb{R}),\quad D\in C(\mathbb{R}_{\ge 0};\mathbb{R}^N),\quad f\in C(\overline{Q}_{\infty};\mathbb{R}), \quad  d_1,d_2\in C^1(\mathbb{R}_{\ge 0};\mathbb{R}),\label{q_c}
\end{align}
and \begin{align}
\sup_{z\in(0,1)} |b(z)|+\sup_{z\in(0,1)}|c(z)| +\sup_{z\in(0,1)}|u_0(z)| <+\infty.\label{equ4}
\end{align}
 Moreover, assume that %$D,f,d_i,d_i'(i=1,2)$  are locally H\"older continuous in $t$ and
   for any interval $I\in \mathbb{R}_{>0}$ there exist  positive constants  $\alpha\in(0,1)$ and $L$ such that
%uniformly with respect to (w.r.t.) $z$, i.e., for any $t\in \mathbb{R}_{>0}$ there exist a neighborhood $I\subset \mathbb{R}_{>0}$ of $t$ and constants $\alpha\in(0,1)$  and $L>0$ such that
\begin{align*}
|D(t_1)-D(t_2)|\le& L\,|t_1-t_2|^{\alpha},  \forall\, t_1,t_2\in I,\\
 |f(z,t_1)-f(z,t_2)| \le& L\,|t_1-t_2|^{\alpha},   \forall\, z\in(0,1),\ \forall\, t_1,t_2\in I,\\
 |d_i(t_1)-d_i(t_2)|+|d_i'(t_1)-d_i'(t_2)|
\le &L\,|t_1-t_2|^{\alpha},  \forall\, t_1,t_2\in I.
\end{align*}

%\textcolor{blue}{In addition,   $D\in C(\mathbb{R}_{\ge 0};\mathbb{R}^N)$ and $d_1,d_2\in C^1(\mathbb{R}_{\ge 0};\mathbb{R})$, and for any $t\in\mathbb{R}_{>0}$ there exist a neighborhood $I\subset \mathbb{R}_{>0}$ of $t$ and constants $\alpha\in(0,1)$ and $L>0$ such that, for each $i\in\{1,2\}$ and all $t_1,t_2\in I$}
%\begin{align*}
%\textcolor{blue}{|D(t_1)-D(t_2)|\le L\,|t_1-t_2|^{\alpha}}
%\end{align*}
%and
%\begin{align*}
%\textcolor{blue}{|d_i(t_1)-d_i(t_2)|+|d_i'(t_1)-d_i'(t_2)|
%\le L\,|t_1-t_2|^{\alpha}.}
%\end{align*}

Next, we present the definition of ISS for system \eqref{original system}.
\begin{dfntn}
	The cascaded ODE-PDE system \eqref{original system} is said to be input-to-state stable (ISS) in the $\sup$-norm with respect to (w.r.t.) disturbances ${D}$, $f$, $d_0$, and $d_1$, if there exist functions $\beta \in \mathcal{KL}$, and $\gamma_i\in\mathcal{K},i=0,1,2,3$ such that (s.t.)  for any  $(X_0,u_0)\in \mathbb{R}^{N}\times C((0,1);\mathbb{R})$, the state $(X,u)$ admits the following estimate:
	\begin{align*}
		|X(T)|+\sup_{z\in(0,1)}\left|u(z,T)\right| \leq& \beta\left(\left|X_0\right|+\sup_{z\in(0,1)}\left|u_0(z)\right|,T\right)+\gamma_0\left(\sup_{t\in(0,T)}\left|{D}(t)\right|\right)+\gamma_1\left(\sup_{(z, t) \in  {Q}_T}\left|f(z,t)\right|\right)\notag\\
		&+\gamma_2\left(\sup_{t\in(0,T)}\left|d_0(t)\right|\right)
		+\gamma_3\left(\sup_{t\in(0,T)}\left|d_1(t)\right|\right), \forall T \in \mathbb{R}_{>0}.
	\end{align*}
%	or,
%	\begin{align*}
%		|X(T)|+\|u[T]\|_{L^\infty(0,1)} \leq& \beta\left(\left|X_0\right|+\left\|u_0\right\|_{L^\infty(0,1)},T\right)+\gamma_0\left(\left\|d_1\right\|_{L^{\infty}(0, T)}\right)+\gamma_1\left(\left\|f\right\|_{\textcolor{blue}{C(\overline{Q}_T})}\right)\notag\\
%		&+\gamma_2\left(\left\|d_0\right\|_{L^{\infty}(0, T)}\right)
%		+\gamma_3\left(\left\|d_1\right\|_{L^{\infty}(0, T)}\right), \forall T \in \mathbb{R}_{\geq 0}.
%	\end{align*}
\end{dfntn}

\section{Controller Design}\label{sec:controller design}
%\subsection{Selection of Target System}
To stabilize  system~\eqref{original system}, let us first introduce the following backstepping transformation:
\begin{subequations}\label{transformation}
	\begin{align}
		%	&\left\{\begin{array}{l}
			X(t)=&X(t), \label{transformation4a}\\
			w(z, t)=&u(z, t)-\int_0^z k(z, y) u(y, t) \text{d} y-\alpha(z)X(t), \label{transformation4b}
			%	\end{array}\right. \\
	\end{align}
\end{subequations}
along with control law
\begin{align}\label{state controller}
	U(t):=\int_0^1 k(1,y) u(y, t) \text{d} y+\alpha(1) X(t),
\end{align}
where $k$ and $\alpha$ are kernel functions to be determined later.

For system~\eqref{original system}, we choose the  following target system:
\begin{subequations}\label{original target system}
	\begin{align}
		\dot{X}(t)=&(A+BK) X(t)+B w(0, t)+{D}(t),  t\in \mathbb{R}_{>0}, \\
		w_t(z, t)=&w_{zz}(z, t)-\lambda_0(z) w(z, t)+\psi(z, t), (z,t)\in Q_\infty,\label{original target equation}\\
		w_z(0, t)=&q w(0, t)+d_0(t),  t\in \mathbb{R}_{>0},  \label{w-system-boundary-0}\\
		w(1, t)=&d_1(t),  t\in \mathbb{R}_{>0}, \label{w-system-boundary-1}\\
		X(0)=&X_0, w(z, 0)=w_0(z), z\in(0,1), \label{original target initial value}
	\end{align}
\end{subequations}
where $K\in\mathbb{R}^{1\times N}$ is a constant matrix s.t. $A+BK$ is   Hurwitz,   $\lambda_0\in C([0,1];\mathbb{R})$ is  a space-varying function satisfying
\begin{align}\label{lambda_0}
  \min_{z\in[0,1]}\lambda_0(z)>0,
\end{align}
and $\psi(z,t)$ is a space-time-varying function given by
\begin{align*}
	\psi(z,t):=f(z, t)-\int_0^z k(z, y) f(y, t) \text{d} y+k(z, 0) d_0(t)-\alpha(z) {D}(t).
\end{align*}
% \subsection{Selection of Kernel Function}

In order to determine the kernel functions  $k$ and $\alpha$, applying the transformation~\eqref{transformation}, we obtain
\begin{align}
	\dot{X}(t)=&A X(t)+B w(0, t)+B\alpha(0) X(t)+{D}(t) \notag\\
	=&\left(A+B\alpha(0)\right) X(t)+B w(0, t)+{D}(t), \label{zz}\\
	w_z(z, t)=&u_z(z, t)-k(z, z) u(z, t)-\int_0^z k_z(z, y) u(y, t) \text{d} y-\alpha^{\prime}(z) X(t), \notag\\
	w_{z z}(z, t)=&u_{z z}(z, t)-\frac{\text{d}}{\text{d} z}(k(z, z)) u(z, t) -k(z, z) u_z(z, t) -k_z(z, z) u(z, t)
	-\int_0^z k_{z z}(z, y) u(y, t) \text{d} y\notag\\
	&-\alpha^{\prime \prime}(z) X(t). \label{wxx}
\end{align}
On the other hand, the derivative of $w(z,t)$ w.r.t. $t$ is given by
\begin{align}\label{wtintial}
	w_t(z, t)=&u_t(z, t)-\int_0^z k(z, y) u_t(y, t) \text{d} y-\alpha(z) \dot{X}(t)  \notag\\
	=&u_{z z}(z, t)+b(z) u(z, t)+c(z) u(0, t)+f(z, t)  -\int_0^z k(z, y)\Big(u_{y y}(y, t)+b(y) u(y, t)+c(y) u(0, t)\notag\\
	&+f(y, t)\Big) \text{d} y -\alpha(z)\Big((A+B \alpha(0)) X(t)+B w(0, t)+{D}(t)\Big).
	%	=&u_{z z}(z, t)+\Lambda(t) u(z, t)+c(t) u(0, t)+f(z, t) \\
	%	& -k(z, z) u_z(z, t)+k(z, 0) u_z(0, t)+u(z, t) k y(z, z) \\
	%	& -u(0, t) k_z(z, 0)-\int_0^z u(y, t) k_{y y}(z, y) d y-\int_0^2 k(z, y) \wedge(y) u(y, t) d y \\
	%	& -\int_0^z k(z, y) c(y) u(0, t) d y-\int_0^z k(z, y) d z(y, t) d y \\
	%	& -\alpha(z)(A+B \alpha(0)) x(t)-\alpha(z) B W(0, t)-\alpha(z) B_1 d_1(t).
\end{align}
Through integration by parts twice, it yields
\begin{align}\label{integration by parts twice}
	\int_0^z k(z, y)u_{y y}(y, t)\text{d}y=&k(z, z) u_z(z, t)-k(z, 0) u_z(0, t) -u(z, t) k_y(z, z) +u(0, t) k_y(z, 0)\notag\\
	&+\int_0^z u(y, t) k_{y y}(z, y) \text{d} y.
\end{align}
Putting~\eqref{integration by parts twice} into~\eqref{wtintial} and by virtue of~\eqref{original boundary 0}, we have
\begin{align}\label{wt}
	w_t(z,t)=&u_{z z}(z, t)+b(z) u(z, t)+c(z) u(0, t)+f(z, t) -k(z, z) u_z(z, t)+k(z, 0) u_z(0, t)+u(z, t) k_y(z, z) \notag\\
	& -u(0, t) k_y(z, 0)-\int_0^z u(y, t) k_{y y}(z, y) \text{d} y
	-\int_0^z k(z, y) b(y) u(y, t) \text{d} y -\int_0^z k(z, y) c(y) u(0, t) \text{d} y\notag\\
	& -\int_0^z k(z, y) f(y, t) \text{d} y-\alpha(z)(A+B \alpha(0))X(t) -\alpha(z) B w(0, t)-\alpha(z) {D}(t).
\end{align}
By~\eqref{transformation}, \eqref{wxx}, and~\eqref{wt}, it yields
\begin{align}\label{target-system-1a}
	& w_t(z, t)-w_{z z}(z, t)+\lambda_0(z) w(z, t) \notag\\
	=&u(z, t)\left(b(z)+k_y(z, z)+\frac{\text{d}}{\text{d} z}(k(z, z))+k_z(z, z)+\lambda_0(z)\right)+\int_0^z u(y, t)\Big(-k_{ y y}(z, y)-k(z, y) b(y)\notag\\
	&+k_{z z}(z, y)-\lambda_0(z) k(z, y)\Big)\text{d} y  +c(z) u(0, t)+f(z, t)+k(z, 0) q u(0, t)+k(z, 0) d_0(t)-u(0, t) k_y(z, 0) \notag\\
	&-\int_0^z k(z, y) c(y) u(0, t) \text{d} y  -\int_0^z k(z, y) f(y, t) \text{d}y-\alpha(z)(A+B \alpha(0)) X(t)-\alpha(z)B u(0, t)\notag\\
	&+\alpha(z) B \alpha(0) X(t)-\alpha(z) {D}(t)  +\alpha^{\prime \prime}(z) X(t)-\lambda_0(z) \alpha(z) X(t) \notag\\
	=&u(z, t)\left(2\frac{\text{d}}{\text{d} z}(k(z, z))+b(z)+\lambda_0(z)\right)
	 +\int_0^z u(y, t)\Big(k_{ z z}(z, y)-k_{y y}(z, y)-(b(y)+\lambda_0(z)) k(z, y)\Big) \text{d} y \notag\\
	& +u(0, t)\left(c(z)+q k(z, 0)-k_y(z, 0)-\int_0^z k(z, y) c(y) \text{d} y-\alpha(z) B\right) +f(z, t)+k(z, 0) d_0(t)\notag\\
	&-\int_0^z k(z, y) f(y, t) \text{d} y -\alpha(z) {D}(t)
	+\Big(-\alpha(z)A+\alpha^{\prime \prime}(z)-\lambda_0(z) \alpha(z)\Big) X(t).
	%	& =f(z, t)+k(z, 0) d_0(t)-\int_0^z k(z, y) d z(y, t) d y-\alpha(z) B, d_1(t) \\
	%	& =\psi(z, t).
\end{align}
%Let $k(z,y)$ and $\alpha(z)$ satisfy the equations~\eqref{theta function} and \eqref{kernel function}. Then, by the definition of $\psi$, \eqref{original equation} is satisfied.
In addition, from \eqref{original boundary 0} and \eqref{transformation4b}, we have
%\begin{subequations}\label{target-system-1b}
\begin{align}\label{target-system-1b}
	w_z(0, t)=&u_z(0, t)-k(0,0) u(0, t)-\alpha^{\prime}(0) X(t) \notag\\
	=&q u(0, t)+d_0(t)-\alpha^{\prime}(0) X(t) \notag\\
	=&q w(0, t)+\left(q \alpha(0)-\alpha^{\prime}(0)\right) X(t)+d_0(t)-k(0,0) u(0, t).
\end{align}
By virtue of \eqref{original boundary 1}, \eqref{transformation4b}, and the state feedback control law \eqref{state controller}, we obtain
\begin{align}\label{target-system-1c}
	%	=&q w(0, t)+d_0(t), \\
	w(1, t)=&u(1, t)-\int_0^1 k(1, y) u(y, t) \text{d} y-\alpha(1) X(t)\notag\\
	=&d_1(t).
\end{align}
%\end{subequations}
%Then the boundary condition~\eqref{original boundary 0} and~\eqref{original boundary 1} are satisfied.

 Comparing \eqref{zz}, \eqref{target-system-1a}, \eqref{target-system-1b}, and \eqref{target-system-1c} with the target system \eqref{original target system}, it is obtained that for $(z,y)\in\overline{\Omega}=\{(z,y)|0\le y\le z\le 1\}$, we choose the kernel function $k(z,y)$ satisfying
\begin{subequations}\label{kernel function}
	\begin{align}
		k_{z z}(z, y)-k_{y y}(z, y)=&(b(y)+\lambda_0(z)) k(z, y), \\
		\frac{\text{d}}{\text{d} z}(k(z, z))=&-\frac{1}{2}(b(z)+\lambda_0(z)), \\
		k_y(z, 0)=&q k(z, 0)-\int_0^z k(z, y) c(y) \text{d} y+c(z)-\alpha(z) B, \\
		k(0,0)=&0
	\end{align}
\end{subequations}
and the matrix-valued function $\alpha(z)$ satisfying
\begin{subequations}\label{theta function}
	\begin{align}
		\alpha^{\prime\prime}(z)=&\alpha(z)(A+\lambda_0(z) I_N),\\
		q\alpha(0)=&\alpha^\prime(0),\\
		\alpha(0)=&K,
	\end{align}
\end{subequations}
where  $\frac{\text{d}}{\text{d}z}\left(k(z,z)\right) :=k_z(z,{y})|_{{y}=z}+k_{{y}}(z,{y})|_{{y}=z}$ and $I_N \in \mathbb{R}^{N \times N}$ is an $N\times N$ identity matrix.

The following Lemma demonstrates the existence of solutions to the equations \eqref{kernel function} and ~\eqref{theta function}.
\begin{prpstn}\label{kernel exitence}
	\begin{enumerate}
		\item[(i)] Equation~\eqref{kernel function} admits a unique solution $k $ in $C^2(\overline{\Omega};\mathbb{R})$.
		%In particular, when $c(z)\equiv c$, it can be express as
		%				\begin{align*}
			%					  k(z, {y})=&-(\Lambda+\lambda) z \frac{\mathcal{I}_1\left(\sqrt{c_0\left(z^2-{y}^2\right)}\right)}{\sqrt{c_0\left(z^2-{y}^2\right)}}\notag \\
			%					&+\frac{q c_0}{\sqrt{c_0+q^2}} \!\!\int_0^{z-{y}}\!\!\! e^{\frac{-q \tau} {2}}\mathcal{I}_0 \left(\!\sqrt{c_0(z+{y})(z-{y}-\tau)}\right)\notag\\
			%					&\quad \quad \quad \qquad \qquad \times \sinh \left(\frac{\sqrt{c_0+q^2}}{2} \tau\right) {\rm d} \tau,
			%		        \end{align*}
		\item[(ii)] Equation~\eqref{theta function} admits  a unique solution, which can be expressed as
		\begin{align}\label{equ19}
			\alpha(z)=\left(\begin{matrix} K & qK \end{matrix}\right)e^{\left(\begin{matrix} 0_{N} & A+\lambda_0(z) I_N\\I_N & 0_{N} \end{matrix}\right)z}\left(\begin{matrix} \textbf{1}_N \\ \textbf{0}_{N} \end{matrix}\right),
		\end{align}
		where $e^{\left(\begin{matrix} 0_N & A+\lambda_0(z) I_N\\I_N & 0_N \end{matrix}\right)z}$ is the matrix exponential  defined in the standard manner (see, e.g., \cite[p.~59, 4.1 Property]{Rugh:1996}) and $0_N  \in \mathbb{R}^{N \times N}$ denotes the $N \times N$ zero matrix, i.e., all its entries are equal to $0$.%and $A_1=\left(\begin{matrix} 0 & A+\lambda_0(z) I\\I & 0 \end{matrix}\right)$.
	\end{enumerate}
\end{prpstn}
\begin{proof}
%\begin{enumerate}
%\item
For (i), equation~\eqref{kernel function} can be solved by using the standard successive approximations. The detailed proof is provided in Appendix~\Rref{solution equation}.
%\item
  For (ii), equation~\eqref{theta function} is a linear ODE of  second order and  can be solved by  using the standard order reduction method (see \cite[pp. 26-28, 2.5 Example]{Rugh:1996}). Thus, the detailed verification is omitted.
%\end{enumerate}
\end{proof}

			The following proposition indicates that    the original system \eqref{original system} is equivalent to the target system \eqref{original target system} in a certain sense,  thereby enabling us to first establish the well-posedness and ISS for the target system and, consequently, for the original system.	
				\begin{prpstn}\label{equivelent}
The transformation  \eqref{transformation}, which maps $(X,u)   $ to $ (X,w) $, has a bounded inverse in $C (\mathbb{R}_{>0};  \mathbb{R}^N)  \times    C(\mathbb{R}_{\geq0};C((0,1);\mathbb{R}))$. %Consequently,  the system \eqref{original system} is equivalent to the system \eqref{original target system}.
					%For any $(z,t)\in\overline Q_T$.
					%Define the linear operator $\Theta, \ell$, and $\vartheta: \mathcal{N} \rightarrow \mathcal{N}$ as follows:
					%\begin{align}\label{Theta-transformation}
					%	\alpha(z, t):=&(\Theta \beta)(z, t)\notag\\
					%	:=&(\ell \beta)(z, t)+\int_0^z(\vartheta \beta)(y, t) \text{d} y,
					%\end{align}
					%where $\alpha=\left(\begin{array}{cc}X \\ w \end{array}\right), \beta=\left(\begin{array}{cc}X \\ u \end{array}\right), \ell=\left(\begin{array}{cc}\left(1-\frac{z}{2}\right) I & 0 \\ -\theta(z) & 1\end{array}\right)$,
					%$\vartheta=\left(\begin{array}{cc}\frac{1}{2} \operatorname{I} & 0 \\ 0 & -k(z,y)\end{array}\right)$.
					
					%Then, $\Theta$ has a bounded linear inverse $\Theta^{-1}: \mathcal{N} \rightarrow \mathcal{N}$, and the system~\eqref{original system} is equivalent with~\eqref{original target system}.
				\end{prpstn}
				\begin{proof}
					The proof is provided in Appendix~\Rref{system-equivelent}.
				\end{proof}
\section{Well-posedness Analysis}\label{well-posedness}
%According to \cite[Theorem 7.7, p.30]{pazy1983}, when assumptions \eqref{solution-exsitence} holds, there exist unique solutions $(X,u) {\in}\mathcal{N}$ and $(X,w) {\in}\mathcal{N}$ for system \eqref{original system} and system \eqref{original target system}, respectively. Now, we first show the ISS of system~\eqref{original target system}.
In this section, we employ the semigroup theory for linear operators to prove  the well-posedness of the target system \eqref{original target system} in the continuous functional spaces.
\begin{thrm}\label{solution exsistence}
	For any $T \in \mathbb{R}_{>0}$ and $(X_0,w_0)\in \mathbb{R}^N\times C((0,1);\mathbb{R})$, the cascaded target system \eqref{original target system} admits a unique  solution $(X,w) {\in} C^1  (\mathbb{R}_{>0};  \mathbb{R}^N)  \times  ( C(\mathbb{R}_{\geq0};C((0,1);\mathbb{R}))\cap C^1(\mathbb{R}_{>0}; C((0,1);\mathbb{R})))$.
\end{thrm}
{To prove this theorem,} we first demonstrate the existence  of a solution to the PDE plant governed by  \eqref{original target equation}-\eqref{w-system-boundary-1}.
We employ  the lifting technique for the proof. Indeed,
letting
\begin{align*}
	g_1(z, t):=& \left(z^2-z\right) d_0(t), \notag\\
	g_2(z, t):=&\left(z^3-2 z^2\right) d_1(t),
\end{align*}
and
\begin{align*}
	\widetilde{w}(z,t):=w(z,t)+g_1(z,t)+g_2(z,t) ,
\end{align*}
we transform system \eqref{original target equation}-\eqref{w-system-boundary-1} into the following system:
%		\begin{align*}
	%	 \widetilde{w_{z z}}=&w_{2 z}+g_{12 z}+g_{2 z z} \\
	%	 \tilde{w}_t=&w_t+g_{1 t}+g_{2 t}=w_{2 z}-\lambda w+4+g_{1 t}+g_{2 t} \\
	%	& =\widetilde{w}_{2 z}-g_{122}-g_{2 z 2}-\lambda \widetilde{\omega}-\lambda g_1-\lambda g_2+4+g_{1 t}+g_{2 t} \\
	%	& \tilde{w}_z(z, t)-q \tilde{w}(z, t)=w_2(z, t)+(2 z-1) d_0(t)+\left(3 z^2-4 z\right) d_1(t) \\
	%	& -9 w(z, t)-9\left(z^2-z\right) d_0(t)-9\left(z^3-2 z^2\right) d_1(t)
	%\end{align*}
	%¼´ $\widetilde{\omega}_2(0, t)-q \tilde{\omega}(0, t)=\omega_2(0, t)-q \omega(0, t)-d_0(t)=0$
	%
	%$$
	%\tilde{w}(1, t)=w(1, t)-d u(t)=0
	%$$
	\begin{subequations}\label{transform-system}
		\begin{align}
			\widetilde{w}_t(z,t)=&\widetilde{w}_{z z}(z,t)-\lambda_0(z) \widetilde{w}(z,t)+\lambda_0(z)\left(g_1(z,t)+g_2(z,t)\right)-\left(g_{1 zz}(z,t)+g_{2 z z}(z,t)\right)\notag\\
			&+g_{1 t}(z,t)+g_{2t}(z,t)+\psi(z,t),  (z,t)\in Q_\infty, \\
			\widetilde{w}_z(0, t)=&q \widetilde{w}(0, t), t\in \mathbb{R}_{>0},\\
			\widetilde{w}(1, t)=&0, t\in \mathbb{R}_{>0},\\
			\widetilde{w}(z, 0)=&	\widetilde{w}_0(z):=w_0(z)+g_1(z,0)+g_2(z,0) , z\in(0,1).
		\end{align}
	\end{subequations}
	For any \begin{align*}
		\widetilde{w}\in D(\mathcal{A}):=\left\{\widetilde w,\widetilde w_z, \widetilde w_{zz}\in C((0,1);\mathbb{R})|\widetilde w_z(0)=q\widetilde{w}(0),\widetilde{w}(1)=0\right\},
	\end{align*}
	let the linear operator $\mathcal{A}$ be given by
	\begin{align*}
		\mathcal{A}\widetilde{w}:=\widetilde{w}_{zz}.
	\end{align*}
It is clear that $0\in {\rho(\mathcal{A})}$.

	For the given functions $g_1$ and $g_2$, let the functional $\widetilde F$ be defined by
	\begin{align*}
		\widetilde F(t,\widetilde w):=&-\lambda_0(z)\widetilde w+\lambda_0(z)(g_1+g_2)-\left(g_{1 zz}+g_{2 z z}\right)+g_{1 t}+g_{2 t}+\psi,\forall t\in \mathbb{R}_{\ge 0}, \widetilde w\in C((0,1);\mathbb{R}).
	\end{align*}
	Then, system \eqref{transform-system} can be written as the following abstract form:
	\begin{subequations}\label{abstract-system}
		\begin{align}
			\dot{\widetilde{w}}=&\mathcal{A}\widetilde{w}+\widetilde F(t,\widetilde w),\\
			\widetilde w(0)=&\widetilde w_0.
		\end{align}
	\end{subequations}
%	It is well known that the operator $\mathcal{A}$ is the infinitesimal generator of a $C_0$-semigroup $S(t)$ on $C([0,1])$ for $t\in \mathbb{R}_{\ge 0}$.

	For system \eqref{abstract-system}, we have the following well-posedness result, which is proved based on the analytic semigroup theory.
	\begin{prpstn}\label{w-solution exsistence}
		System  \eqref{abstract-system} admits a solution {$\widetilde{w}\in C(\mathbb{R}_{\geq0};C((0,1);\mathbb{R}))\cap C^1(\mathbb{R}_{>0}; C((0,1);\mathbb{R}))$.}
	\end{prpstn}
%	Before proceeding to prove the aforementioned proposition, we establish the following Lemma that guarantees the compactness and analyticity of $S(t)$.
%	\begin{lmm}\label{lemma1}
%		Let the operator $A$ be defined as in \eqref{operatorA}. Subsequently, the operator $A$ is  the infinitesimal generator of a compact analytic semigroup $S(t)$ on $C([0,1])$ for $t\in \mathbb{R}_{\ge 0}$.
%	\end{lmm}
	\begin{proof}
		We prove this proposition in a way similar to  the proof of \cite[Lemma 2.1, pp. 234-235]{pazy1983}. {Indeed, we consider the  {argument} in the complex domain.} Let $h \in C((0,1);\mathbb{R})$ and $ \lambda=r e^{i \theta}$ with $ r \in \mathbb{R}_{>0}$  and $\theta \in\left(-\frac{\pi}{2}, \frac{\pi}{2}\right)$. Consider the boundary value problem:
		\begin{subequations}\label{boundary-value}
			\begin{align}
				\lambda^2 u(z)-u^{\prime \prime}(z)=&h(z),z\in(0,1), \\
				u(1)=&0, \\
				u^{\prime}(0)-q u(0)=&0.
			\end{align}
		\end{subequations}
		Let  $G(z, y)$  be the Green's function that satisfies the following conditions over $(0,1)$:
		\begin{align*}
			\lambda^2 G(z, y)-G_{zz}(z, y)=&\delta(z-y),   \\
			G_z(0, y)-q G(0, y)=&0, \\
			G(1, y)=&0,
		\end{align*}
		where $\delta(\cdot)$ is the standard Dirac delta function.
		%ÏÂÃæÇó½â (3)£¬µ± $x=y$ ʱ·½³Ì±äΪÆë´Î΢·Ö·½³Ì£¬Æä¶ÔÓ¦µÄ½â¿ÉдΪ
		%\begin{align*}
		% G(z, y)=\left\{\begin{array}{l}
			%		A(y) e^{\lambda z}+B(y) e^{-\lambda z}, z<y \\
			%		C(y) e^{\lambda z}+D(y) e^{-\lambda z}, z>y
			%	\end{array}\right.
		%\end{align*}
		%\begin{align*}
		%	& G_z(z, y)=\left\{\begin{array}{l}
			%	\lambda A(y) e^{\lambda z}-\lambda B(y) e^{-\lambda z}, z<y \\
			%	\lambda C(y) e^{\lambda z}-\lambda D(y) e^{-\lambda z},z>y
			%\end{array}\right.
			%\end{align*}
			
			By direct computations, we obtain
			\begin{align*}
				G(z, y)=\left\{\begin{array}{l}
					\frac{(\lambda+q)\left(e^{\lambda-\lambda y+\lambda z}-e^{\lambda y-\lambda+\lambda z}\right)+(\lambda-q)\left(e^{\lambda-\lambda y-\lambda z}-e^{\lambda y-\lambda-\lambda z}\right)}{2\lambda(\lambda+q)e^{\lambda}+2\lambda(\lambda-q)e^{-\lambda}}, z<y, \\
					\frac{(\lambda+q)\left(e^{\lambda y+\lambda-\lambda z}-e^{\lambda y-\lambda+\lambda z}\right)+(\lambda-q)\left(e^{-\lambda y+\lambda-\lambda z}-e^{-\lambda y-\lambda+\lambda z}\right)}{2\lambda(\lambda+q)e^{\lambda}+2\lambda(\lambda-q)e^{-\lambda}}, z>y.
				\end{array}\right.
			\end{align*}
			Then, the solution to \eqref{boundary-value} is given by
			\begin{align*}
				u(z)=&\int_0^1 G(z, y) h(y) \text{d} y \\
				=&\int_0^z \frac{(\lambda+q)\left(e^{\lambda y+\lambda-\lambda z}-e^{\lambda y-\lambda+\lambda z}\right)+(\lambda-q)\left(e^{-\lambda y+\lambda-\lambda z}-e^{-\lambda y-\lambda+\lambda z}\right)}{2\lambda(\lambda+q)e^{\lambda}+2\lambda(\lambda-q)e^{-\lambda}}h(y) \text{d} y \\
				& +\int_z^1 \frac{(\lambda+q)\left(e^{\lambda-\lambda y+\lambda z}-e^{\lambda y-\lambda+\lambda z}\right)+(\lambda-q)\left(e^{\lambda-\lambda y-\lambda z}-e^{\lambda y-\lambda-\lambda z}\right)}{2\lambda(\lambda+q)e^{\lambda}+2\lambda(\lambda-q)e^{-\lambda}} h(y)\text{d} y \\
				=&\int_0^1 \frac{(\lambda-q) e^{\lambda-\lambda y-\lambda z}-(\lambda+q) e^{\lambda y-\lambda+\lambda z}}{2 \lambda(\lambda+q) e^\lambda+2 \lambda(\lambda-q) e^{-\lambda}} h(y) \text{d} y
				 +\int_0^z \frac{(\lambda+q) e^{\lambda y+\lambda-\lambda z}-(\lambda-q) e^{-\lambda y-\lambda+\lambda z}}{2 \lambda(\lambda+q) e^\lambda+2 \lambda(\lambda-q) e^{-\lambda}} h(y) \text{d} y\\
				&+\int_z^1 \frac{(\lambda+q) e^{\lambda-\lambda y+\lambda z}-(\lambda-q) e^{\lambda y- \lambda-\lambda z}}{2 \lambda(\lambda+q) e^\lambda+2 \lambda(\lambda-q) e^{-\lambda}} h(y) \text{d} y,\forall z\in(0,1).
			\end{align*}
			It follows that
			\begin{align}\label{|u(z)|}
				|u(z)|\leq &\int_0^1 \left|\frac{(\lambda-q) e^{\lambda-\lambda y-\lambda z}-(\lambda+q) e^{\lambda y-\lambda+\lambda z}}{2 \lambda(\lambda+q) e^\lambda+2 \lambda(\lambda-q) e^{-\lambda}}\right| |h(y)| \text{d} y
				 +\int_0^z  \left|\frac{(\lambda+q) e^{\lambda y+\lambda-\lambda z}-(\lambda-q) e^{-\lambda y-\lambda+\lambda z}}{2 \lambda(\lambda+q) e^\lambda+2 \lambda(\lambda-q) e^{-\lambda}}\right| |h(y)| \text{d} y\notag\\
				&+\int_z^1  \left|\frac{(\lambda+q) e^{\lambda-\lambda y+\lambda z}-(\lambda-q) e^{\lambda y- \lambda-\lambda z}}{2 \lambda(\lambda+q) e^\lambda+2 \lambda(\lambda-q) e^{-\lambda}}\right| |h(y)| \text{d} y,\forall z\in(0,1).
			\end{align}
			We  estimate each term on the right-hand side of \eqref{|u(z)|}.
			First, note that  for a complex number $z:=a+b \mathrm{i}$ with $a, b \in \mathbb{R}$ and $ \mathrm{i}^2 =-1$ we have
			\begin{align*}
				\left|e^z+e^{-z}\right| & =\left|e^a(\cos (b)+\mathrm{i} \sin (b))+e^{-a}(\cos (-b)+\mathrm{i} \sin (-b))\right| \\
				& =\sqrt{\left(e^a \cos b+e^{-a} \cos (-b)\right)^2+\left(e^a \sin (b)+e^{-a} \sin (-b)\right)^2} \\
				& \leq \sqrt{\left(e^a+e^{-a}\right)^2+\left(e^a-e^{-a}\right)^2} \\
				& =\sqrt{2e^{2 a}+2e^{-2 a}}
			\end{align*}
			and
			\begin{align*}
				\left|e^z+e^{-z}\right| & =\left|e^a(\cos (b)+\mathrm{i} \sin (b))+e^{-a}(\cos (-b)+\mathrm{i} \sin (-b))\right| \\
				& =\sqrt{\left(e^a \cos (b)+e^{-a} \cos (-b)\right)^2+\left(e^a \sin (b)+e^{-a} \sin (-b)\right)^2} \\
				& \geq \sqrt{\left(e^a \cos (b)+e^{-a} \cos (-b)\right)^2} \\
				& =\left(e^a+e^{-a}\right) |\cos (b)|.
			\end{align*}
			%Next, we analyze \eqref{|u(z)|} term by term.
			Denote   $\lambda:=\mu+\nu \mathrm{i}$ with $\mu:=r \cos (\theta)>0$. For the first term on the right-hand side of  \eqref{|u(z)|}, we deduce that
			\begin{align}\label{M1}
				& \int_0^1\left|\frac{(\lambda-q) e^{\lambda-\lambda y-\lambda z}-(\lambda+q) e^{\lambda y-\lambda+\lambda z}}{2 \lambda(\lambda+q) e^\lambda+2 \lambda(\lambda-q) e^{-\lambda}}\right||h(y)| \text{d}y \notag\\
				\leq & \frac{\left(\sup_{z\in(0,1)}|h(z)|\right)|\lambda+q|}{2|\lambda|\left|e^\lambda+e^{-\lambda}\right||\lambda-q|} \int_0^1\left| e^{\lambda-\lambda y-\lambda z}+e^{\lambda y-\lambda+\lambda z}\right| \text{d} y\notag\\
				\leq&\frac{\left(\sup_{z\in(0,1)}|h(z)|\right)|\lambda+q|}{2|\lambda|^2 | e^\lambda+e^{-\lambda }||\lambda-q|}\left| e^{-\lambda z}+e^{\lambda-\lambda z}+ e^{\lambda z}+e^{\lambda z-\lambda}\right| \notag\\
				\leq& \frac{\left(\sup_{z\in(0,1)}|h(z)|\right)|\lambda+q|}{2|\lambda|^2 \cos (\theta)\left(e^\mu+e^{-\mu}\right)|\lambda-q|}\left(\sqrt{2 e^{2 \mu z}+2 e^{-2 \mu z}}+\sqrt{2e^{2 \mu(1-z)}+2 e^{-2 \mu(1-z)}}\right) , \forall z\in[0,1].
				% \leq& \frac{\max _{z\in [0,1]}|h(z)|}{2|\lambda| \cos (\theta)\left(e^\mu+e^{-\mu}\right)}\left((1+q)\left(\sqrt{2 e^{2 \mu z}+2 e^{-2 \mu z}}+\sqrt{2e^{2 \mu(1-z)}+2 e^{-2 \mu(1-z)}}\right)\right), \forall z\in[0,1].
			\end{align}
			Note that for all $z\in[0,1]$, we always have
			\begin{align}\label{eq-24}
				\lim _{\mu \rightarrow 0} \frac{\left(\sqrt{2 e^{2 \mu z}+2 e^{-2 \mu z}}+\sqrt{2 e^{2 \mu(1-z)}+2 e^{-2 \mu(1-z)}}\right)\sqrt{(\mu+q)^2+\nu^2}}{(e^\mu+e^{-\mu})\sqrt{(\mu-q)^2+\nu^2}}=&2,
						\end{align}
			Meanwhile, for all $z\in(0,1)$, we have
						\begin{align}\label{eq-25}		
				\lim _{\mu \rightarrow +\infty} \frac{\left(\sqrt{2 e^{2 \mu z}+2 e^{-2 \mu z}}+\sqrt{2 e^{2 \mu(1-z)}+2 e^{-2 \mu(1-z)}}\right)\sqrt{(\mu+q)^2+\nu^2}}{(e^\mu+e^{-\mu})\sqrt{(\mu-q)^2+\nu^2}}=&0,
			\end{align}
			and for $z=0$ and $z=1$, we have
			  			\begin{align}\label{eq-26}			
			  	\lim _{\mu \rightarrow +\infty} \frac{\left(2+\sqrt{2 e^{2 \mu}+2 e^{-2 \mu}}\right)\sqrt{(\mu+q)^2+\nu^2}}{(e^\mu+e^{-\mu})\sqrt{(\mu-q)^2+\nu^2}}=&\sqrt{2}.
			  \end{align}
			By~\eqref{eq-24}, \eqref{eq-25}, and \eqref{eq-26}, we  infer that there exists $M_1 \in \mathbb{R}_{>0}$ s.t.
			\begin{align*}
				\frac{\left(\sqrt{2 e^{2 \mu z}+2 e^{-2 \mu z}}+\sqrt{2 e^{2 \mu(1-z)}+2 e^{-2 \mu(1-z)}}\right)\sqrt{(\mu+q)^2+\nu^2}}{(e^\mu+e^{-\mu})\sqrt{(\mu-q)^2+\nu^2}} \leq M_1, \forall \mu,\nu \in \mathbb{R}_{>0}, z \in[0,1].
			\end{align*}
			which, along with \eqref{M1}, ensures that
			\begin{align}\label{M1-}
				 \int_0^1\left|\frac{(\lambda-q) e^{\lambda-\lambda y-\lambda z}-(\lambda+q) e^{\lambda y-\lambda+\lambda z}}{2 \lambda(\lambda+q) e^\lambda+2 \lambda(\lambda-q) e^{-\lambda}}\right||h(y)| \text{d}y \leq \frac{M_1}{|\lambda|^2\cos(\theta)}\sup_{z\in(0,1)}|h(z)|, \forall z\in[0,1].
			\end{align}
For the second term on the right-hand side of \eqref{|u(z)|}, we deduce that
			\begin{align}\label{M2}
				& \int_0^z\left|\frac{(\lambda+q) e^{\lambda y+\lambda-\lambda z}-(\lambda-q) e^{-\lambda y-\lambda+\lambda z}}{2 \lambda(\lambda+q) e^\lambda+2 \lambda(\lambda-q) e^{-\lambda}}\right||h(y)| \text{d} y \notag\\
				\leq& \frac{\sup_{z\in(0,1)}|h(z)||\lambda+q|}{2\left|\lambda \| e^\lambda+e^{-\lambda}\right||\lambda-q|} \int_0^z\left|e^{\lambda y+\lambda-\lambda z}+ e^{-\lambda y-\lambda+\lambda z}\right| \text{d} y \notag\\
				%=&\frac{\max _{z\in [0,1]}|h(z)||\lambda+q|}{2|\lambda|\left|e^\lambda+e^{-\lambda}|\lambda-q|\right|}\left|\frac{\lambda+q}{\lambda} e^\lambda-\frac{\lambda+q}{\lambda} e^{\lambda-\lambda z}-\frac{\lambda+q}{\lambda} e^{-\lambda}+\frac{\lambda+q}{\lambda} e^{-\lambda+\lambda z}\right| \notag\\
				\leq& \frac{\sup_{z\in(0,1)}|h(z)||\lambda+q|}{2|\lambda|^2 \left| e^\lambda+e^{-\lambda}\right||\lambda-q|}\left|e^\lambda+e^{-\lambda}+e^{\lambda-\lambda z}+e^{-\lambda+\lambda z}\right|\notag\\
				\leq& \frac{\sup_{z\in(0,1)}|h(z)|\sqrt{(\mu+q)^2+\nu^2}}{2|\lambda|^2 \cos (\theta)\left(e^\mu+e^{-\mu}\right)\sqrt{(\mu-q)^2+\nu^2}}\left(\sqrt{2 e^{2 \mu}+2 e^{-2 \mu}}+\sqrt{2 e^{2 \mu(1-z)}+2 e^{-2 \mu(1-z)}}\right), \forall z\in[0,1].
			\end{align}
			Note that for all $z\in[0,1]$, we have
			\begin{align}\label{e-28}
				\lim _{\mu \rightarrow 0} \frac{\left(\sqrt{2 e^{2 \mu}+2 e^{-2 \mu}}+\sqrt{2 e^{2 \mu(1-z)}+2 e^{-2 \mu(1-z)}}\right)\sqrt{(\mu+q)^2+\nu^2}}{(e^\mu+e^{-\mu})\sqrt{(\mu-q)^2+\nu^2}}=&2.
						\end{align}
						Meanwhile,	for all $z\in(0,1]$, we have
							\begin{align}\label{e-29}
				\lim _{\mu \rightarrow +\infty} \frac{\left(\sqrt{2 e^{2 \mu}+2 e^{-2 \mu}}+\sqrt{2 e^{2 \mu(1-z)}+2 e^{-2 \mu( 1-z)}}\right)\sqrt{(\mu+q)^2+\nu^2}}{(e^\mu+e^{-\mu})\sqrt{(\mu-q)^2+\nu^2}}=&\sqrt{2},
			\end{align}
				and	for all $z=0$, we have
			\begin{align}\label{e-30}
				\lim _{\mu \rightarrow +\infty} \frac{2\sqrt{2 e^{2 \mu}+2 e^{-2 \mu}}\sqrt{(\mu+q)^2+\nu^2}}{(e^\mu+e^{-\mu})\sqrt{(\mu-q)^2+\nu^2}}=&2\sqrt{2}.
			\end{align}
			By \eqref{e-28}, \eqref{e-29}, and \eqref{e-30}, we   infer that there exists $M_2 \in \mathbb{R}_{>0}$ s.t.
			\begin{align*}
				\frac{\left(\sqrt{2 e^{2 \mu}+2 e^{-2 \mu}}+\sqrt{2 e^{2 \mu(1-z)}+2 e^{-2 \mu(1-z)}}\right)\sqrt{(\mu+q)^2+\nu^2}}{(e^\mu+e^{-\mu})\sqrt{(\mu-q)^2+\nu^2}}\leq M_2, \forall \mu, \nu\in \mathbb{R}_{>0}, z \in[0,1],
			\end{align*}
			which, along with \eqref{M2}, implies that
			\begin{align}\label{M2-}
				 \int_0^z\left|\frac{(\lambda+q) e^{\lambda y+\lambda-\lambda z}-(\lambda-q) e^{-\lambda y-\lambda+\lambda z}}{2 \lambda(\lambda+q) e^\lambda+2 \lambda(\lambda-q) e^{-\lambda}}\right||h(y)| \text{d} y\leq \frac{M_2}{|\lambda|^2\cos(\theta)}\sup_{z\in(0,1)}|h(z)|, \forall z\in[0,1].
			\end{align}
Analogously, for the third term on the right-hand side of \eqref{|u(z)|}, we deduce that
			\begin{align}\label{M3}
				 &\int_z^1\left|\frac{(\lambda+q) e^{\lambda-\lambda y+\lambda z}-(\lambda-q) e^{\lambda y-\lambda-\lambda z}}{2 \lambda(\lambda+q) e^\lambda+2 \lambda(\lambda-q) e^{-\lambda}}\right||h(y)| \text{d} y \notag\\
				\leq& \frac{\sup_{z\in(0,1)}|h(z)||\lambda+q|}{2 \lambda| | e^\lambda+e^{-\lambda}||\lambda-q|} \int_z^1\left| e^{\lambda-\lambda y+\lambda z}+ e^{\lambda y-\lambda-\lambda z}\right| \text{d} y  \notag\\
				\leq&\frac{\sup_{z\in(0,1)}|h(z)||\lambda+q|}{2|\lambda|^2 | e^\lambda+e^{-\lambda}||\lambda-q|} \left|e^{\lambda z}+e^\lambda+ e^{-\lambda z}+ e^{-\lambda}\right|  \notag\\
				% \leq& \frac{\max _{z\in [0,1]}|h(z)|}{2|\lambda|\left|e^\lambda+e^{-\lambda}\right|}\left((1+q)\left|e^{\lambda z}+e^\lambda+e^{-\lambda z}+e^{-\lambda}\right|\right)  \notag\\
				\leq& \frac{\sup_{z\in(0,1)}|h(z)|\sqrt{(\mu+q)^2+\nu^2}}{2|\lambda|^2 \cos (\theta)\left(e^\mu+e^{-\mu}\right)\sqrt{(\mu-q)^2+\nu^2}}\left(\sqrt{2 e^{2 \mu z}+2 e^{-2 \mu z}}+\sqrt{2 e^{2 \mu}+2 e^{-2 \mu}}\right), \forall z\in[0,1].
			\end{align}
			Note that for all $z\in[0,1]$, we have
			\begin{align}\label{eq-32}
				\lim _{\mu \rightarrow 0} \frac{\left(\sqrt{2 e^{2 \mu z}+2 e^{-2 \mu z}}+\sqrt{2 e^{2 \mu}+2 e^{-2 \mu}}\right)\sqrt{(\mu+q)^2+\nu^2}}{(e^\mu+e^{-\mu})\sqrt{(\mu-q)^2+\nu^2}}=&2.
						\end{align}
						At the same time, for all $z\in[0,1)$, we always have
							\begin{align}	\label{eq-33}
				\lim _{\mu \rightarrow +\infty} \frac{\left(\sqrt{2 e^{2 \mu z}+2 e^{-2 \mu z}}+\sqrt{2 e^{2 \mu}+2 e^{-2 \mu}}\right)\sqrt{(\mu+q)^2+\nu^2}}{(e^\mu+e^{-\mu})\sqrt{(\mu-q)^2+\nu^2}}=&\sqrt{2},
			\end{align}
				and for all $z=1$, we   have
			\begin{align}	\label{eq-34}
				\lim _{\mu \rightarrow +\infty} \frac{2\sqrt{2 e^{2 \mu z}+2 e^{-2 \mu z}}\sqrt{(\mu+q)^2+\nu^2}}{(e^\mu+e^{-\mu})\sqrt{(\mu-q)^2+\nu^2}}=&2\sqrt{2}.
			\end{align}
			By \eqref{eq-32}, \eqref{eq-33}, and \eqref{eq-34}, we   infer that there exists $M_3 \in \mathbb{R}_{>0}$ s.t.
			\begin{align*}
				\frac{\left(\sqrt{2 e^{2 \mu z}+2 e^{-2 \mu z}}+\sqrt{2 e^{2 \mu}+2 e^{-2 \mu}}\right)\sqrt{(\mu+q)^2+\nu^2}}{(e^\mu+e^{-\mu})\sqrt{(\mu-q)^2+\nu^2}} \leq M_3 , \forall \mu,\nu \in \mathbb{R}_{>0}, z \in[0,1],
			\end{align*}
				which, along with \eqref{M3}, implies that
				\begin{align}\label{M3-}
					\int_z^1\left|\frac{(\lambda+q) e^{\lambda-\lambda y+\lambda z}-(\lambda-q) e^{\lambda y-\lambda-\lambda z}}{2 \lambda(\lambda+q) e^\lambda+2 \lambda(\lambda-q) e^{-\lambda}}\right||h(y)| \text{d} y\leq \frac{M_3}{|\lambda|^2\cos(\theta)}\sup_{z\in(0,1)}|h(z)|, \forall z\in[0,1].
					\end{align}
			
Substituting \eqref{M1-}, \eqref{M2-}, and \eqref{M3-} into \eqref{|u(z)|}, we obtain
			\begin{align*}
				|u(z)| \leq \frac{M_1+M_2+M_3}{2 |\lambda|^2 \cos (\theta)}\sup_{z\in(0,1)}|h(z)|, \forall z\in[0,1].%\label{Equ.42}
			\end{align*}
In particular, it holds that
\begin{align}
				\sup_{z\in(0,1)}|u(z)| \leq \frac{M}{|\lambda|}\sup_{z\in(0,1)}|h(z)|,\label{Equ.42}
			\end{align}
where $M=\frac{M_1+M_2+M_3}{2 \cos (\theta_0)}$.

  For any fixed $\theta_0 \in\left(\frac{\pi}{4},\frac{\pi}{2}\right)$, we define
			\begin{align*}
				\Sigma\left(\theta_0\right):=\left\{\lambda:|\arg \lambda|<2 \theta_0\right\} \subset {\rho(\mathcal{A})}.
			\end{align*}
		It follows from \eqref{Equ.42} that
			\begin{align*}
				&\|R(\lambda: \mathcal{A})\| \leq \frac{M}{|\lambda|},\forall \lambda \in \Sigma\left(\theta_0\right).
			\end{align*}

			Note that $D(\mathcal{A})$ is dense in $C([0,1];\mathbb{R})$. From \cite[Theorem 7.7, p. 30]{pazy1983}, we infer that $\mathcal{A}$ is the infinitesimal generator of a $C_0$-semigroup $S(t)$  satisfying
			\begin{align*}
	\|S(t)\|\leq c_0,\forall t\in \mathbb{R}_{\ge 0}
			\end{align*}
			with some positive constant $c_0$.
			Furthermore, we   deduce from \cite[Theorem 5.2, p. 61]{pazy1983} that the semigroup $S(t)(t\in\mathbb{R}_{\ge 0})$ is analytic.
			
		%	Finally, accdording to \cite[Lemma 2.1, pp. 234-235]{pazy1983}, it can be shown that the semigroup $S(t)(t\in\mathbb{R}_{\ge 0})$ is compact.
			Finally, noting that  the results of \cite[Theorem 3.1 and 3.3, pp. 196-199]{pazy1983} remains true for $X_{\alpha}:=X$, we deduce that \eqref{abstract-system} admits a unique  global  solution {$\widetilde{w}\in C(\mathbb{R}_{\geq0};C((0,1);\mathbb{R}))\cap C^1(\mathbb{R}_{>0}; C((0,1);\mathbb{R}))$} for any $T>0$.
		Thus, the proof is completed.	
		\end{proof}

%		\begin{pf4}
%			For any given $g_1$ and $g_2$, it is evident that
%			\begin{align}\label{F}
%				\max_{z\in[0,1]}\left|\widetilde F(t,\widetilde w)\right|\le \lambda_0 	\max_{z\in[0,1]}|\widetilde w|+\max_{z\in[0,1]} |k_2|,\forall t\in \mathbb{R}_{\ge 0},\widetilde w\in C([0,1]),
%			\end{align}
%			where $k_2:=\lambda_0(g_1+g_2)-\left(g_{1 zz}+g_{2 z z}\right)+g_{1 t}+g_{2 t}+\psi$.
%			
%			By \eqref{F} and Lemma \Rref{lemma1}, we deduce form \cite[Corollary 2.3, p. 194]{pazy1983} that system  \eqref{original target equation}-\eqref{w-system-boundary-1}  admits a solution $\widetilde{w}\in C(\mathbb{R}_{\ge 0};C([0,1]))$. Thus, the proof is complete.
%		\end{pf4}
		 \begin{pf5}
             According to the definition of $\widetilde{w}$ and Proposition \ref{w-solution exsistence},   we deduce that the PDE plant of system \eqref{original target system} admits a unique solution $w \in C(\mathbb{R}_{\geq0};C((0,1);\mathbb{R}))\cap C^1(\mathbb{R}_{>0}; C((0,1);\mathbb{R}))$. Moreover,  \cite[Theorem 2.4.1  p.69]{boyce2021} ensures  that the ODE plant of  system~\eqref{original target system} admits a unique solution $X\in C^1(\mathbb{R}_{>0};\mathbb{R}^N)$.
        \end{pf5}
     Note that a consequence of  Proposition \ref{equivelent} and Theorem \ref{solution exsistence} is that the original cascaded ODE-PDE system~\eqref{original system} with  the state feedback control law \eqref{state controller} admits a unique  solution $(X,u) {\in} C^1  (\mathbb{R}_{>0};  \mathbb{R}^N)  \times  ( C(\mathbb{R}_{\geq0};C((0,1);\mathbb{R}))\cap C^1(\mathbb{R}_{>0}; C((0,1);\mathbb{R})))$.

\section{Stability Assessment}\label{sec:input-to-state stability}
In this section, to overcome   the obstacle brought by the Dirichlet boundary disturbances to stability assessment, we employ the generalized Lyapunov method    to establish the  ISS of the closed-loop system~\eqref{original system}. This result is formalized in the following theorem.
\begin{thrm}\label{original main result}
	The cascaded ODE-PDE system~\eqref{original system} with  the state feedback control law \eqref{state controller}
	%admits a unique  solution $(X,u) {\in}{C\left([0,T];\mathbb{R}^n\right)}\times C(\overline{Q}_T)$ and
	is {\rm ISS} in the $\sup$-norm, having the estimate %namely, for any initial state $(X_0,u_0)\in  \mathbb{R}^N \times C((0,1);\mathbb{R})$, it holds that
	\begin{align*}
		| X(T)|+\sup_{z\in(0,1)}\left|u(z,T)\right| \leq& C_0 \left(e^{-\zeta T}\left(|X_0|+\sup_{z\in (0,1)}\left|u_0(z)\right|\right)+\sup_{t\in(0,T)}\left|{D}(t)\right|\right. \\
		& \left.+\sup_{(z, t) \in {Q}_T}\left|f(z,t)\right|+\sup_{t\in(0,T)}\left|d_0(t)\right|+\sup_{t\in(0,T)}\left|d_1(t)\right|\right),\forall T\in\mathbb{R}_{>0},
	\end{align*}
	%	or
	%	\begin{align*}
		%		| X(T)| +\|u[T]\|_{L^\infty(0,T)}
		%		\leq & C \left(e^{-\zeta T}\left(\left|X_0\right|+\left\|u_0\right\|_{L^\infty(0,1)}\right)+\left\|d_1\right\|_{L^\infty(0,T)}\right. \\
		%		& \left.+\left\|f\right\|_{L^\infty((0,1);L^\infty(0,T))}+\left\|d_0\right\|_{L^\infty(0,T)}+\left\|d_1\right\|_{L^\infty(0,T)}\right),
		%	\end{align*}
	where $\zeta$  and $C_0$ are  positive constants, which are independent of disturbances $D, f, d_0, d_1$ and initial datum $(X_0,u_0)$.
\end{thrm}

%ÎÒÃÇÖ÷ÒªÀûÓùãÒåLyapunov·½·¨Ö¤Ã÷ÉÏÊöÖ÷Òª½áÂÛ, ΪÁ˸üºÃµÄչʾ¹ãÒåLyapunov·½·¨, ÎÒÃǸø³ö¸ü¼Ó¾ßÓÐÒ»°ãÐԵĽáÂÛ¡£¶ÔÓÚÒ»ÀຬÓÐʱ¿Õ±äϵÊýµÄÅ×Îïϵͳ:

{To  illustrate the universal applicability  of  the generalized Lyapunov method, we conduct  the stability analysis for the case of PDE system \eqref{original target equation}-\eqref{original target initial value} with only spatially dependent coefficients} in a generic case. More specially, we establish the ISS  for a parabolic  equation with Dirichlet boundary disturbances and space-time-varying coefficients:
	\begin{subequations}\label{v-system}
	\begin{align}
		v_t(z, t)=&v_{zz}(z, t)-b_0(z,t) v(z, t)+ {h(z, t)}, (z,t)\in Q_\infty,\label{v equation}\\
		v_z(0, t)=&{q_0} v(0, t)+ {{h}_{0}(t)},  t\in \mathbb{R}_{>0},  \label{v-system-boundary-0}\\
		v(1, t)=& {h_{1}(t)},  t\in \mathbb{R}_{>0}, \label{v-system-boundary-1}\\
		v(z, 0)=&v_0(z), z\in(0,1), \label{v initial value}
	\end{align}
		\end{subequations}
where $b_0(z,t)$ is a space-time-varying function, $q_0\in\mathbb{R}_{>0}$ is a constant,   {{$h$}} represents in-domain disturbance, and  {${{h_0}}$, ${{h_1}}$} represent boundary disturbances. In this section, we always assume that $b_0,{{h}},{{h_0}},{{h_1}},v_0$ are continuous functions and
\begin{align*}
 \underline b:=\inf_{(z,t)\in Q_\infty}b_0(z,t)>0.
\end{align*}
%Similar to Proposition \ref{w-solution exsistence}, for any $T\in\mathbb{R}_{>0}$,
%system \eqref{v-system} admits a unique solution $v\in C(\overline Q_T)$.

For system \eqref{v-system}, we have the following proposition.
\begin{prpstn}\label{method-v}
 %For any $T\in\mathbb{R}_{>0}$,
 System \eqref{v-system}
 %admits a unique solution $v\in C(\overline Q_T)$ and
 is ISS in the $\sup$-norm, having the estimate
		\begin{align*}
		\sup_{z\in(0,1)}\left|{v}(z,T)\right| \leq C\left(e^{-\sigma T}\sup_{z\in(0,1)}\left|{v}_0(z)\right|+\sup_{(z,t)\in Q_T}\left|{{h}}(z,t)\right|+\sup_{t\in(0,T)}\left|{{h_0}}(t)\right|+\sup_{t\in(0,T)}\left|{{h_1}}(t)\right|\right),\forall T\in\mathbb{R}_{>0},
	\end{align*}
		where ${\sigma}\in (0,\underline b)$  and $C$  are positive constants, which are independent of the disturbances ${{h}},{{h_0}}, {{h_1}}$ and initial datum $v_0$.
\end{prpstn}
\begin{proof}
 By virtue of the standard  density argument  \cite{Dashkovskiy2013}, without loss of generality, we may assume that  ${{h}}\in C^1(\overline Q_T;\mathbb{R})$, ${{h_0}}$, ${{h_1}}\in C^2([0,T];\mathbb{R})$ for any $T\in\mathbb{R}_{>0}$ and   $v_0\in D(\mathcal{A})$. Then, system \eqref{v-system} admits a classical solution $v \in C^2(Q_T)$. Therefore, the derivations  of $v$ w.r.t. $z$ and $t$ are   well-defined over $Q_T$.  Now, we apply  the generalized Lyapunov method to prove the ISS of   system~\eqref{v-system}. The proof consists of {four} main steps.
	
	\textbf{Step 1}: Define truncation functions.
	%And then using the equivenlent relationship between~\eqref{original system} and~\eqref{original target system}, we get ISS estimate of the close-loop system~\eqref{original system}.
	For an  arbitrary constant $r>1$, define truncation functions
	\begin{align}\label{gG}
		g(a ):= \begin{cases}a^r, & a \geq 0 \\
			0, & a <0\end{cases}\ \text{and}\
		G(a ):=\int_0^a  g(\tau) \mathrm{d} \tau,
	\end{align}
	which satisfy
	\begin{subequations}
		\begin{align}
			g(a )  \geq& 0, g^{\prime}(a) \geq 0, G(a) \geq 0,  \forall a  \in \mathbb{R},\label{gG1} \\
			g(a ) =&G(a )=0, \forall a \in \mathbb{R}_{\leq 0}.\label{gG2}
			%G(\theta ) =&\frac{1}{r+1} g(\theta ) \theta , \forall\theta  \in \mathbb{R} .\label{gG3}
		\end{align}
	\end{subequations}
	
	\textbf{Step 2:} Prove the upper bound of ${v}$.
	%We choose $\widetilde{\sigma}>0$ such that $ \underline {c}-\widetilde{\sigma}>0$
	For any ${\sigma}\in (0,\underline b)$, let
	\begin{align*}
		{{\phi}(z, t):=}e^{{\sigma} t} {v}(z, t),{\phi}_0(z):={v}_0(z), 
		\check{h}(z, t):= e^{{\sigma} t} {h}(z, t),
		{\check{{h}}_0(t):=} e^{{\sigma} t} {{h}_{0}}(t), {\check{h}_{1}}(t):= e^{{\sigma} t} {h}_{1}(t).
	\end{align*}
	For any $ T\in\mathbb{R}_{>0}$, it is easy to see that
	%\begin{subequations}\label{transferred error system}
	\begin{align*}
		{\phi}_t(z, t)=&{\phi}_{zz}(z, t)+({\sigma}-b_0(z,t)) {\phi}(z, t)+\check{h}(z, t),(z,t)\in Q_T, \\ %\label{29b}
		{\phi}_z(0, t)=&q_0 {\phi}(0, t)+\check{h}_{0}(t),t\in(0,T),\\
		{\phi}(1, t)=& \check{h}_{1}(t),t\in(0,T),\\
		\phi(z,0)=&\phi_0(z),z\in(0,1).
	\end{align*}
	%\end{subequations}
	
	Let
	\begin{align*}
		{\mathcal{H}}:=&\sup_{z\in (0,1)}\left|{\phi}_0(z)\right|+\frac {1}{\underline b-{\sigma}}\sup_{(z, t) \in {Q}_T}\left|\check{h}(z,t)\right|+\frac{1}{{q_0}}\sup_{t\in(0,T)}\left|\check {h}_{0}(t)\right|+\sup_{t\in(0,T)}\left|\check{h}_{1}(t)\right|.
	\end{align*}
	%	\begin{align*}
		%		{\mathcal{H}}:=&\max_{z\in [0,1]}\left|{\phi}_0(z)\right|+\frac{1}{{q}}\max_{t\in[0,T]}\left|\check {d}_0(t)\right|+\left\|\check{d}_1\right\|_{L^{\infty}(0, T)}  +\frac {1}{\lambda_0-{\sigma}}\left\|\check{\psi}\right\|_{L^{\infty}\left((0,T);L^\infty(0,1)\right)}.
		%	\end{align*}
	Note that
	\begin{align}\label{phi le hat k}
		{\phi}(1,t)\leq{\mathcal{H}},\forall t\in(0,T).
	\end{align}

	By the definitions of $g$ and $G$  and integrating by parts, we get
	\begin{align}\label{widetildephi}
		\frac{\text{d}}{\text{d} t} \int_0^1 G\left({\phi}-{\mathcal{H}}\right) \text{d} z
		=&\int_0^1 g\left({\phi}-{\mathcal{H}}\right) {\phi}_t \text{d} z \notag\\
		=&\int_0^1 g\left({\phi}-{\mathcal{H}}\right)\left({\phi}_{zz} +({\sigma}- b_0){\phi} +\check{h} \right)\text{d} z\notag\\
		%=&\int_0^1 g\left({\phi}-\check{\mathcal{D}}\right) {\phi}_{zz}(z,t) \text{d} x\notag\\
		%&+\int_0^1 g\left({v}-\check{\mathcal{D}}\right) \left(({\sigma}- c(t)){v}(z,t)+\check{\psi}(z,t)\right)\text{d}x \notag\\
		% =&\left.g\left({v}-\check{\mathcal{D}}\right) {v}_x\right|_{x=0} ^{x=1}-\int_0^1 g^{\prime}\left({v}-\check{\mathcal{D}}\right) {v}_x^2(z,t) \text{d} x \notag\\
		%& +\int_0^1 g\left({v}-\check{\mathcal{D}}\right)\left(({\sigma}- c(t)) {v}(z,t)+\check{\psi}(z,t)\right) \text{d} x \notag\\
		% = &g\left({v}(1, t)-\check{\mathcal{D}}\right) {v}_x(1, t)-g\left({v}(0, t)-\check{\mathcal{D}}\right) {v}_x(0, t) \notag\\
		%&-\int_0^1 g^{\prime}\left({v}-\check{\mathcal{D}}\right) {v}_x^2(z,t) \text{d} x\notag\\
		%&+ \int_0^1 g\left({v}-\check{\mathcal{D}}\right)\left(({\sigma}-c(t)) {v}(z,t)+\check{\psi}(z,t)\right) \text{d} x \notag\\
		= &g\left({\phi}(1, t)-{\mathcal{H}}\right) {\phi}_z(1, t) -g\left({\phi}(0, t)-{\mathcal{H}}\right)\left({q_0}{\phi}(0, t)+\check{h}_{0}(t)\right) \notag\\
		&-\int_0^1 g^{\prime}\left({\phi}-{\mathcal{H}}\right) {\phi}_z^2  \text{d} z+ \int_0^1 g\left({\phi}-{\mathcal{H}}\right)\left(({\sigma}- b_0) {\phi} +\check{h} \right) \text{d} z \notag\\
		:=&I_1(t)+I_2(t)+I_3(t)+I_4(t),\forall t\in (0,T),
	\end{align}
	where
	\begin{align*}
		I_1(t):=&g\left({\phi}(1, t)-{\mathcal{H}}\right) {\phi}_z(1, t),\\
		I_2(t):=&-g\left({\phi}(0, t)-{\mathcal{H}}\right)\left({q_0} {\phi}(0, t)+\check{h}_{0}(t)\right),\\
		I_3(t):=&-\int_0^1 g^{\prime}\left({\phi}-{\mathcal{H}}\right) {\phi}_z^2\text{d} z,\\
		I_4(t):=&\int_0^1 g\left({\phi}-{\mathcal{H}}\right)\left(({\sigma}-b_0) {\phi}+\check{h}\right) \text{d} z.
	\end{align*}

	For $I_1(t)$, we deduce from~\eqref{gG2} and~\eqref{phi le hat k}  that $g\left({\phi}(1,t)-{\mathcal{H}}\right)=0$ for all $ t\in (0,T)$. Thus, $I_1(t)=0$ in $(0,T)$.

	For $I_2(t)$, we claim  that $I_2(t)\leq 0$ in $(0,T)$. Indeed, when ${\phi}(0,t)\leq{\mathcal{H}}$ for some  $t\in(0,T)$,  we obtain from~\eqref{gG2}  that $g\left({\phi}(0,t)-{\mathcal{H}}\right)={0}$  for this $t\in (0,T)$ and hence, $I_2(t)=0$  for this $t\in (0,T)$. When ${\phi}(0,t)> {\mathcal{H}}$   for some   $t\in (0,T)$, by the definition of ${\mathcal{H}}$ and {$q_0>0$}, we have ${\phi}(0,t)\geq  \frac{1}{{q_0}} \left|\check{h}_{0}(t)\right|$ for such $t\in (0,T)$. Therefore, for this $t\in (0,T)$, it holds that
	\begin{align*}
		{q_0} {\phi}(0,t)+\check{h}_{0}(t)\geq |\check{h}_{0}(t)|+ \check{h}_{0}(t) \geq 0,
	\end{align*}
	which, along with \eqref{gG1}, ensures that $I_2(t)\leq 0$ for such  $t\in (0,T)$.

	For $I_3(t)$, by \eqref{gG1}, we get $I_3(t)\leq 0$ in $(0,T)$.

	For $I_4(t)$, we claim that $I_4(t)\leq 0$ in $(0,T)$. Indeed, when ${\phi}(z,t)\leq{\mathcal{H}}$  for some  $(z,t)\in Q_T$, we obtain by \eqref{gG2} that $g\left({\phi(z,t)}-{\mathcal{H}}\right)=0$  for such $(z,t)\in Q_T$  and hence, $I_4(t)=0$ for  the same $t\in (0,T)$. When ${\phi}(z,t) > {\mathcal{H}}$ for some $(z,t)\in Q_T$, by the definition of ${\mathcal{H}}$ and $\underline b-{\sigma}> 0$, we have $\phi(z,t)\geq \frac{1}{\underline b-{\sigma}}\left|\check{h}(z,t)\right| $ for such $(z,t)\in Q_T$, which leads to
	\begin{align*}
		(b_0-{\sigma}) {\phi}(z,t) -\check{h}(z,t)\geq(\underline b-{\sigma}) {\phi}(z,t) -\check{h}(z,t)\geq\left|\check{h}(z,t)\right| - \check{h}(z,t)  \geq 0 ,
		%(\widetilde{\sigma}-\lambda_0) \tilde{\phi}(z, t)+&\hat{f}(z, t) \leq 0.
	\end{align*}
	or, equivalently,
	\begin{align*}
		(\sigma-b_0) \phi(z,t)+\check{h}(z,t) \leq 0,
	\end{align*}
for this $(z,t)\in Q_T$.	This, along with \eqref{gG1}, implies that $I_4(t)\leq 0$ for the same $t\in (0,T)$.
	
	Since we have obtained $I_i(t)\le 0$ for $i=1,2,3,4$, we infer from \eqref{widetildephi} that
	\begin{align*}
		\frac{\text{d}}{\text{d} t} \int_0^1 G\left({\phi}-{\mathcal{H}}\right) \text{d} z \leq 0,\forall t\in (0,T).
	\end{align*}
	Therefore,   it holds that
	\begin{align*}
		\int_0^1G\left({\phi}(z,T)-{\mathcal{H}}\right)\text{d}z \leq \int_0^1G\left({\phi}_0(z)-{\mathcal{H}}\right)\text{d}z.
	\end{align*}
	This, along with the continuity of ${\phi}$, implies that
	\begin{align*}
		{\phi}(z,T)\leq{\mathcal{H}},\forall z\in(0,1).
	\end{align*}
	Then, we have
	\begin{align}\label{upper}
		{v}(z, T) =&e^{-{\sigma} T} {\phi}(z, T)  \notag\\
		%\leq& e^{-{\sigma} T} \check{\mathcal{D}}  \notag\\
		\leq &e^{-{\sigma} T}
		\left(\sup_{z\in (0,1)}\left|{\phi}_0(z)\right| +\frac {1}{\underline b-{\sigma}}\sup_{(z, t) \in {Q}_T}\left|\check{h}(z,t)\right|+\frac{1}{{q_0}}\sup_{t\in(0,T)}\left|\check {h}_{0}(t)\right|+\sup_{t\in(0,T)}\left|\check{h}_{1}(t)\right|\right)\notag\\ %\left(\left\|{\phi}_0\right\|_{L^{\infty}(0,1)}+\frac{1}{{q}}\left\|\check{d}_0\right\|_{L^{\infty}(0, T)}+\left\|\check{d}_4\right\|_{L^{\infty}(0, T)}+\frac{1}{\lambda_0-{\sigma}}\left\|\check{f}\right\|_{L^{\infty}\left((0,T);L^\infty(0,1)\right)}\right)  \notag\\
		%\leq& e^{-{\sigma} T} \max \left\{ \left\|\widetilde{w}_0\right\|_{L^{\infty}(0,1)},\frac{e^{{\sigma} T}}{q}\left\|\widetilde{d}_0\right\|_{L^{\infty}(0, T)}, \right.  \notag\\
		%& e^{{\sigma} T}\left\|\widetilde{d}_1\right\|_{L^{\infty}(0, T)},\left.\frac{e^{{\sigma} T}}{c_{2}}\left\|\widetilde{\psi}\right\|_{L^{\infty}\left((0,T);L^\infty(0,1)\right)}\right\} \notag\\
		%=&\max \left\{ e^{-\sigma T}\left\|\widetilde{w}_0\right\|_{L^{\infty}(0,1)},\frac{1}{b}\left\|\widetilde{d}_0\right\|_{L^{\infty}(0, T)}, \right. \notag\\
		%& \left.\left\|\widetilde{d}_1\right\|_{L^{\infty}(0, T)},\frac{1}{c_{2}}\left\|\widetilde{\psi}\right\|_{L^{\infty}\left((0,T);L^\infty(0,1)\right)}\right\}\notag\\
		\leq& e^{-\sigma T}\sup_{z\in(0,1)}\left|{w}_0(z)\right|+\frac{1}{\underline b-{\sigma}}\sup_{(z,t)\in Q_T}\left|{h(z,t)}\right|+\frac{1}{{q_0}}\sup_{t\in(0,T)}\left|{h}_{0}(t)\right|  + \sup_{t\in(0,T)}\left|{h}_{1}(t)\right|.
		%e^{-\sigma T}\max_{z\in[0,1]}\left|{w}_0(z)\right|_{L^{\infty}(0,1)}+\frac{1}{{q}}\left\|{d}_0\right\|_{L^{\infty}(0, T)}  + \left\|{d}_4\right\|_{L^{\infty}(0, T)}+\frac{1}{\lambda_0-{\sigma}}\left\|{f}\right\|_{L^{\infty}\left((0,T);L^\infty(0,1)\right)}.
	\end{align}
	\textbf{Step 3:} Prove the lower bound of {${v}$}.
	%Let $\overline{w}(z,t):=-{w}(z,t)$. It is clear that
	%\begin{subequations}\label{overline w}
	%\begin{align}
	%\overline{w}_t(z,t)=& \overline{w}_{zz}(z,t)- c(t) \overline{w}(z,t)\notag\\
	%&-\widetilde{\psi}(z,t), (z,t)\in Q_T,\\
	% \overline{w}_x(0, t)=& b\overline{w}(0, t)-{d}_0(t) ,t\in (0,T),\\
	% \overline{w}(1, t)=& -\widetilde{d}_1(t) ,t\in (0,T),\\
	% \overline{w}(x,0)=&\overline{w}_0(x),x\in(0,1).
	%\end{align}
	%\end{subequations}
	For the same ${\sigma}\in (0,\underline b)$, let
	\begin{align*}
		{\overline{\phi}(z, t):=} -e^{{\sigma} t}{v}(z, t),\overline{\phi}_0(z):=-{v}_0(z),
		\overline{h}(z, t):= -e^{{\sigma} t} {h}(z, t), {\overline{h}_{0}(t):=}  -e^{{\sigma} t} {h}_{0}(t), \overline{h}_{1}(t):= -e^{{\sigma} t} {h}_{1}(t).
	\end{align*}
	For any $ T\in\mathbb{R}_{>0}$, it is clear that $\overline{\phi}$ satisfies
	%\begin{subequations}\label{overlineoverlinev}
	\begin{align*}
		\overline{\phi}_t(z, t)=& \overline{\phi}_{zz}(z, t)+({\sigma}-b_0(z,t)) \overline{\phi}(z, t)+\overline{h}(z,t),(z,t)\in Q_T ,\\%\label{overlineoverlinev1}\\
		\overline{\phi}_z(0, t)=&{q_0}\overline{\phi}(0, t)+\overline{h}_{0}(t),t\in(0,T),\\
		\overline{\phi}(1, t)=&  \overline{h}_{1}(t),t\in(0,T),\\
		\overline{\phi}(z,0)=&\overline{\phi}_0(z),z\in(0,1).
	\end{align*}
	%\end{subequations}
	Let $g$, $G$ be defined by~\eqref{gG}, and
	\begin{align*}
		\overline{\mathcal{H}}:=&\sup_{z\in(0,1)}\left|\overline{\phi}_0(z)\right|+\frac{1}{\underline b-{\sigma}}\sup_{(z, t) \in {Q}_T}\left|\overline{h}(z,t)\right|+\frac{1}{{q_0}}\sup_{t\in(0,T)}\left|\overline{h}_{0}(t)\right|+\sup_{t\in(0,T)}\left|\overline{h}_{1}(t)\right|.
	\end{align*}
	%	\begin{align*}
		%		\overline{\mathcal{H}}:=&\left\|\overline{\phi}_0\right\|_{L^{\infty}(0,1)}+\frac{1}{{q}}\left\|\overline{d}_3\right\|_{L^{\infty}(0, T)}+\left\|\overline{d}_4\right\|_{L^{\infty}(0, T)}+\frac{1}{\lambda_0-{\sigma}}\left\|\overline{\psi}\right\|_{L^{\infty}((0,T);L^\infty(0,1))}.
		%	\end{align*}
	
	Analogous to Step 2, by deriving
	\begin{align*}
		\frac{\text{d}}{\text{d} t} \int_0^1 G\left(\overline{\phi}-\overline{\mathcal{H}}\right) \text{d} z \leq 0,\forall t\in (0,T),
	\end{align*}
	we get
	\begin{align*}
		\int_0^1G\left(\overline{\phi}(z, T)-\overline{\mathcal{H}}\right) \text{d} z \leq \int_0^1 G\left(\overline{\phi}_0(z)-\overline{\mathcal{H}}\right) \text{d} z.
	\end{align*}
	Therefore, it holds that
	%which along with the continuity of $\overline{v}$ implies that
	\begin{align*}
		\overline{\phi}(z, T) \leq \overline{\mathcal{H}}, \forall z \in(0,1) ,
	\end{align*}
	%
	%
	%Therefore, for any $ x\in(0,1)$, we have
	%\begin{align*}
	% \overline w(z,t)=&e^{-{\sigma} T} \overline{v}(z,t) \\
	% \leq& e^{-{\sigma} T} \overline{\mathcal{D}} \\
	% =&e^{-{\sigma} T} \max \left\{\left\|\overline{w}_0\right\|_{L^ \infty(0,1)},\frac{1}{b}\left\|\overline{d}_0\right\|_{L^{\infty}(0, T)},\right.\\
	%&\left\|\overline{d}_1\right\|_{L^{\infty}(0, T)}, \left.\frac{1}{c_{2}}\left\|\overline{\psi}\right\|_{L^{\infty}((0,T);L^\infty(0,1))}\right\} .
	%\end{align*}
	%
	%
	%In particular,
	which implies that
	\begin{align}\label{lower}
		{-{v}(z, T)}
		%\leq& e^{-{\sigma} T}\max \left\{\left\|\widetilde{w}_0\right\|_{L^\infty(0,1)},\frac{e^{{\sigma}T}}{b}\left\|\widetilde{d}_0\right\|_{L^{\infty}(0, T)},\right.\notag\\
		%&\left.e^{{\sigma}T}\left\|\widetilde{d}_1\right\|_{L^{\infty}(0, T)},\frac{e^{{\sigma}T}}{c_{2}} \left\|\widetilde{\psi}\right\|_{L^{\infty}((0, T);L^\infty(0,1))}\right\}\notag\\
		% =&\max \left\{e^{-{\sigma} T}\left\|\widetilde{w}_0\right\|_{L^\infty(0,1)},\frac{1}{b}\left\|\widetilde{d}_0\right\|_{L^{\infty}(0, T)},\right. \notag\\
		%& \left.\left\|\widetilde{d}_1\right\|_{L^{\infty}(0 , T)}, \frac{1}{c_{2}}\left\|\widetilde{\psi}\right\|_{L^{\infty}((0, T);L^\infty(0,1))}\right\}\notag\\
		\leq e^{-{\sigma} T}\sup_{z\in(0,1)}\left|{v}_0(z)\right|+\frac{1}{\underline b-{\sigma}}\sup_{(z,t)\in Q_T}\left|{h}(z,t)\right|+\frac{1}{q_0}\sup_{t\in(0,T)}\left|{h}_{0}(t)\right|  +\sup_{t\in(0,T)}\left|{h}_{1}(t)\right|.
	\end{align}
	%	\begin{align}\label{lower}
		%		{-{w}(z, T)}
		%		%\leq& e^{-{\sigma} T}\max \left\{\left\|\widetilde{w}_0\right\|_{L^\infty(0,1)},\frac{e^{{\sigma}T}}{b}\left\|\widetilde{d}_0\right\|_{L^{\infty}(0, T)},\right.\notag\\
		%		%&\left.e^{{\sigma}T}\left\|\widetilde{d}_1\right\|_{L^{\infty}(0, T)},\frac{e^{{\sigma}T}}{c_{2}} \left\|\widetilde{\psi}\right\|_{L^{\infty}((0, T);L^\infty(0,1))}\right\}\notag\\
		%		% =&\max \left\{e^{-{\sigma} T}\left\|\widetilde{w}_0\right\|_{L^\infty(0,1)},\frac{1}{b}\left\|\widetilde{d}_0\right\|_{L^{\infty}(0, T)},\right. \notag\\
		%		%& \left.\left\|\widetilde{d}_1\right\|_{L^{\infty}(0 , T)}, \frac{1}{c_{2}}\left\|\widetilde{\psi}\right\|_{L^{\infty}((0, T);L^\infty(0,1))}\right\}\notag\\
		%		\leq e^{-{\sigma} T}\left\|{w}_0\right\|_{L^\infty(0,1)}+\frac{1}{q}\!\left\|{d}_3\right\|_{L^{\infty}(0, T)}  +\left\|{d}_4\right\|_{L^{\infty}(0 , T)}+\frac{1}{\lambda_0-{\sigma}}\left\|{\psi}\right\|_{L^{\infty}((0, T);L^\infty(0,1))}.
		%	\end{align}
	
\textbf{Step 4:} Estimate of $v$.	
%	 {By the definition of $f$ (see~\eqref{def-psi}),} we have
%	\begin{align}\label{defenition psi}
%		{\max_{(z,t)\in\overline Q_T}\left|{f}(z,t)\right|}\le& b_0\max_{t\in[0,T]}\left|{d}_1(t)\right|+\left(\overline{k}+1\right)\max_{t\in[0,T]}\left|{f}(t)\right|
%		+\overline{k}\max_{t\in[0,T]}\left|{d}_3(t)\right|,
%	\end{align}
	%	\begin{align}\label{defenition psi}
		%		{\left\|{\psi}\right\|_{L^{\infty}((0, T);L^\infty(0,1))}}\le& b_0\left\|{d}_1\right\|_{L^{\infty}(0, T)}+\left(\overline{k}+1\right)\left\|{f}\right\|_{L^{\infty}((0, T);L^\infty(0,1))}
		%		+\overline{k}\left\|{d}_3\right\|_{L^{\infty}(0, T)},
		%	\end{align}
%	{where $\overline{k}:=\max _{(z,{y})\in \Omega}|k(z, {y})|$ and $b_0:=\max_{z\in[0,1]}|\alpha(z)B_1|$.}
	By \eqref{upper} and  \eqref{lower},  we deduce that for any $ T\in\mathbb{R}_{>0}$,
	\begin{align*}
		\sup_{z\in(0,1)}\left|{v}(z,T)\right| \leq&  {C}\left(e^{-\sigma T}\sup_{z\in(0,1)}\left|{v}_0(z)\right|+\sup_{(z,t)\in Q_T}\left|h(z,t)\right|+\sup_{t\in(0,T)}\left|{h}_0(t)\right|+\sup_{t\in(0,T)}\left|{h}_{1}(t)\right|\right),
	\end{align*}
	%	\begin{align}\label{w-estimate}
		%		\left\|{w}[T]\right\|_{L^{\infty}(0,1)} \leq& c\left(e^{-\sigma T}\left\|{w}_0\right\|_{L^\infty(0,1)}+\left\|{d}_1\right\|_{L^{\infty}(0, T)}+\left\|{d}_2\right\|_{L^{\infty}\left((0,T);L^\infty(0,1)\right)}+\left\|{d}_3\right\|_{L^{\infty}(0, T)}+\left\|{d}_1\right\|_{L^{\infty}(0, T)}\right),
		%	\end{align}
	where  {$C$} is a positive constant depended on $q_0,\underline b$, and $\sigma$. The proof is complete.
		\end{proof}
%In this section, we prove   Proposition~\ref{target main result} and Theorem~\ref{original main result} given above.
%In this section, we investigate the ISS of system~\eqref{original system} under the state feedback control law~\eqref{state controller}.
%\subsection{main results}
%To establish ISS of the closed loop~\eqref{original system} under the control law~\eqref{state controller}, we first prove that the target system \eqref{original target system} is ISS.
%it needs to show that the transformation~\eqref{transformation} is invertible. We have the following Proposition holds.

 Based on  Proposition \ref{method-v},  we can show that the target system \eqref{original target system} is ISS, as stated in the following proposition.
\begin{prpstn}\label{target main result}
	The target system \eqref{original target system}
	is {\rm ISS}  in the $\sup$-norm, having the estimate%namely, for  any initial state $(X_0,w_0)\in \mathbb{R}^N \times C([0,1];\mathbb{R})$ and all $ T\in\mathbb{R}_{>0}$,
	%it \textcolor{blue}{holds that}
		\begin{align*}
		|X(T)|+\sup_{z\in(0,1)}\left|w(z,T)\right| \leq& C_1 \left(e^{-\zeta T}\left(\left|X_0\right|+\sup_{z\in (0,1)}\left|w_0(z)\right|\right)+\sup_{t\in(0,T)}\left|{D}(t)\right|\right. \\
		& \left.+\sup_{(z, t) \in {Q}_T}\left|f(z,t)\right|+\sup_{t\in(0,T)}\left|d_0(t)\right|+\sup_{t\in(0,T)}\left|d_1(t)\right|\right),\forall T\in\mathbb{R}_{>0},
	\end{align*}
		where  $C_1$ and $\zeta$ are positive constants,  {which are independent  of} disturbances $f, D, d_0, d_1$ and initial datum $(X_0,w_0)$.
		\end{prpstn}
	\begin{proof}
 In view of the  pole placement theorem (see \cite{Wonham1967}), we  {construct the  following Lyapunov} function
		\begin{align}\label{lyapunov functional}
			V(t)=X(t)^{\top} P X(t), \forall t\in(0,T),
		\end{align}
		where  the $N\times N$ matrix $P=P^{\top}>0$ is the solution to the Lyapunov equation
		\begin{align*}
			P(A+B K)+(A+B K)^{\top} P=-H
		\end{align*}
		with some  $H=H^{\top}>0$.

		 {For \eqref{lyapunov functional}, it is clear that}
		\begin{align}\label{V-leq-geq}
			\lambda_{\min}(P)|X(t)|^2\le V(t)\le \lambda_{\max}(P)|X(t)|^2, \forall t\in(0,T),
		\end{align}
		where 	$\lambda_{\min}(P)$ and  $\lambda_{\max}(P)$ are respectively the smallest and the largest eigenvalues of $P$, and both of them are positive numbers.
		
Computing the derivative of the Lyapunov function~\eqref{lyapunov functional} w.r.t. $t$, we can arrive at
		\begin{align}\label{derivate-V}
			\dot{V}(t)=&-X(t)^{\top} H X(t)+2 X(t)^{\top} P B w(0, t)+2 X(t)^{\top} P {D}(t)\notag\\
			\leq&-\frac{\lambda_{\min}(H)}{2}|X(t)|^2+\frac{4|P B|^2}{\lambda_{\min}(H)}|w(0, t)|^2+\frac{4\left|P {D}(t)\right|^2}{\lambda_{\min}(H)}, \forall t\in (0,T).
		\end{align}
		 {Note that  system  \eqref{original target system} has a special form of  system \eqref{v-system}.} Thus, Proposition \Rref{method-v} ensures that
			\begin{align}\label{estimate-w}
			\sup_{z\in(0,1)}\left|{w}(z,T)\right| \leq& c_1\left(e^{-\sigma_0 T}\sup_{z\in(0,1)}\left|{w}_0(z)\right|+\sup_{(z,t)\in Q_T}\left|f(z,t)\right|+\sup_{t\in(0,T)}\left|{d}_0(t)\right|+\sup_{t\in(0,T)}\left|{d}_1(t)\right|\right),
		\end{align}
		where $\sigma_0\in(0,\inf_{z\in(0,1)}\lambda_0(z))$  and $c_1$ are positive constants,  {which are} independent of   disturbances $f, d_0,d_1$ and initial {datum} $w_0$.

We deduce by  \eqref{estimate-w} that
		\begin{align}\label{w0t}
			|w(0,t)|^2 \leq 2 c_1^2\left(e^{-2 \sigma_0 t}\left(\sup_{z\in(0,1)}|w_0(z)|\right)^2+\mathcal{D}^2\right), \forall t\in (0,T),
		\end{align}
		%	\begin{align}\label{w0t}
			%		|w(0,T)|^2 \leq 2 c^2\left(e^{-2 \sigma t}\|w_0\|^2_{L^\infty(0,1)}+\mathcal{D}^2\right),
			%	\end{align}
		where  $\mathcal{D}:=\sup_{t\in(0,T)}\left|D(t)\right|+\sup_{(z,t)\in Q_T}\left|f(z,t)\right|+\sup_{t\in(0,T)}\left|{d}_0(t)\right|+\sup_{t\in(0,T)}\left|{d}_1(t)\right|$.
		
		Substituting \eqref{w0t}   into \eqref{derivate-V},  we get
		\begin{align*}
			\dot{V}(t)%\leq&-\frac{\lambda_{\min}(Q)}{2}|X|^2+\frac{4|P B|^2}{\lambda_{\min}(Q)}|w(0, t)|^2\\
			%	&+\frac{4\left|P B_1\right|^2}{\lambda_{\min}(Q)}\left|d_1(t)\right|^2\\
			\leq&-\frac{\lambda_{\min}(H)}{2\lambda_{\max}(P)}V(t)+\frac{8|P B|^2c^2}{\lambda_{\min}(H)}e^{-2\sigma_0 t}\left(\sup_{z\in(0,1)}|w_0(z)|\right)^2+\frac{8\left|P B\right|^2c^2}{\lambda_{\min}(H)}\mathcal{D}^2
			+\frac{4\left|P {D}(t)\right|^2}{\lambda_{\min}(H)}\notag\\
			\leq&-c_2V(t)+c_3\left(e^{-2\sigma_0 t}\left(\sup_{z\in(0,1)}|w_0(z)|\right)^2+\mathcal{D}^2+\left|{D}(t)\right|^2\right), \forall t\in (0,T),
		\end{align*}
		where  $c_2:=\frac{\lambda_{\min}(H)}{2\lambda_{\max}(P)}$ and $c_3:=\max\left\{\frac{8\left|P B\right|^2c_1^2}{\lambda_{\min}(H)},\frac{4\left|P \right|^2}{\lambda_{\min}(H)}\right\}$.
		
		By the Gronwall's inequality, for any $T\in\mathbb{R}_{>0}$, we   obtain
		\begin{align*}
			V(T) \leq& e^{-c_2 T} V(0)+c_3\left(\sup_{z\in(0,1)}|w_0(z)|\right)^2\int_0^T e^{-c_2(T-s)-2\sigma_0 s}\text{d} s+ c_3\mathcal{D}^2\int_0^T e^{-c_2(T-s)} \text{d} s\notag\\
			&+c_3 \left(\sup_{t\in(0,T)}|{D}(t)|\right)^2\int_0^T e^{-c_2(T-s)} \text{d} s\notag\\
			\leq& e^{-c_2 T}V (0)+\frac{c_3\left(e^{-2\sigma_0 T}-e^{-c_2 T}\right)}{c_2-2\sigma_0}\left(\sup_{z\in(0,1)}|w_0(z)|\right)^2
			+\frac{c_3(1-e^{-c_2 T })}{c_2}  \left( \mathcal{D}^2+\left(\sup_{t\in(0,T)}|{D}(t)|\right)^2\right),
		\end{align*}
		%	\begin{align*}
			%		V(T) \leqslant& e^{-c_2 T} V(0)+c_2\|w_0\|^2_{L^\infty(0,1)}\int_0^T e^{-c_1(T-s)-2\sigma s}\text{d} s+ c_2\mathcal{D}^2\int_0^T e^{-c_1(T-s)} \text{d} s
			%		+c_2 \|d_1\|^2_{L^\infty(0,T)}\int_0^T e^{-c_1(T-s)} \text{d} s\notag\\
			%		\leq& e^{-c_1 T}V (0)+\frac{c_2\left(e^{-2\sigma T}-e^{-c_1 T}\right)}{c_1-2\sigma}\|w_0\|^2_{L^\infty(0,1)}
			%		+\frac{c_2(1-e^{-c_1 T })}{c_1}  \left( \mathcal{D}^2+\|d_1\|^2_{L^\infty(0,T)}\right).
			%	\end{align*}
		which, along with \eqref{V-leq-geq}, yields
		\begin{align*}
			|X(T)|^2\leq& \frac{\lambda_{\max}(P)}{\lambda_{\min}(P)}e^{-c_2T}|X_0|^2+\frac{c_3\left(e^{-2\sigma_0 T}+e^{-c_2 T}\right)}{\lambda_{\min}(P)|c_2-2\sigma_0|}\left(\sup_{z\in(0,1)}|w_0(z)|\right)^2
			+\frac{c_3}{c_2\lambda_{\min}(P)}\left(\mathcal{D}^2+\left(\sup_{t\in(0,T)}|{D}(t)|\right)^2\right).
			% \leq& c_3\left(e^{-c_2T}|X_0|^2+e^{-2\sigma_0 T}\|w_0\|^2_{L^\infty(0,1)}\right)\notag\\
			%&+c_3\left(\mathcal{D}^2+\|d_1\|^2_{L^\infty(0,T)}\right),
		\end{align*}
		It follows that
		\begin{align*}
			|X(T)|\leq&  c_4\left(e^{-\frac{c_2}{2}T}|X_0|+\left(e^{-\sigma_0 T}+e^{-\frac{c_2}{2}T}\right)\sup_{z\in(0,1)}|w_0(z)|\right)+c_4\left(\mathcal{D}+\sup_{t\in(0,T)}|{D}(t)|\right)\notag\\
			\leq&  c_4\left(e^{-\zeta T}|X_0|+2e^{-\zeta T}\sup_{z\in (0,1)}|w_0(z)|\right)+c_4\left(\mathcal{D}+\sup_{t\in(0,T)}|{D}(t)|\right),
		\end{align*}
		where $c_4:=\max\left\{\sqrt{\frac{\lambda_{\max}(P)}{\lambda_{\min}(P)}},\sqrt{\frac{c_3}{\lambda_{\min}(P)|c_2-2\sigma_0|}},\sqrt{\frac{c_3}{c_2\lambda_{\min}(P)}}\right\}$ and  $\zeta:=\min\left\{\sigma_0,\frac{c_2}{2}\right\}$.

		%Putting~\eqref{w0t} into~\eqref{VT}, we can arrive at
		%	\begin{align*}
			%	 V(T) \leq& e^{-\lambda_0 T} V(0)\notag\\
			%	 &+\lambda_1 \int_0^T e^{-\lambda_0(T-S)} 2 c_1^2 e^{-2\sigma_0 s}\left\|w_0\right\|_{L^\infty(0,1)}^2 \text{d} s \\
			%	& +\lambda_1 \int_0^T e^{-\lambda_0(T-s)} 2 c_1^2 \mathcal{D}^2 \text{d} s\notag\\
			%	&+\lambda_2 \int_0^T e^{-\lambda_0(T- s)}\left\|d_1\right\|_{L^\infty(0,T)}^2 \text{d} s \\
			%	 \leq& e^{-\lambda_0 T} V(0)+\frac{2 \lambda_1 c_1^2 e^{-2\sigma_0 T}}{\lambda_0-2\sigma_0}\left\|w_0\right\|^2_{L^{\infty}(0,1)} \\
			%	& +\frac{2 c_1^2 \lambda_1}{\lambda_0} \mathcal{D}^2+\frac{\lambda_2}{\lambda_0}\left\|d_1\right\|^2_{L^\infty\left(0, T\right) }\\
			%	& +\frac{2 c_1^2 \lambda_1}{\lambda_0} \mathcal{D}^2+\frac{\lambda_2}{\lambda_0}\left\|d_1\right\|^2_{L^\infty(0, T)} ,
			%\end{align*}
			%where $\lambda_0:=\frac{\lambda_{\min (Q)}}{2}$, $\lambda_1:=\frac{4|PB|^2}{\lambda_{\min (Q)}}$, $\lambda_2:=\frac{4|PB_1|^2}{\lambda_{\min (Q)}}$.
			
Then, for any $ T\in\mathbb{R}_{>0}$, by the definition of $\mathcal{D}$ and \eqref{estimate-w}, we have
			%\begin{align*}
			%\frac{\lambda_{\min }(P)}{2}|x|^2 \leq V \leqq \quad e^{-\lambda_0 T} \cdot \frac{\lambda_{\max (P)}^2}{2}\left|x_0\right|^2+\frac{2 \lambda_1 c_1^2 e^{-26 T}}{\lambda_0-26}\left\|w_0\right\|^2
			%\end{align*}
			\begin{align*}
				| X(T)| +\sup_{z\in(0,1)}|w(z,T)|\leq&\left(c_1+2c_4\right)e^{-\zeta T}\sup_{z\in(0,1)}\left|w_0(z)\right|+c_4e^{-\zeta T}|X_0|
				+(c_1+2c_4)\sup_{t\in(0,T)}|{D}(t)|\notag\\
				&+(c_1+c_4)\left(\sup_{(z, t) \in {Q}_T}|f(z,t)|+\sup_{t\in(0,T)}|d_0(t)|+\sup_{t\in(0,T)}|d_1(t)|\right)\notag\\
				\leq&c_5e^{-\zeta T}\left(\sup_{z\in(0,1)}\left|w_0(z)\right|+|X_0|\right)\notag\\
				&+c_5\left(\sup_{t\in(0,T)}|{D}(t)|+\sup_{(z,t)\in Q_T}|f(z,t)|+\sup_{t\in(0,T)}|d_0(t)|+\sup_{t\in(0,T)}|d_1(t)|\right),
			\end{align*}
			where $c_5:=c_1+2c_4$.

			%By the definition of $\mathcal{D}$, we have
			%\begin{align*}
			%	\mathcal{D}^2+\|d_1\|^2_{L^\infty(0,T)}\leq& 5\|d_1\|^2_{L^\infty(0,T)}+4\|f\|^2_{L^{\infty}((0, T);L^\infty(0,1))}\notag\\
			%	&+4\|d_0\|^2_{L^\infty(0,T)}+4\|d_1\|^2_{L^\infty(0,T)}.
			%\end{align*}
			%
			%Putting it into \eqref{x+w}, we get
			%	\begin{align*}
				%			&| X |^2 +\|w[T]\|_{L^\infty(0,1)}^2\\
				%	\leqslant & c_3 \left(e^{-\mu T}\left(\left|X_0\right|^2+\left\|w_0\right\|^2_{L^\infty(0,1)}\right)+\left\|d_1\right\|^2_{L^\infty(0,T)}\right. \\
				%	& \left.+\left\|f\right\|^2_{L^\infty((0,1);L^\infty(0,T))}+\left\|d_0\right\|^2_{L^\infty(0,T)}+\left\|d_1\right\|^2_{L^\infty(0,T)}\right),
				%\end{align*}
				%where $\mu:=\min\left\{c_1,2\sigma\right\}$,$c_3:=\max\left\{\frac{\lambda_{\max}(Q)}{\lambda_{\min}(Q)},5c^2+\frac{5c_2}{\lambda_{\min}(Q)}\right\}$ .
				
				Thus, the ISS of system~\eqref{original target system} is established, and the proof of Proposition~\Rref{target main result} is complete.
				%$\mid x\left\|^2+\right\| w_0 \|_{2^{\infty}}^2 \leqslant C_2\left(e^{-\lambda_0 T}\left|x_0\right|^2+\left.e^{-6 T}\left\|w_0\right\|\right|^2+\left\|d_1\right\|\left\|^2+\right\| f\left\|^2+\right\| d_0 \|^2\right.$ $\left.+\left\|d_1\right\|^2\right)$
		\end{proof}
%	or
%	\begin{align*}
%		| X(T)| +\|w[T]\|_{L^\infty(0,T)}\leq & C_0 \left(e^{-\zeta T}\left(\left|X_0\right|+\left\|w_0\right\|_{L^\infty(0,1)}\right)+\left\|d_1\right\|_{L^\infty(0,T)}\right. \\
%		& \left.+\left\|f\right\|_{L^\infty((0,1);L^\infty(0,T))}+\left\|d_0\right\|_{L^\infty(0,T)}+\left\|d_1\right\|_{L^\infty(0,T)}\right),
%	\end{align*}

	%where $\sigma$ is a positive constant depended  on $\lambda$, $C_1$ is a positive constant depended  on $q,\lambda,\sigma,max_{z\in[0,1]}|\theta(z)B_1|,max_{(z,y)\in D}|k(z,y)|, \lambda_{\min }(Q), \lambda_{\max }(Q), c_2$, and $c_3$.
%\end{prpstn}

		\begin{pf2}
		  By virtue of Proposition~\Rref{kernel exitence}, Proposition~\Rref{equivelent}, and Proposition~\Rref{target main result},   Theorem~\Rref{original main result} can be proved in a standard way and  the details are omitted.
		\end{pf2}

		\section{Numerical Results}\label{sec:numerical results}
		In this section, we present some numerical results  for  system \eqref{original system} under the state feedback control law~\eqref{state controller} in the presence of different disturbances.
		In simulations, we set
		\begin{align*}
					A=&\left(\begin{matrix}0.5 & 2\\{0.1} & 0.6 \end{matrix}\right),\quad B=\left(\begin{matrix}5 \\ 4\end{matrix}\right),  %B_1=\left(\begin{matrix} 1 \\ 2\end{matrix}\right),
		 \quad b(z)= 1.1\pi(1+2z),\quad c(z)=0.1\pi(0.5+z)\sin(1-z), \quad
  q= 0.4, \\
			 	{D}(t)=&j(0.05\sin(t),  0.1 \sin(t))^{\top}, \quad
		 f(t)= j_0\sin(t+z),\quad
			d_0(t)=0.2j_1\sin(2t),\quad d_1(t)=0.3j_2\cos(t),\\
X_0=&m_0(0.25,-0.5)^\top, \quad u_0(z)=0.1m_0z^2(z-2)(z+0.5)(z-1),
		\end{align*}
		where $j,j_0\in\{0,2,4\}$, $j_1,j_2\in\{0,3,5\}$, and  $m_0\in\{1,2\}$ are used to describe the amplitudes of  disturbances and initial data.

In addition, for  the control law \eqref{state controller}, we set
\begin{align*}
K=\left(\begin{matrix} -4 & -2\end{matrix}\right), \quad \lambda_0(z)=0.01\sin(1+2z).
\end{align*}
The  kernel functions $k $  and $\alpha$  in the control law \eqref{state controller} are computed   via \eqref{sequence} and \eqref{equ19}, respectively.

Note that the conditions \eqref{q_c}, \eqref{equ4}, and \eqref{lambda_0} are fulfilled. Moreover, the eigenvalues of $A$ are {$1$ and $0.1$} while $A+BK$ is  Hurwitz.

	\begin{figure}[htbp]
	\centering
	\subfigure[Evolution of  $X$  when $m_0=1$]{
		\includegraphics[width=0.35\textwidth,height=6pc]{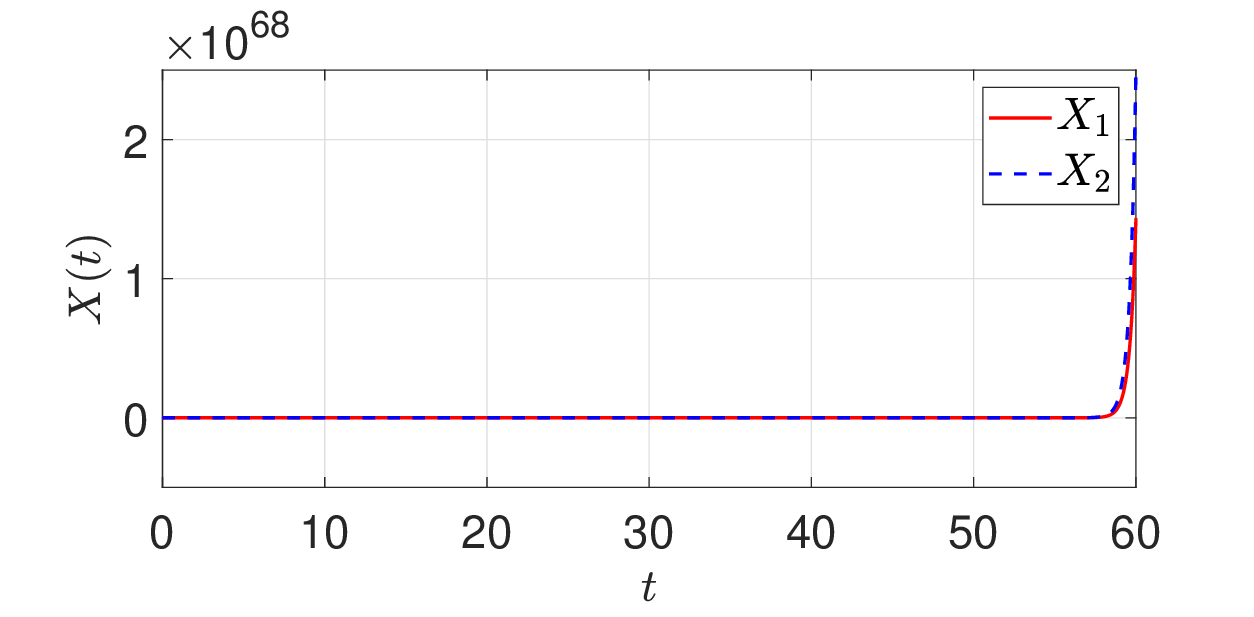}	
	}\noindent\hspace{10pt}
	\subfigure[Evolution of  $X$  when $m_0=2$]{
		\includegraphics[width=0.35\textwidth,height=6pc]{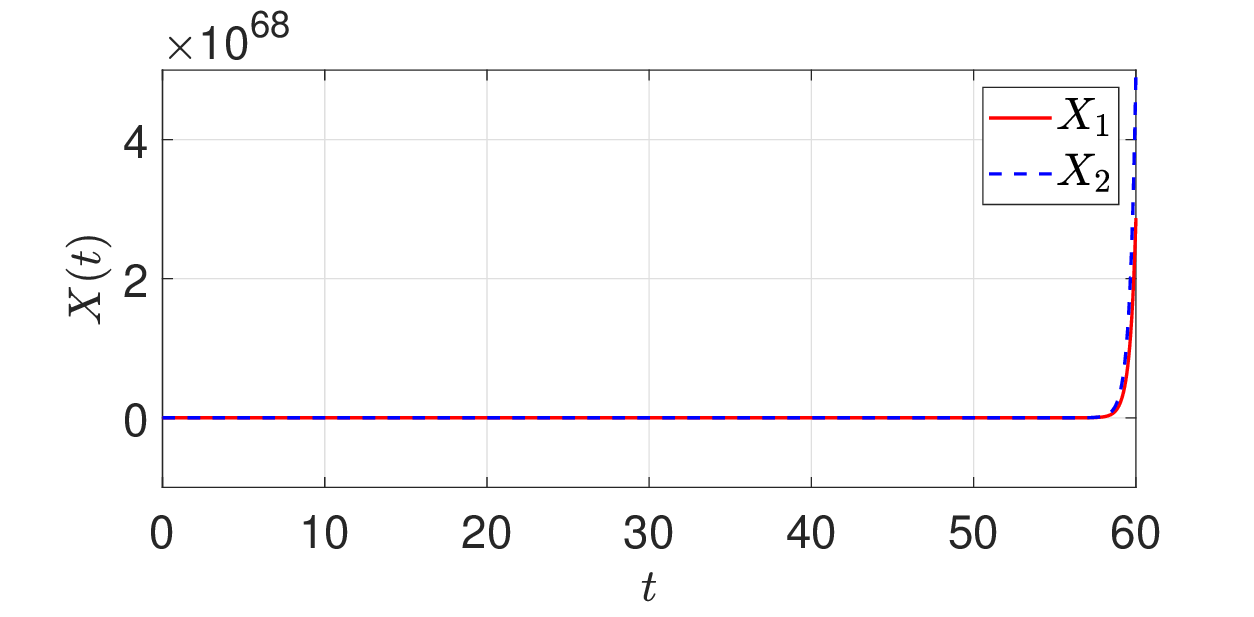}	
	}
	\subfigure[Evolution of  $u$ when $m_0=1$]{
		\includegraphics[width=0.35\textwidth,height=10pc]{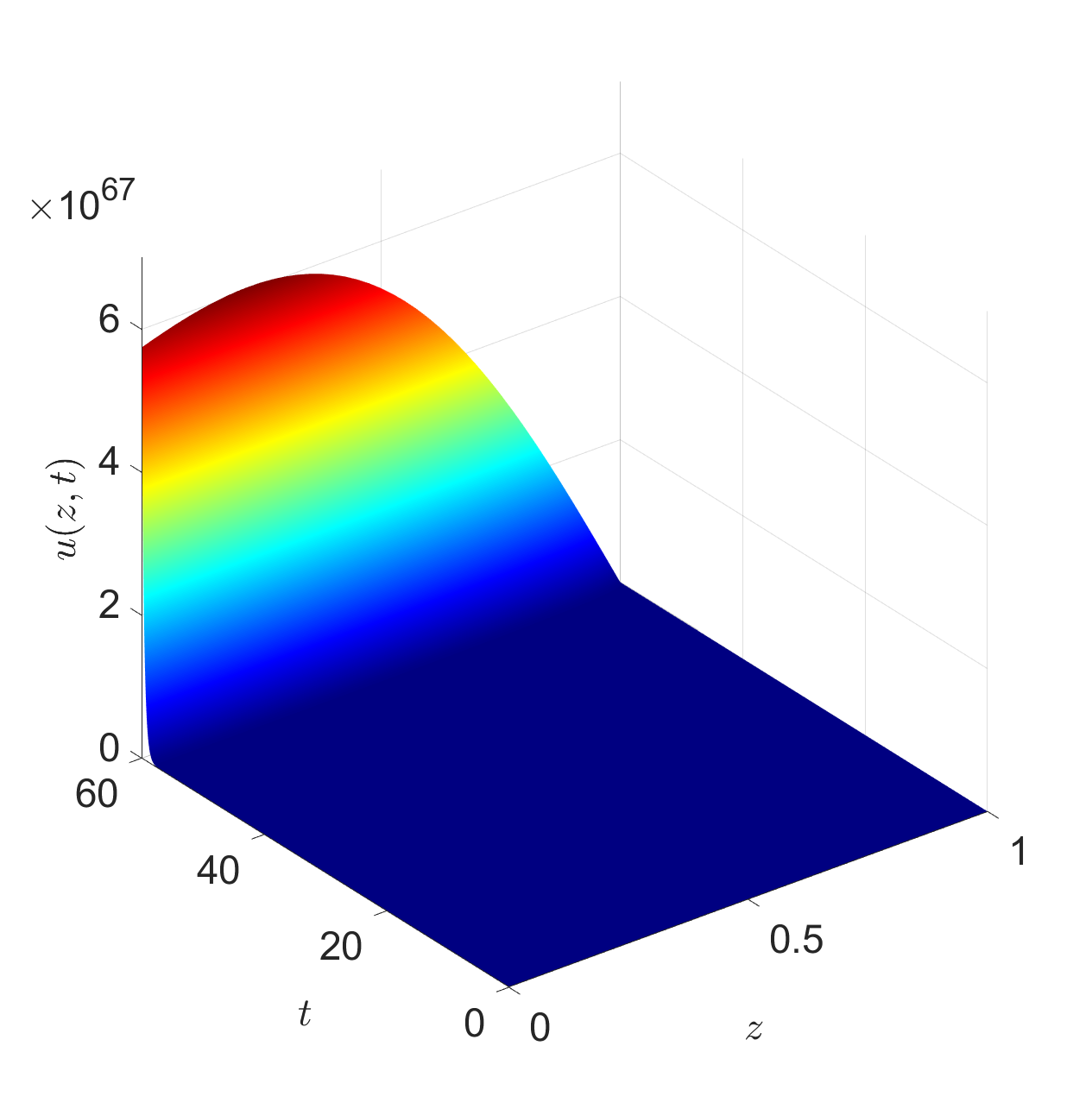}	
	}\hspace{10pt}
	\subfigure[Evolution of  $u$ when $m_0=2$]{
		\includegraphics[width=0.35\textwidth,height=10pc]{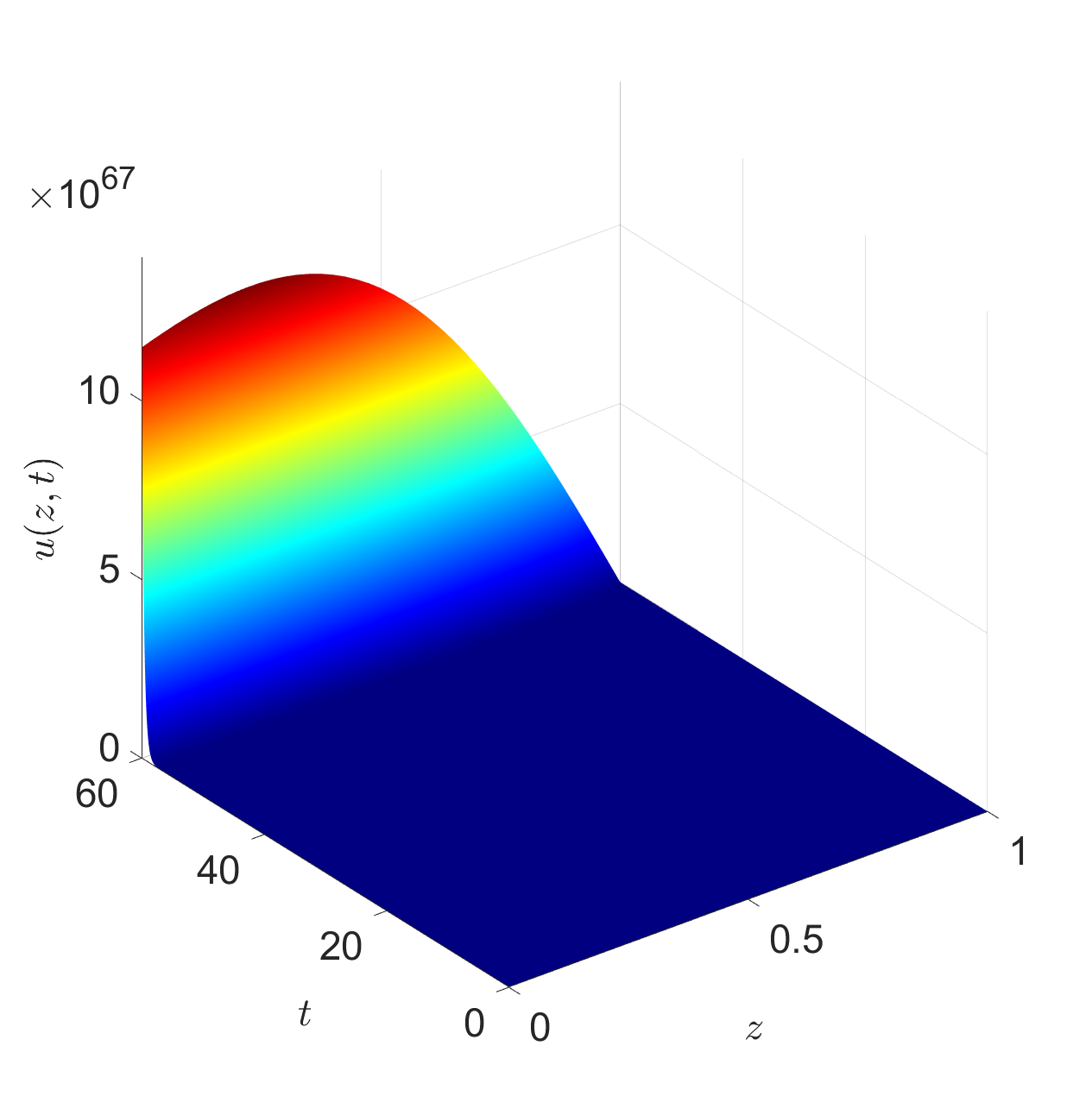}	
	}
	\subfigure[{Evolution of  $|X(t)|+\sup_{z\in (0,1)}\left|u(z,t)\right|$  when $m_0=1$}]{
		\includegraphics[width=0.35\textwidth,height=6pc]{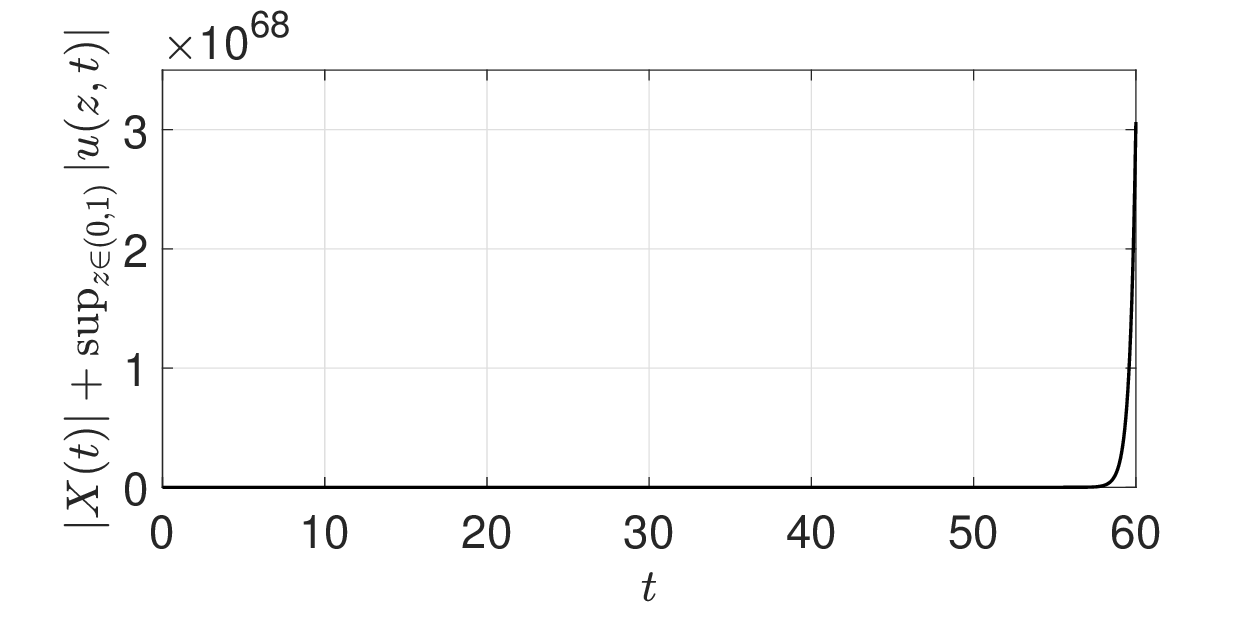}
	}\hspace{10pt}
	\subfigure[ {Evolution of  $|X(t)|+\sup_{z\in (0,1)}\left|u(z,t)\right|$  when $m_0=2$}]{
		\includegraphics[width=0.35\textwidth,height=6pc]{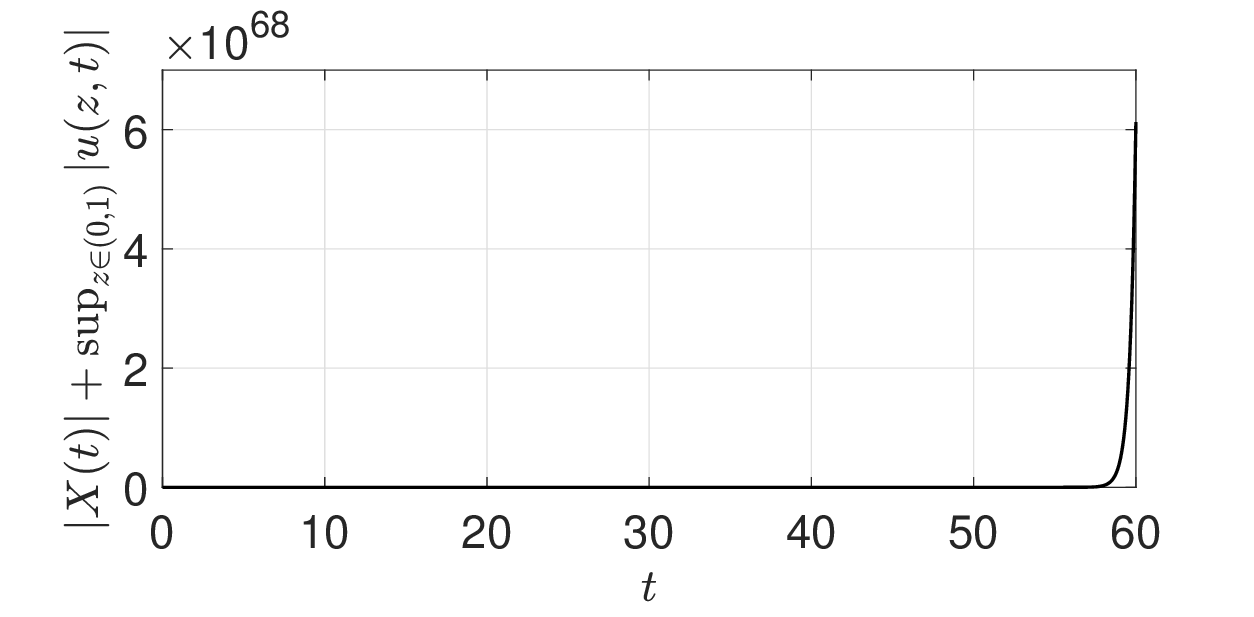}
	}
	%	\end{minipage}
\caption{Evolution of  $X$, $u$, and $|X(t)|+\sup_{z\in (0,1)}\left|u(z,t)\right|$  for system~\eqref{original system} in open loop with different {initial data}}		\label{fig1}
\end{figure}

Figure~\Rref{fig1} shows that system \eqref{original system} with different initial data in open loop, i.e., $U\equiv 0$, is unstable at the origin;  whereas Fig.~\Rref{fig2} (or black solid and dashed curves {in Fig.~\Rref{fig7}}) shows that  the disturbance-free system \eqref{original system}  is asymptotically stable at the origin under the control law \eqref{state controller}.

			\begin{figure}[t!]
			\centering
			\subfigure[Evolution of  $X$ when $m_0=1$]{
				%		\label{fig:2}
				\includegraphics[width=0.35\textwidth,height=6pc]{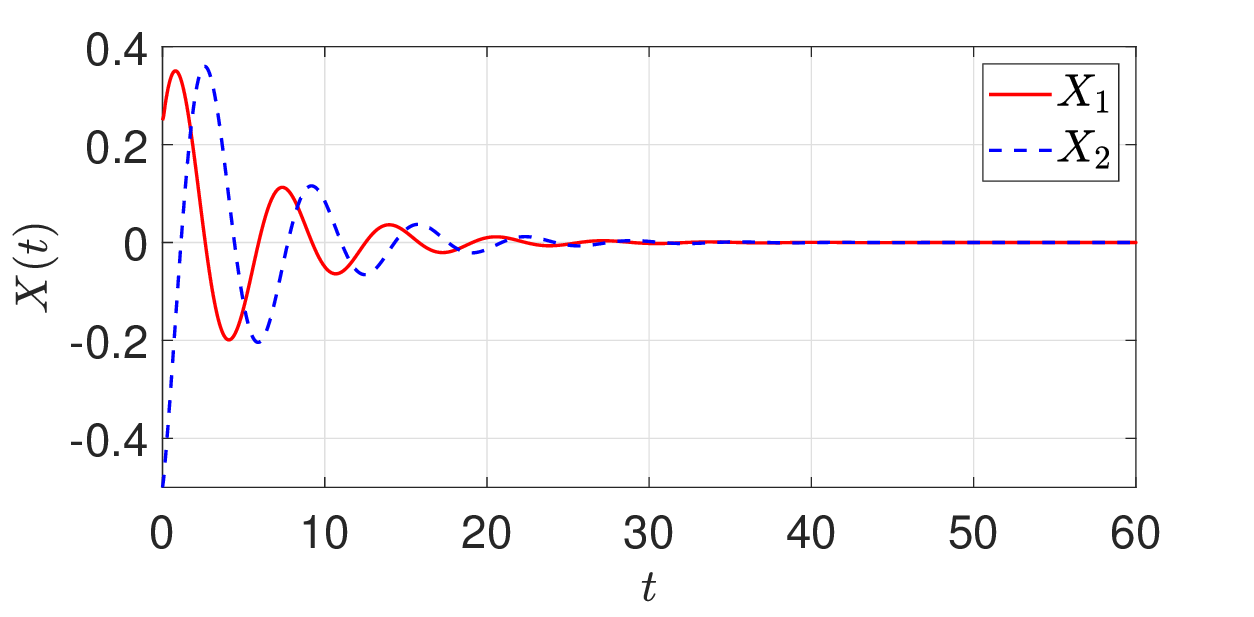}
			}\hspace{10pt}
			\subfigure[  Evolution of  $X$ when $m_0=2$]{
				%		\label{fig:2}
				\includegraphics[width=0.35\textwidth,height=6pc]{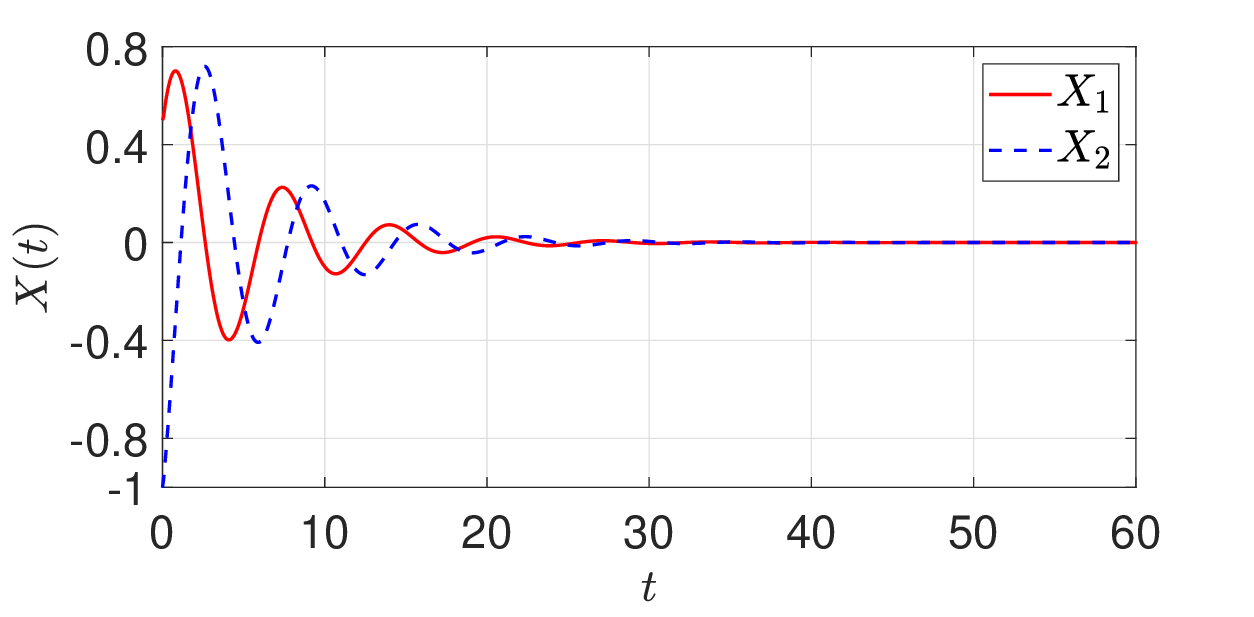}
			}
			\subfigure[Evolution of  $u$ when $m_0=1$]{
				\includegraphics[width=0.35\textwidth,height=10pc]{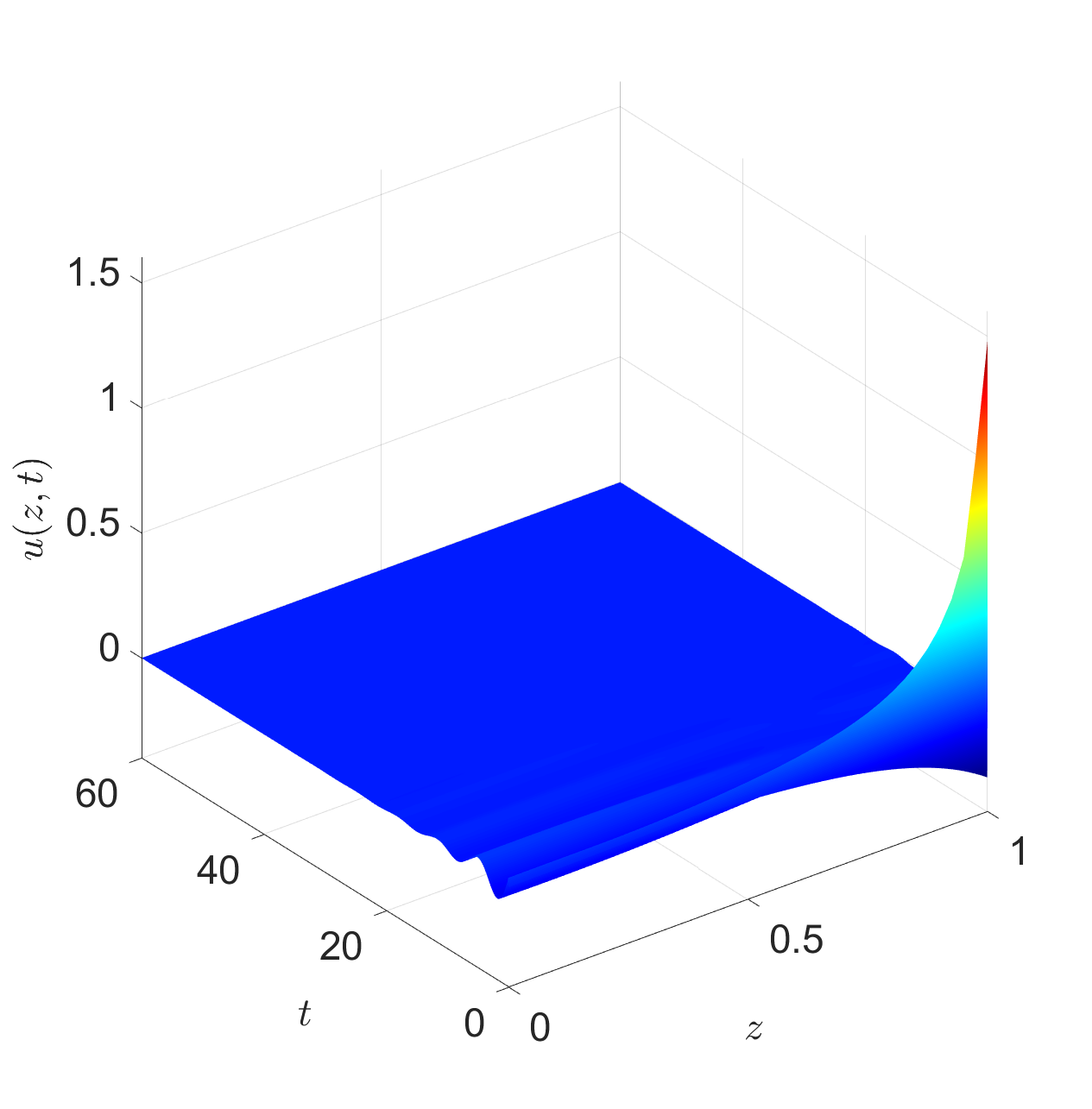}	
			}\noindent\hspace{10pt}
			\subfigure[Evolution of  $u$ when $m_0=2$]{
				\includegraphics[width=0.35\textwidth,height=10pc]{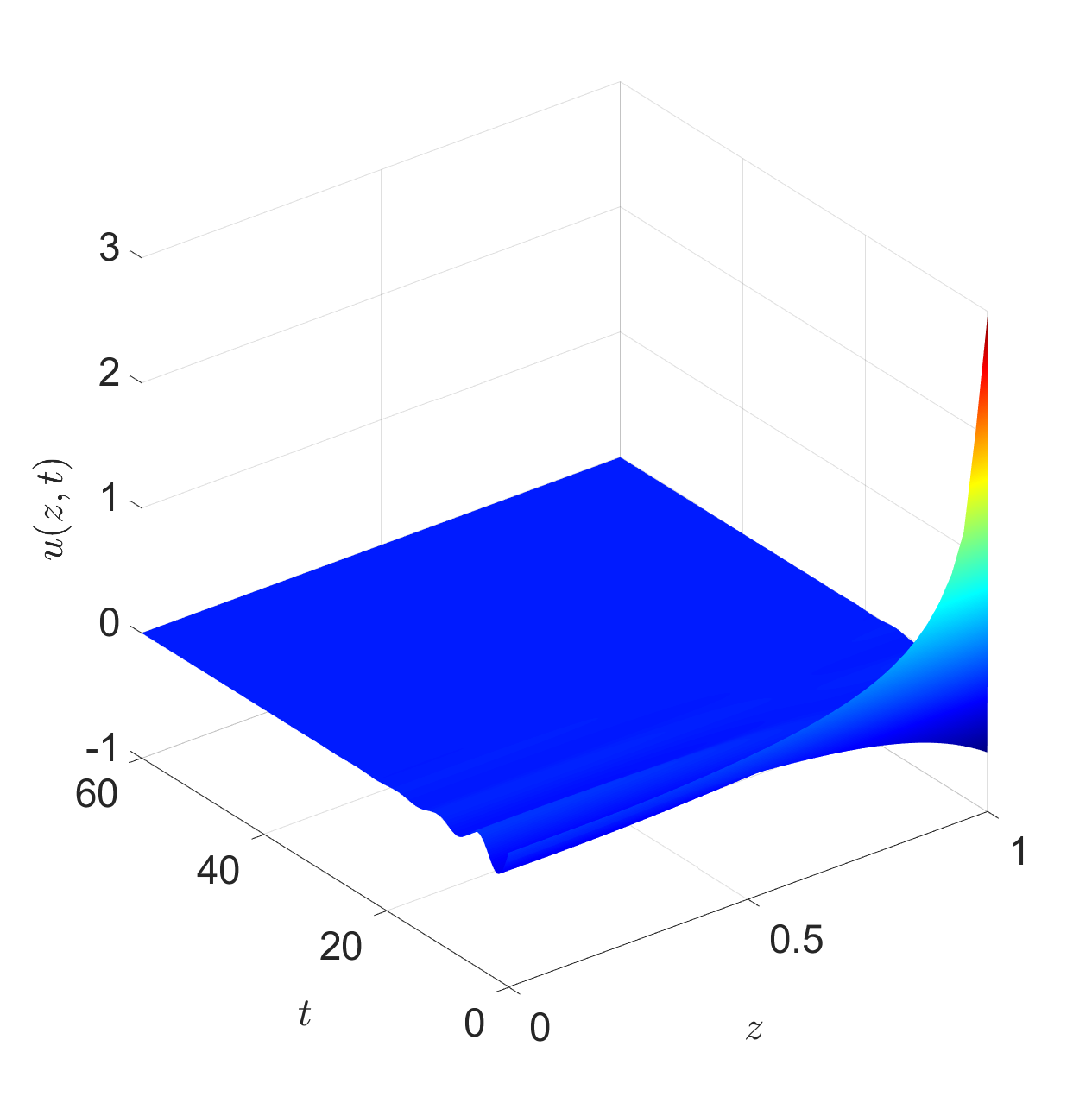}	
			}
			\caption{Evolution of $X$ and $u$ for the {closed-loop system} \eqref{original system}  with different {initial data in the absence of disturbances}}
			\label{fig2}
		\end{figure}

		Figure~\Rref{fig3}, Fig.~\Rref{fig4}, Fig.~\Rref{fig5}, and Fig.~\Rref{fig6} show that,   for different disturbances and initial data, the solutions to  the closed-loop system \eqref{original system}  remains bounded. For the same initial data, Figure~\Rref{fig3} (or Fig.~\Rref{fig4}) and Fig.~\Rref{fig5} (or Fig.~\Rref{fig6}) {present}  the evolution of the states for the closed-loop system \eqref{original system} with  different disturbances,  showing that the amplitude of the states  decreases as the   amplitude of disturbances   decreases. For the same  disturbances, Figure~\Rref{fig3}  and Fig.~\Rref{fig4} (or Fig.~\Rref{fig5} and Fig.~\Rref{fig6}) present the evolution of the states of closed-loop system \eqref{original system} with different initial data, showing that the amplitude of the states decreases as the amplitude of the initial data  decreases. 		
		In addition,  in accordance  to   Fig.~\Rref{fig2}, Fig.~\Rref{fig3}, Fig.~\Rref{fig4}, Fig.~\Rref{fig5}, and Fig.~\Rref{fig6}, the evolution of the states' norms for the closed-loop system \eqref{original system} with different disturbances (solid  or dashed lines in different colors) and initial data (lines in the same color) are presented   {in  Fig.~\Rref{fig7}(a) (or Fig.~\Rref{fig7}(b))}.  Especially, the black solid and dashed curves show  that the {states'} norms of the closed-loop system \eqref{original system} converge rapidly to the origin in the absence of disturbances for different initial data, and the bound of the states' norms decreases as the bound of the initial data decreases. The red and blue solid curves (or the red and blue dashed curves) show that,  for the same initial data, the bound of the states' norms  decreases as the amplitude of the disturbances diminishes.
		Notably, the ISS property of a system implies that trajectories under different initial data are driven by the same decay rate. Therefore, the long-time behavior  of the states is mainly determined by the disturbances rather than the initial data when  time tends to infinity, as shown by the solid and dashed lines in  Fig.~\Rref{fig7}.  	Overall,  {Fig.~\Rref{fig2}- Fig.\Rref{fig7}} well  depict the ISS of system~\eqref{original system} in closed loop.
		
				\begin{figure}[htbp!]
			\centering
			\subfigure[ Evolution of  $X$ when $j=j_0=2$]{
				\includegraphics[width=0.35\textwidth,height=6pc]{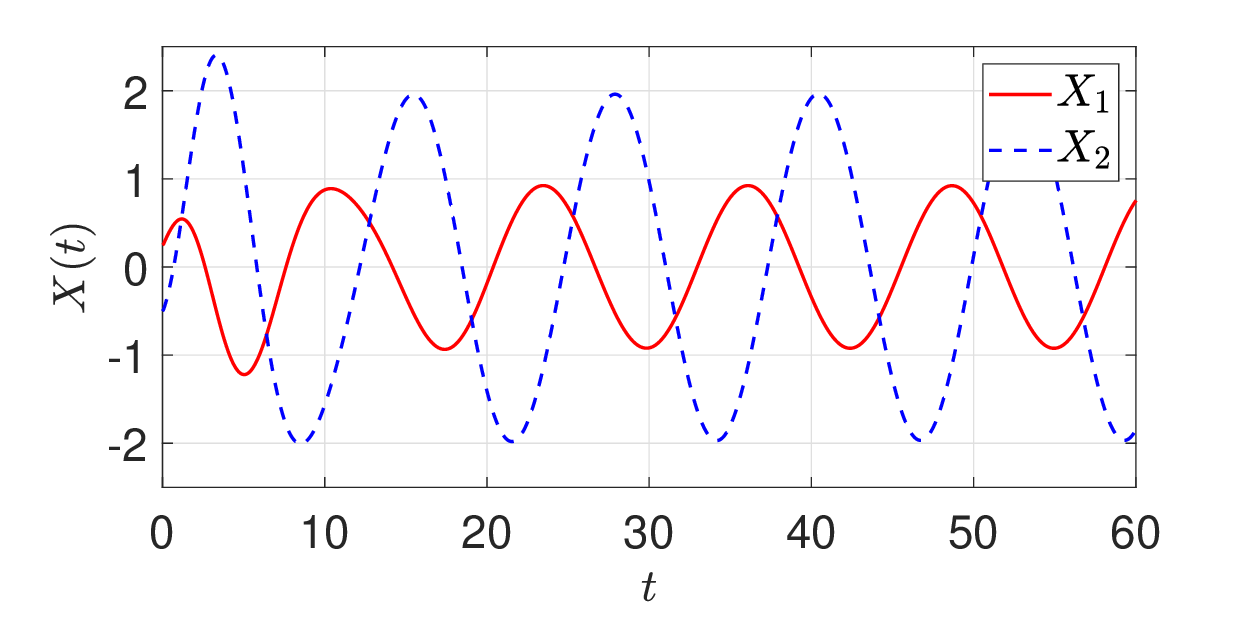}	
			}\noindent\hspace{10pt}
			\subfigure[ Evolution of  $X$ when $j=j_0=4$]{
				\includegraphics[width=0.35\textwidth,height=6pc]{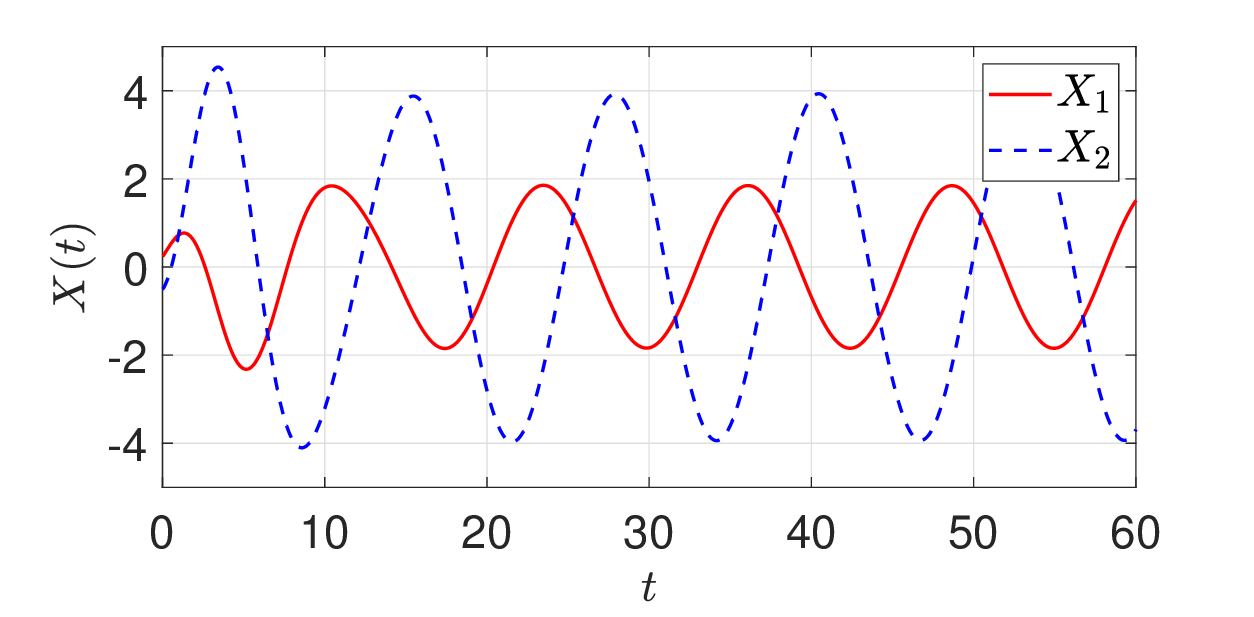}
			}
			\subfigure[ Evolution of  $u$ when $j=j_0=2$]{
				\includegraphics[width=0.35\textwidth,height=10pc]{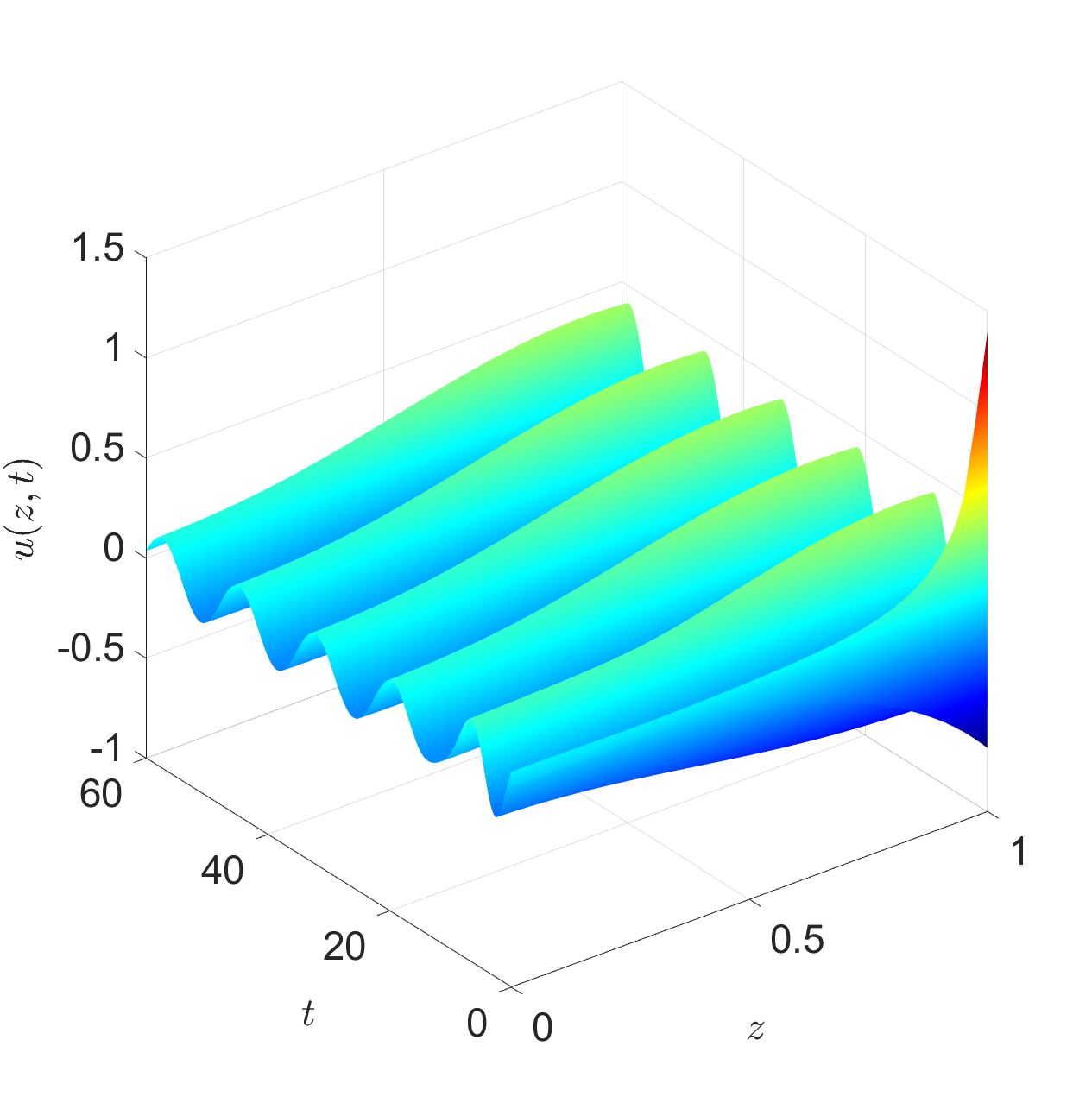}	
			}\noindent\hspace{10pt}
			\subfigure[  Evolution of  $u$ when $j=j_0=4$]{
				\includegraphics[width=0.35\textwidth,height=10pc]{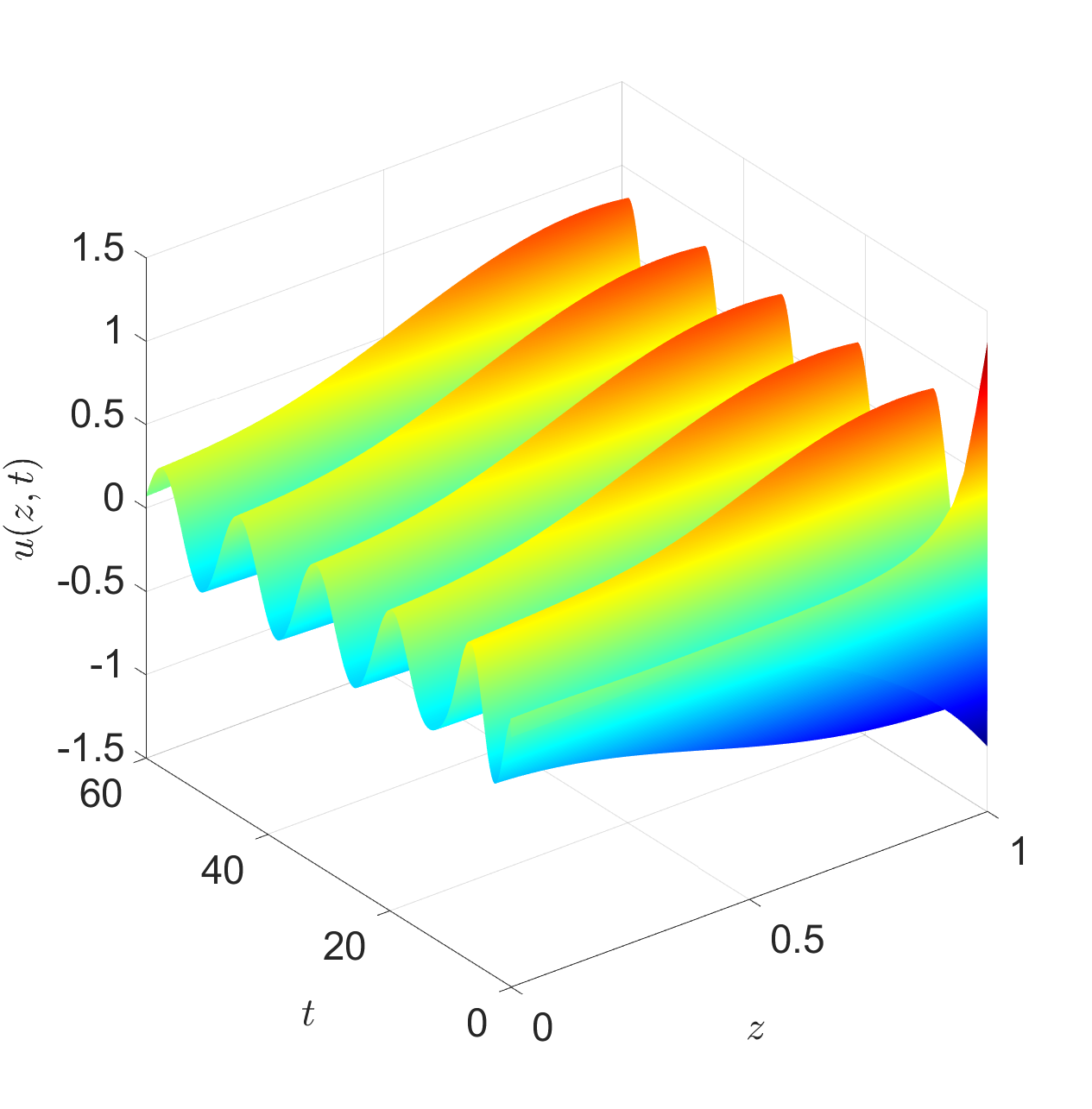}
			}
			\caption{Evolution of $X$ and $u$  {for the closed-loop system \eqref{original system} in the presence of different internal and in-domain disturbances}     {when $m_0=1$ and $j_1=j_2=0$}}
			\label{fig3}
		\end{figure}
		
		\begin{figure}[htbp!]
			\centering
			\subfigure[ Evolution of  $X$ when $j=j_0=2$]{
				\includegraphics[width=0.35\textwidth,height=6pc]{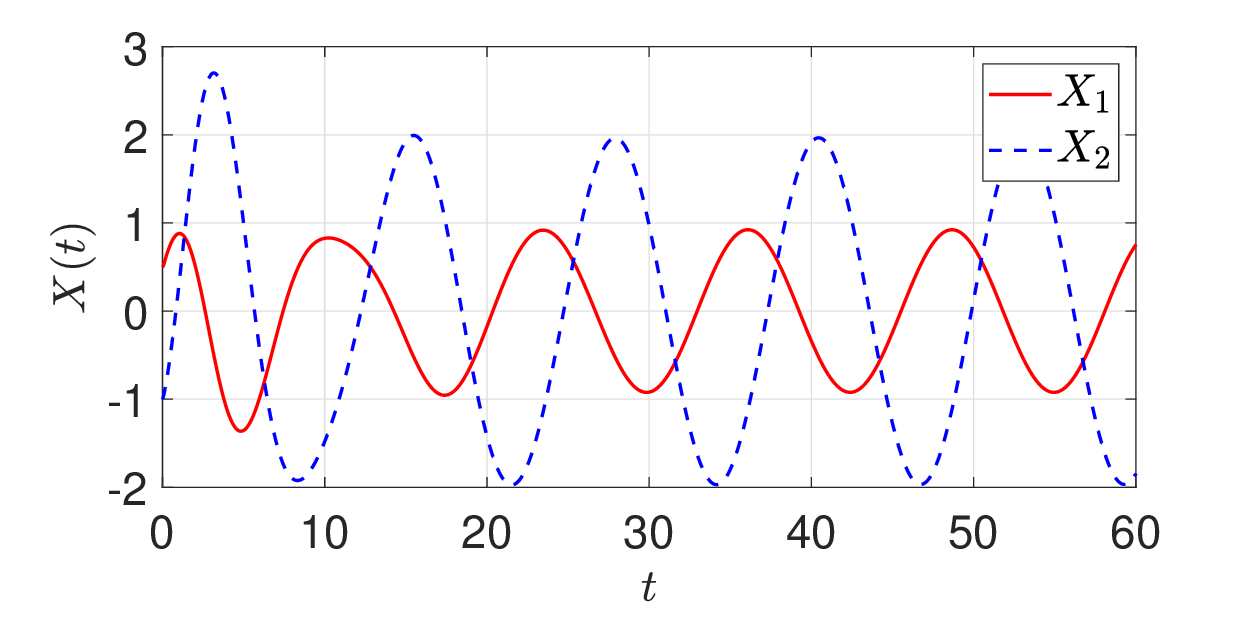}	
			}\noindent\hspace{10pt}
			\subfigure[ Evolution of  $X$ when $j=j_0=4$]{
				\includegraphics[width=0.35\textwidth,height=6pc]{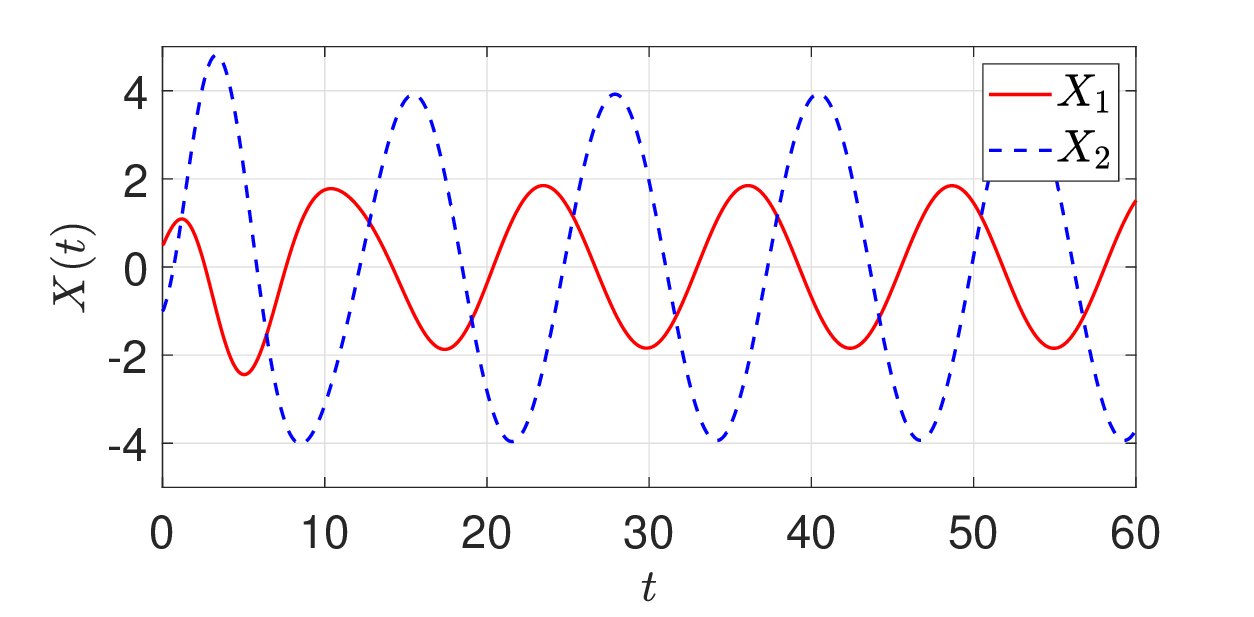}
			}
			\subfigure[ Evolution of  $u$ when $j=j_0=2$]{
				\includegraphics[width=0.35\textwidth,height=10pc]{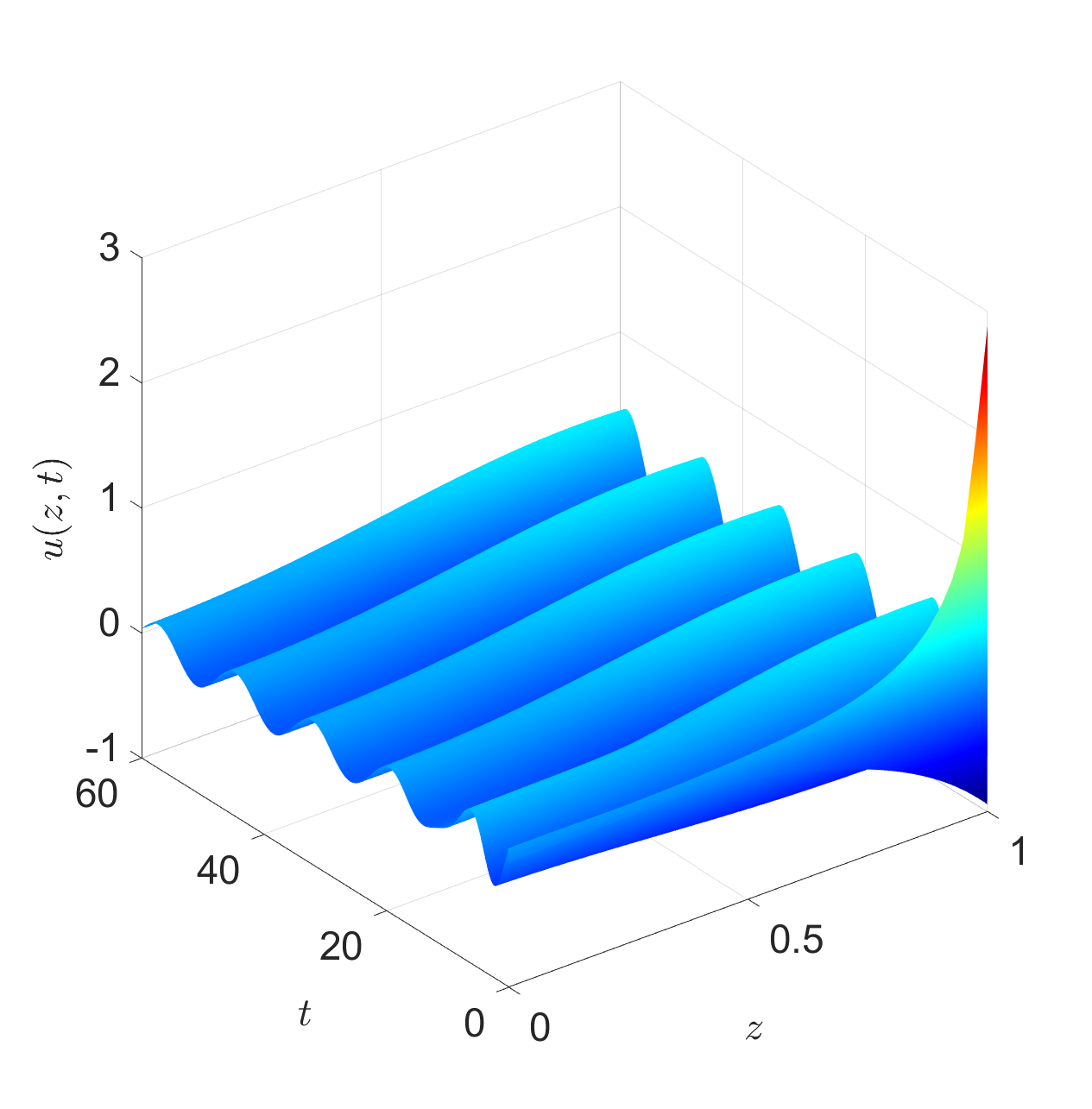}	
			}\noindent\hspace{10pt}
			\subfigure[  Evolution of  $u$ when $j=j_0=4$]{
				\includegraphics[width=0.35\textwidth,height=10pc]{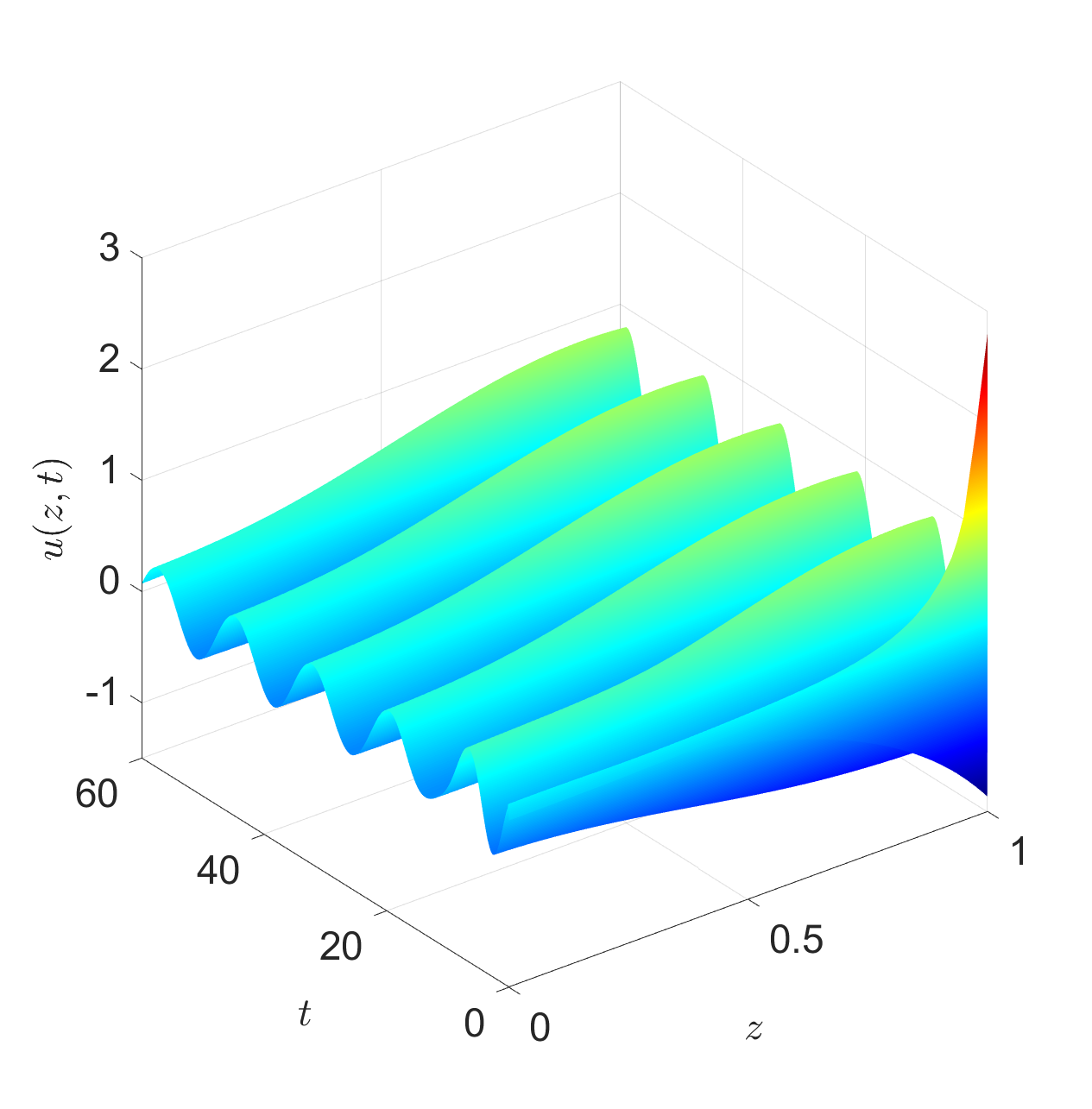}
			}
			\caption{Evolution of $X$ and $u$   {for the closed-loop system \eqref{original system} in the presence of different internal and in-domain disturbances}  {when $m_0=2$ and $j_1=j_2=0$}}
			\label{fig4}
		\end{figure}
		\begin{figure}[htbp!]
			\centering
			\subfigure[Evolution of  $X$ when $j_1=j_2=3$]{
				%		\label{fig:3}
				\includegraphics[width=0.35\textwidth,height=6pc]{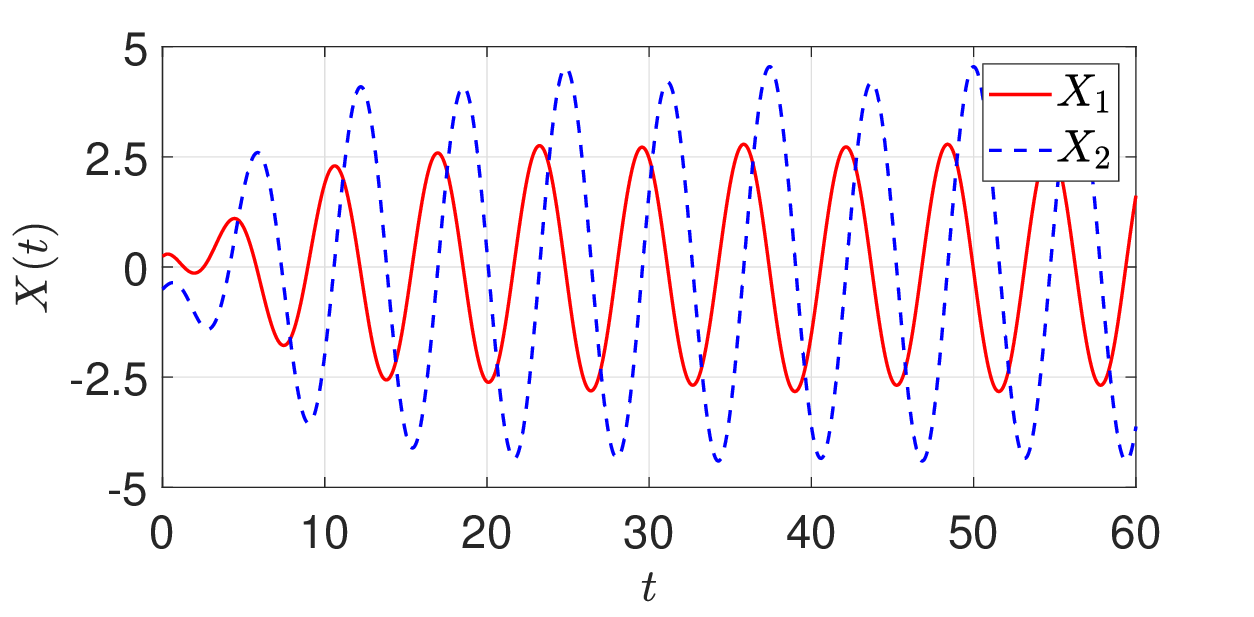}
			}\noindent\hspace{10pt}
			\subfigure[Evolution of  $X$ when $j_1=j_2=5$]{
				%		\label{fig:4}
				\includegraphics[width=0.35\textwidth,height=6pc]{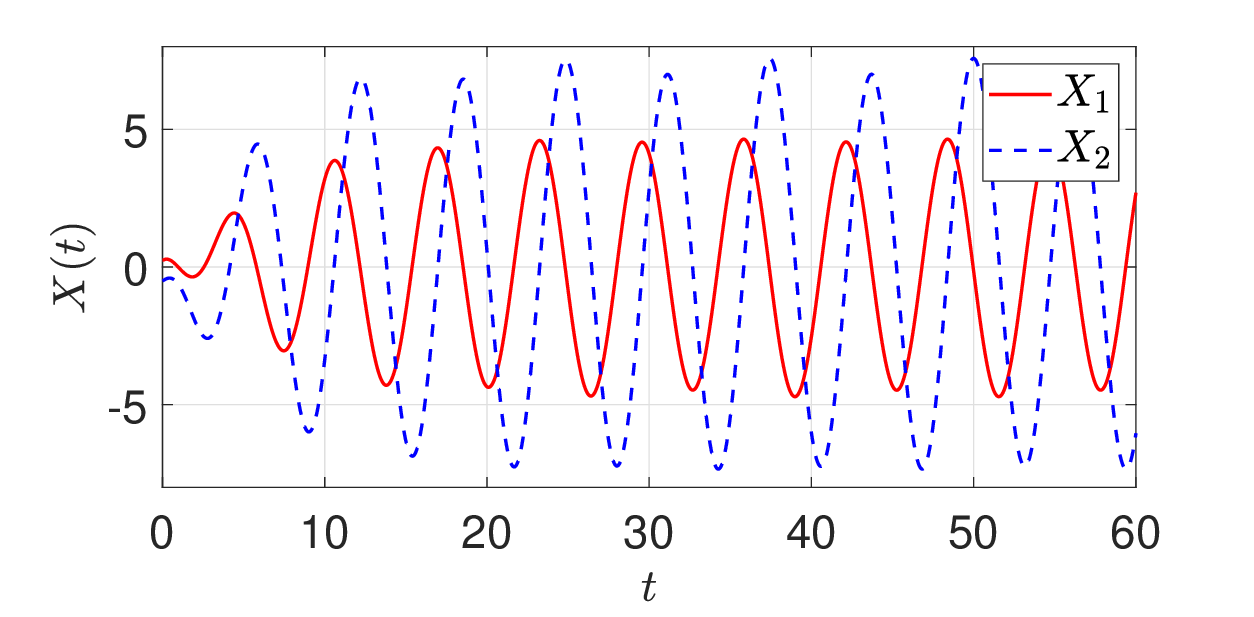}
			}
			\subfigure[ Evolution of  $u$ when $j_1=j_2=3$]{
				%		\label{fig:3}
				\includegraphics[width=0.35\textwidth,height=10pc]{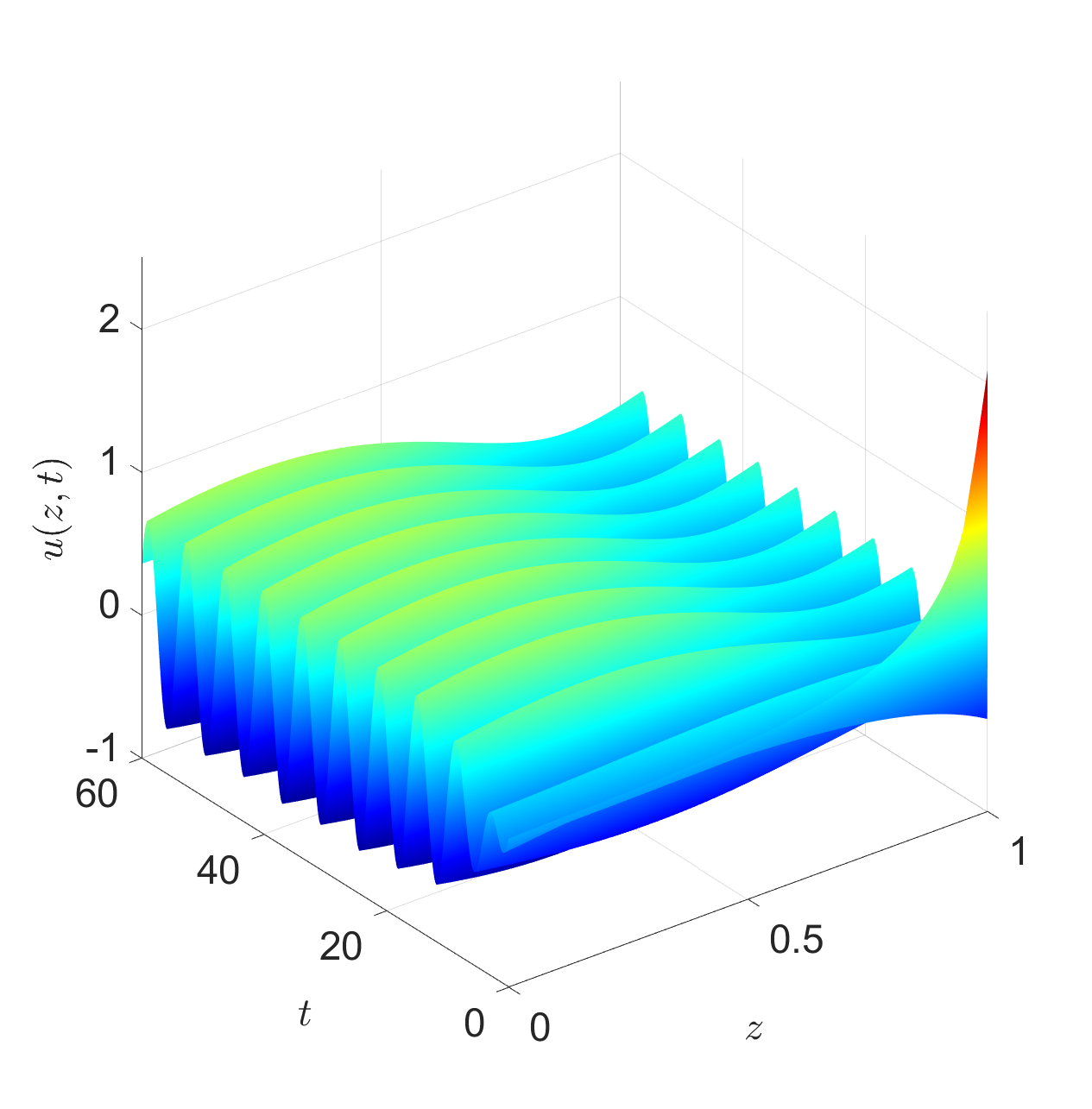}
			}\noindent\hspace{10pt}
			\subfigure[Evolution of  $u$ when $j_1=j_2=5$]{
				%		\label{fig:4}
				\includegraphics[width=0.35\textwidth,height=10pc]{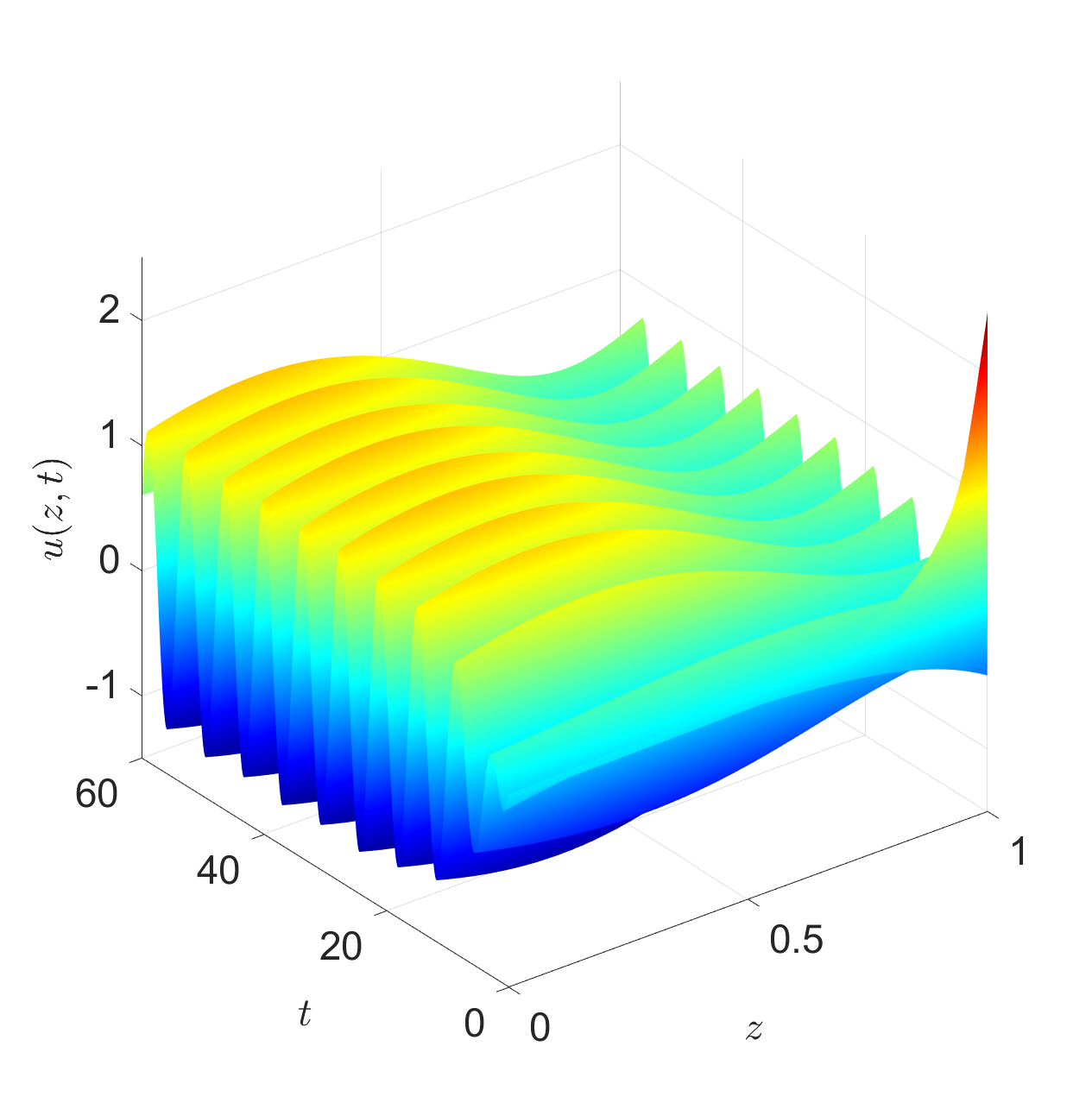}
			}
			\caption{Evolution of $X$ and $u$ {for the closed-loop system \eqref{original system} in the presence of} different boundary disturbances {when  $m_0=1$ and $j=j_0=0$}}
			\label{fig5}
		\end{figure}

\begin{figure}[t!]
	\centering
	\subfigure[Evolution of  $X$ when $j_1=j_2=3$]{
		%		\label{fig:3}
		\includegraphics[width=0.35\textwidth,height=6pc]{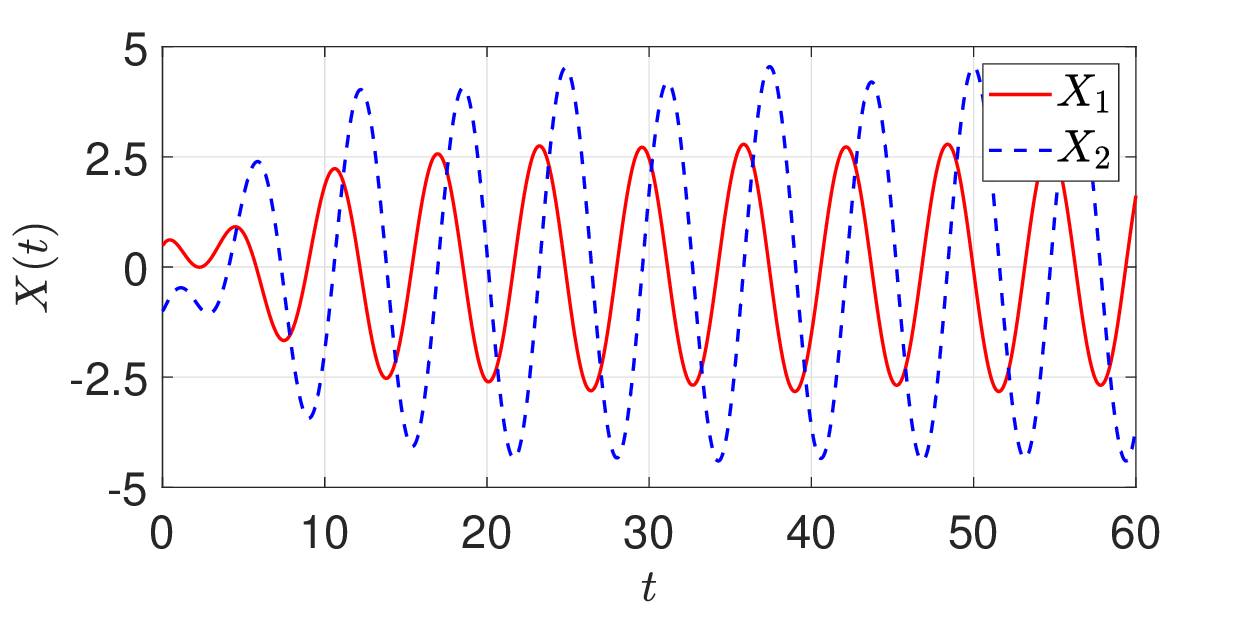}
	}\noindent\hspace{10pt}
	\subfigure[Evolution of  $X$ when $j_1=j_2=5$]{
		%		\label{fig:4}
		\includegraphics[width=0.35\textwidth,height=6pc]{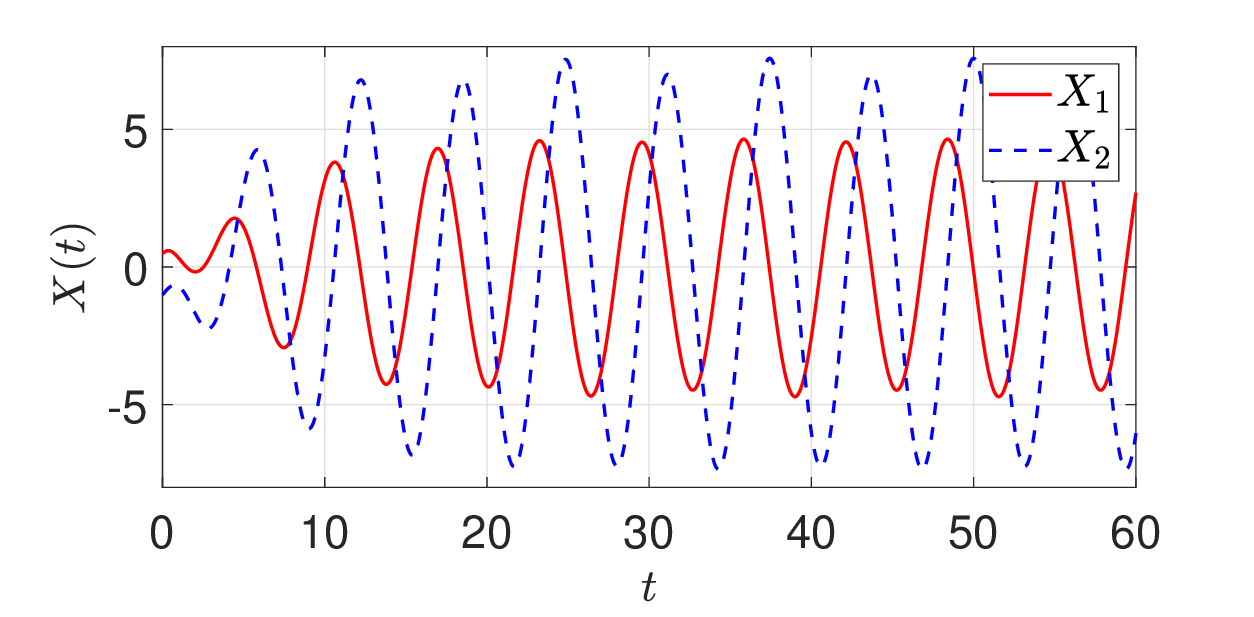}
	}
		\subfigure[ Evolution of  $u$ when $j_1=j_2=3$]{
		%		\label{fig:3}
		\includegraphics[width=0.35\textwidth,height=10pc]{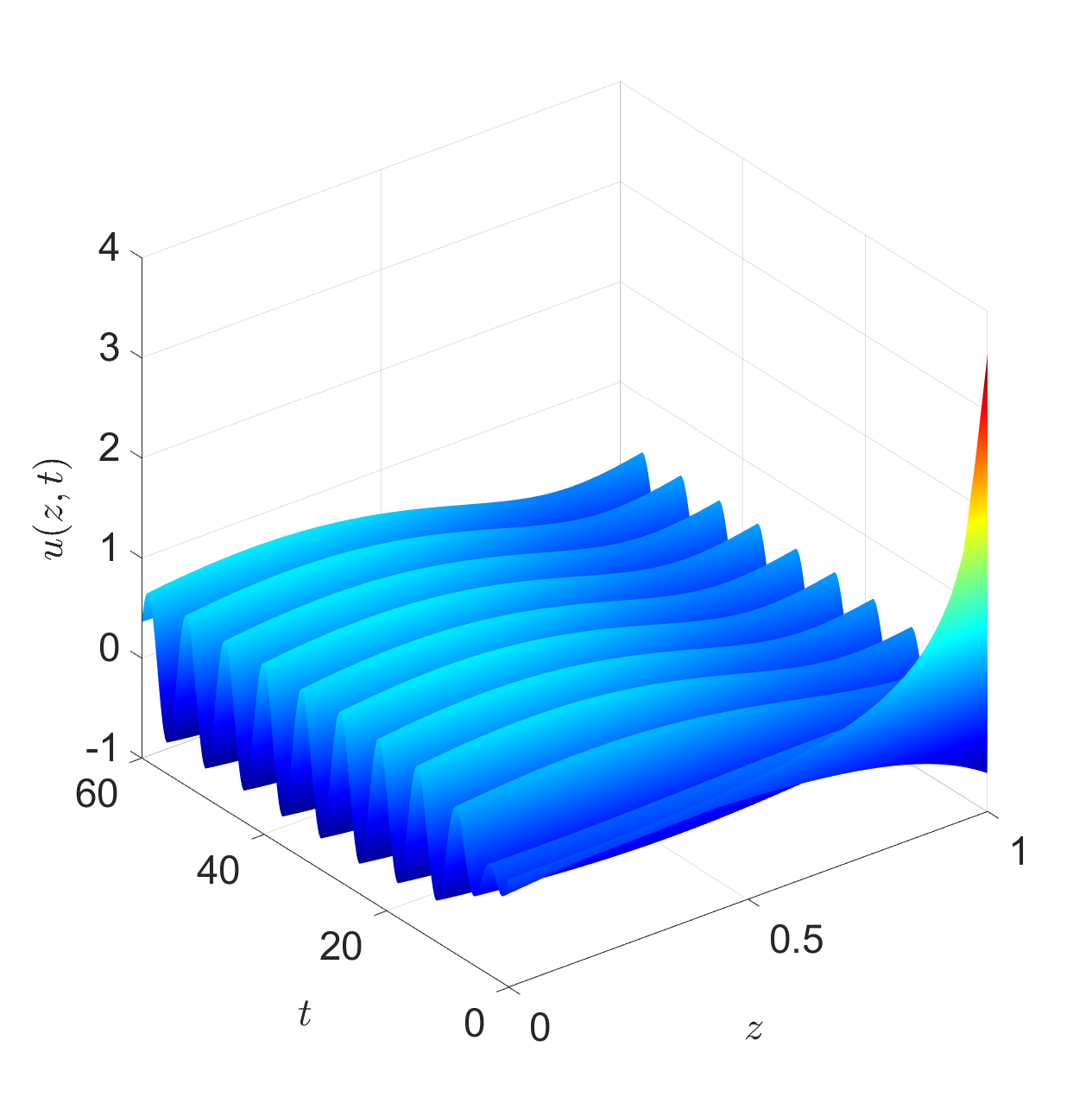}
	}\noindent\hspace{10pt}
	\subfigure[Evolution of  $u$ when $j_1=j_2=5$]{
		%		\label{fig:4}
		\includegraphics[width=0.35\textwidth,height=10pc]{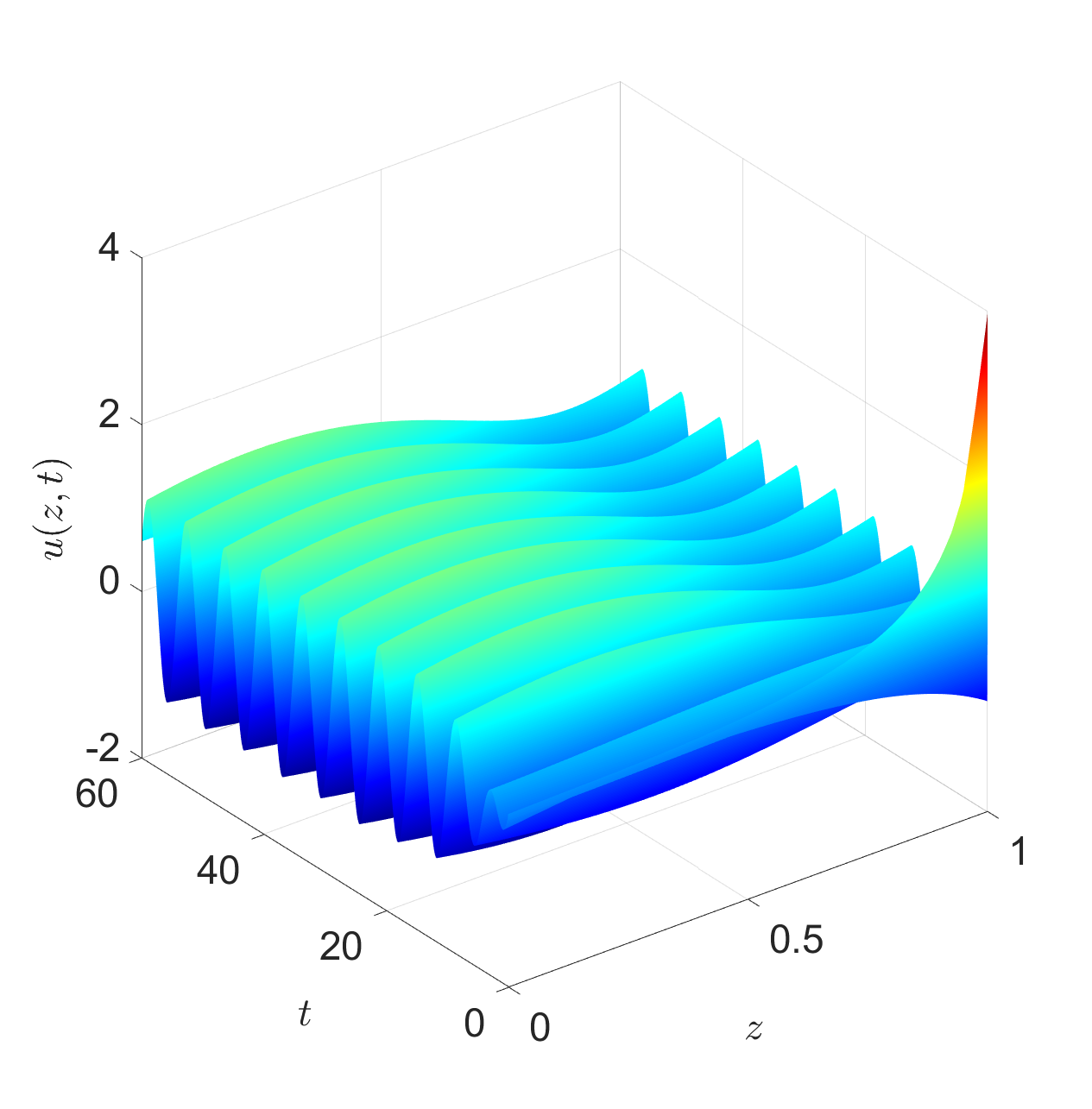}
	}
	\caption{Evolution of $X$ and $u$ {for the closed-loop system \eqref{original system} in the presence of} different boundary disturbances  {when $m_0=2$ and $j=j_0=0$}}
	\label{fig6}
\end{figure}

\begin{figure}[htbp]
	\centering
	\subfigure[{Evolution of $|X(t)|+\sup_{z\in (0,1)}\left|u(z,t)\right|$ {for the closed-loop system \eqref{original system} in the presence of}   {different internal      and in-domains disturbances and initial data} {when $j_1=j_2=0$}} ]{
		%		\label{fig:3}
		\includegraphics[width=0.35\textwidth,height=6pc]{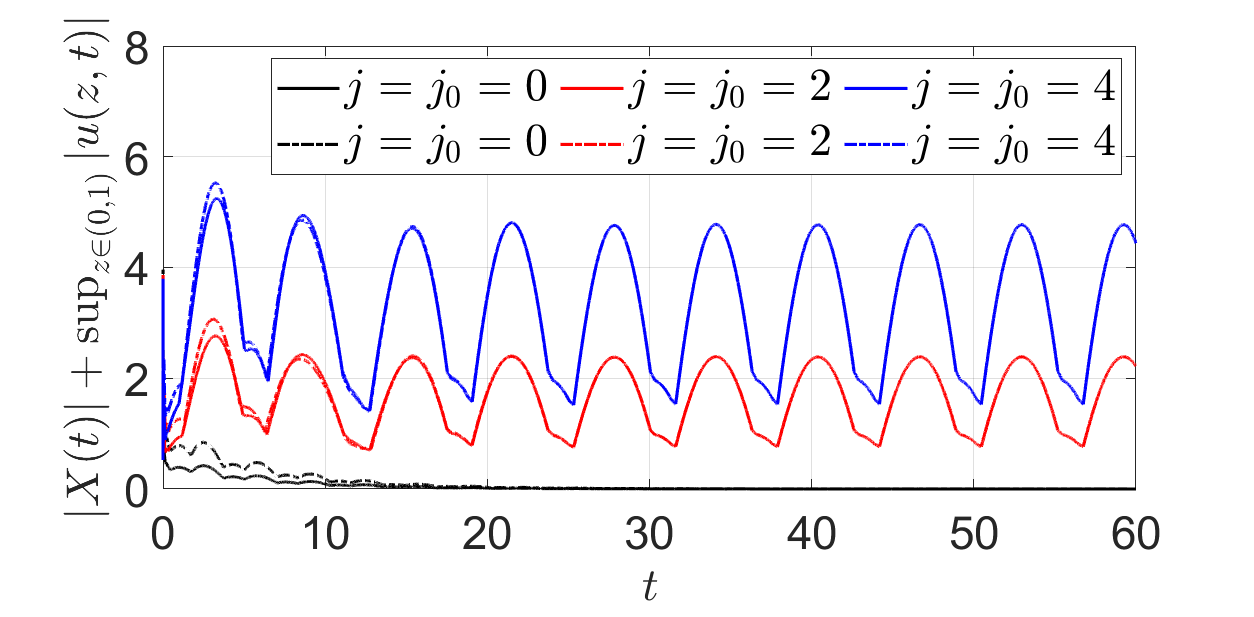}
	}\noindent\hspace{10pt}
	\subfigure[{Evolution of $|X(t)|+\sup_{z\in (0,1)}\left|u(z,t)\right|$ {for the closed-loop system \eqref{original system} in the presence of}  different boundary disturbances and initial data {when $j=j_0=0$}}]{
		%		\label{fig:4}
		\includegraphics[width=0.35\textwidth,height=6pc]{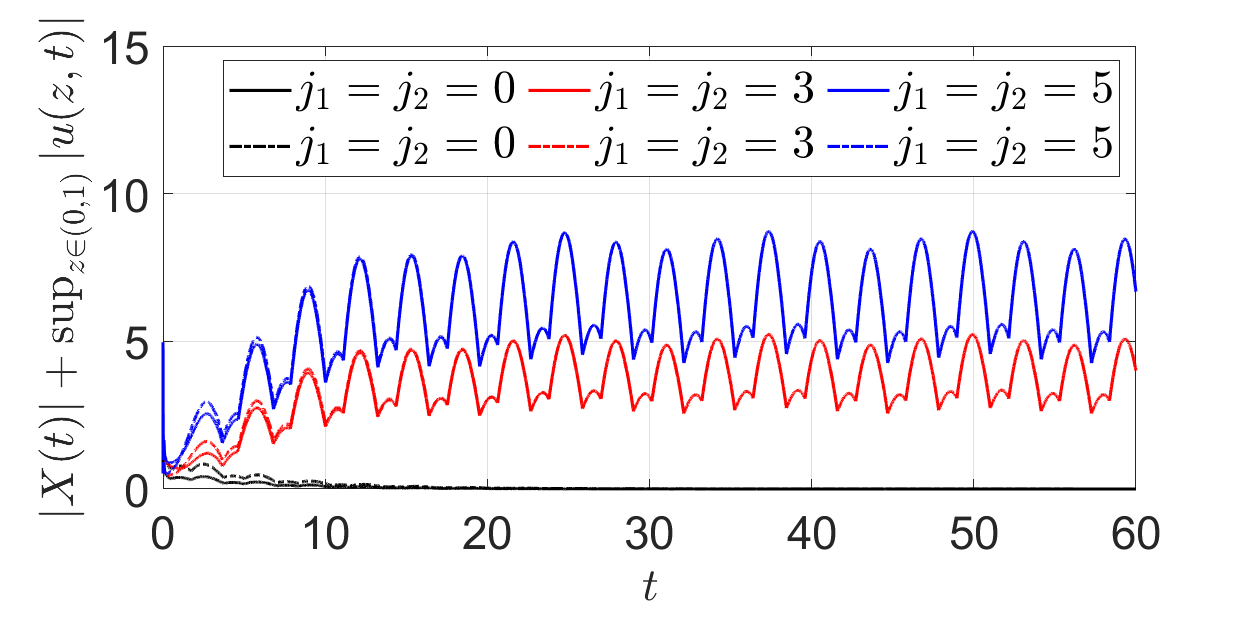}
	}
	\caption{Evolution of  $|X(t)|+\sup_{z\in (0,1)}\left|u(z,t)\right|$  {for the closed-loop system \eqref{original system} in the presence of}  different disturbances {and initial data}}
	\label{fig7}
\end{figure}

\section{Conclusion}\label{conclusion}
This paper addressed the problem of input-to-state stabilization for an ODE cascaded by a parabolic equation in the presence of Dirichlet-Robin boundary disturbances and in-domain disturbances. We employed the backstepping method  to design  a continuous state feedback controller, which allows us to avoid the use of sliding mode control and the boundedness parameter of disturbances.
Then, we used the generalized Lyapunov method to establish the ISS in the $\sup$-norm for the cascaded system with different  disturbances. Notably, under a generic setting, we  showed how to assess the ISS for parabolic PDEs with  space-time-varying coefficients in the presence of different  disturbances.  % For well-posedness analysis,   semigroup method  was employed to prove the well-posedness of system's  solution  in the space of continuous functions.
Future work includes  designing  an observer and an output feedback controller for a broader class of coupled ODE-PDE systems under the framework of ISS  and applying the ISS theory of PDEs to engineering problems.

\appendix
\section{Proof of Proposition \Rref{kernel exitence} (i)}\label{solution equation}
We prove Proposition \Rref{kernel exitence}  {(i)} by using the method of successive approximations as in, e.g.,\cite{Krstic2008book}. Let $\xi:=z+y$  and  $\eta:=z-y$, we have $\eta \in[0,1]$ and $\xi \in[\eta, 2-\eta]$. Let $G(\xi, \eta):=k(z, y)=k\left(\frac{\xi+\eta}{2}, \frac{\xi-\eta}{2}\right)$. Thus, the gain kernel PDE \eqref{kernel function} becomes
\begin{subequations}
	\begin{align}
		G _{\xi \eta}(\xi, \eta)=&\frac{1}{4} \left(b\left(\frac{\xi-\eta}{2}\right)+\lambda_0\left(\frac{\xi+\eta}{2}\right)\right) G(\xi, \eta), \label{varphi1}\\
		G_ \xi(\xi, 0)=&-\frac{1}{4}\left(b\left(\frac{\xi}{2}\right)+\lambda_0\left(\frac{\xi}{2}\right)\right) , \label{varphi2}\\
		q G(\xi, \xi)=& G_{\xi}(\xi, \xi)-G_\eta(\xi, \xi)-c(\xi)+\alpha(\xi)B+\int_{0}^{\xi}G(\xi+\tau,\xi-\tau)c(\tau)\text{d}\tau ,\label{varphi3} \\
		G(0,0)=&0.\notag
	\end{align}
\end{subequations}
Integrating \eqref{varphi1} w.r.t. $\eta$ from $0$ to $\eta$, and by \eqref{varphi2}, we get
\begin{align}\label{xi}
	G_{\xi}(\xi, \eta)=&-\frac{1}{4}\left(b\left(\frac{\xi}{2}\right)+\lambda_0\left(\frac{\xi}{2}\right)\right)+\frac{1}{4} \int_0^\eta \left(b\left(\frac{\xi-s}{2}\right)+\lambda_0\left(\frac{\xi+s}{2}\right)\right) G(\xi, s) \text{d} s.
\end{align}

 Next, we integrate \eqref{xi} w.r.t. $\xi$ from $\eta$ to $\xi$ to obtain
\begin{align}\label{xieta1}
	G(\xi, \eta)  =G(\eta, \eta)-\frac{1}{4} \int_\eta^{\xi}\left(b\left(\frac{\tau}{2}\right)+\lambda_0\left(\frac{\tau}{2}\right)\right) \text{d} \tau +\frac{1}{4} \int_\eta^\xi \int_0^\eta\left(b\left(\frac{\tau-s}{2}\right)+\lambda_0\left(\frac{\tau+s}{2}\right)\right) G(\tau ,s) \text{d} s \text{d} \tau.
\end{align}
By \eqref{varphi3}, we get
\begin{align*}
	\frac{\text{d}}{\text{d} \xi}(G(\xi, \xi))  =&G_\xi(\xi, \xi)+G_\eta(\xi, \xi)\notag\\
	=&2G_\xi(\xi, \xi)+\int_0^\xi G(\xi+\tau, \xi-\tau)c (\tau)\text{d} \tau-c(\xi) +\alpha(\xi) B-q G(\xi, \xi) \notag\\
	= & -\frac{1}{2}\left(b\left(\frac{\xi}{2}\right)+\lambda_0\left(\frac{\xi}{2}\right)\right)+\frac{1}{2} \int_0^{\xi}\left(b\left(\frac{\xi-s}{2}\right)+\lambda_0\left(\frac{\xi+s}{2}\right)\right) G(\xi, s)\text{d} s \notag\\
	& +\int_0^{\xi} G(\xi+\tau, \xi-\tau) c(\tau) \text{d} \tau-c(\xi)  +\alpha(\xi) B-q G(\xi, \xi).
\end{align*}
According to the method of variation of constants, we have
\begin{align*}
	G(\xi, \xi)=&\frac{1}{2} \int_0^{\xi}  \int_0^\tau e^{q(\tau-\xi)}\left(b\left(\frac{\tau-s}{2}\right)+\lambda_0\left(\frac{\tau+s}{2}\right)\right)G(\tau, s) \text{d} s \text{d} \tau
	 +\int_0^{\xi}  \int_0^\tau e^{q(\tau-\xi)}G(\tau+s, \tau-s)c (s) \text{d} s \text{d} \tau \\
	& +\int_0^\xi e^{q(\tau-\xi)}\left(\alpha(\tau) B-c(\tau)-\frac{1}{2}\left(b\left(\frac{\tau}{2}\right)+\lambda_0\left(\frac{\tau}{2}\right)\right)\right) \text{d} \tau.
\end{align*}
Substituting $\eta$ for $\xi$ yields
\begin{align}\label{G(eta,eta)}
	G(\eta, \eta)= & \frac{1}{2} \int_0^\eta  \int_0^\tau e^{q(\tau-\eta)}\left(b\left(\frac{\tau-s}{2}\right)+\lambda_0\left(\frac{\tau+s}{2}\right)\right) G(\tau, s) \text{d} s \text{d} \tau
	 +\int_0^\eta  \int_0^\tau e^{q(\tau-\eta)}G(\tau+s, \tau-s) c(s) \text{d} s \text{d} \tau \notag\\
	& +\int_0^\eta e^{q(\tau-\eta)}\left(\alpha(\tau) B-c(\tau)-\frac{1}{2}\left(b\left(\frac{\tau}{2}\right)+\lambda_0\left(\frac{\tau}{2}\right)\right)\right) {\text{d} \tau}.
\end{align}
By plugging \eqref{G(eta,eta)} into \eqref{xieta1}, we  get
\begin{align*}
	G(\xi, \eta)=&\int_0^\eta e^{q(\tau-\eta)}\left(\alpha(\tau) B-c(\tau)-\frac{1}{2}\left(b\left(\frac{\tau}{2}\right)+\lambda_0\left(\frac{\tau}{2}\right)\right)\right) \text{d} \tau -\frac{1}{4} \int_\eta^\xi\left(b\left(\frac{\tau}{2}\right)+\lambda_0\left(\frac{\tau}{2}\right)\right) \text{d} \tau\\
	&+\frac{1}{4} \int_\eta^{\xi} \int_0^\eta\left(b\left(\frac{\tau-s}{2}\right)+\lambda_0\left(\frac{\tau+s}{2}\right)\right) G(\tau, s) \text{d} s \text{d} \tau \notag\\
	 &+\frac{1}{2} \int_0^\eta  \int_0^\tau e^{q(\tau-\eta)}\left(b\left(\frac{\tau-s}{2}\right)+\lambda_0\left(\frac{\tau+s}{2}\right)\right) G(\tau, s) \text{d} s \text{d} \tau \\
	& +\int_0^\eta  \int_0^\tau e^{q(\tau-\eta)}G(\tau+s, \tau-s)c( s) \text{d} s \text{d} \tau.
\end{align*}
We rewrite  the integral equation as
\begin{align}\label{integral equation}
	G(\xi,\eta)=G_0(\xi, \eta)+\Phi_G(\xi, \eta),
\end{align}
where
\begin{align*}
	G_0(\xi, \eta):=&\int_0^\eta e^{q(\tau-\eta)}\left(\alpha(\tau) B-c(\tau)-\frac{1}{2}\left(b\left(\frac{\tau}{2}\right)+\lambda_0\left(\frac{\tau}{2}\right)\right)\right) \text{d}\tau-\frac{1}{4} \int_\eta^{\xi}\left(b\left(\frac{\tau}{2}\right)+\lambda_0\left(\frac{\tau}{2}\right)\right) \text{d} \tau
\end{align*}
and
\begin{align*}
	\Phi_G(\xi, \eta):=&\frac{1}{4} \int_\eta^{\xi} \int_0^\eta\left(b\left(\frac{\tau-s}{2}\right)+\lambda_0\left(\frac{\tau+s}{2}\right)\right) G(\tau, s) \text{d}s\text{d} \tau
	\notag\\
	& +\frac{1}{2} \int_0^\eta  \int_0^\tau e^{q(\tau-\eta)}\left(b\left(\frac{\tau-s}{2}\right)+\lambda_0\left(\frac{\tau+s}{2}\right)\right) G(\tau, s) \text{d} s \text{d}\tau \\
	& +\int_0^\eta  \int_0^\tau e^{q(\tau-\eta)}G(\tau+s, \tau-s) c(s) \text{d} s \text{d} \tau.
\end{align*}
Define the sequence ${G_n(\xi,\eta)}$ for $n\in \mathbb{N}_0$ via
\begin{align}\label{sequence}
	G_{n+1}(\xi,\eta):=G_0(\xi,\eta)+\Phi_{G_n}(\xi, \eta) ,
\end{align}
and let $\Delta G_n(\xi, \eta)=G_{n+1}(\xi, \eta)-G_n(\xi, \eta) $. It holds that
\begin{align*}
	\Delta G_{n+1}(\xi, \eta)=\Phi_{\Delta G_n}(\xi, \eta),
\end{align*}
and
\begin{align}\label{series}
	G_{n+1}(\xi, \eta)=G_0(\xi, \eta)+\sum_{j=0}^n \Delta G_j(\xi, \eta).
\end{align}

 By virtue of \eqref{sequence},   if ${G_{n}(\xi, \eta)}$ converges uniformly w.r.t. $(\xi, \eta)$ as $n\rightarrow\infty,$ then $G(\xi, \eta){:=}\lim_{n\rightarrow\infty}G_n(\xi,\eta)$ is a solution to the integral equation \eqref{integral equation}, thereby ensuring the existence of a kernel function $k$ satisfying the equation \eqref{kernel function}. In addition, in view of \eqref{series}, the convergence of ${G_n(\xi, \eta)}$ is equivalent to the convergence of the series $\sum_{j=0}^n \Delta G_j(\xi, \eta)$. Thus, it suffices to prove that the series $\sum_{j=0}^n \Delta G_j(\xi, \eta)$ converges uniformly w.r.t. $(\xi,\eta)$. Furthermore, recalling \eqref{equ4} and the choice of $\alpha, \lambda$, if we let \begin{align*}
 M:=&2\max\left\{  \sup_{z\in(0,1)} |\alpha(z) B|, 2\sup_{z\in(0,1)}\left|c(z)\right|,   \sup_{z\in(0,2)} \left|b\left(\frac{z}{2}\right)+\lambda_0\left(\frac{z}{2}\right)\right|\right\}<+\infty,
\end{align*}
then, it only need to prove that
\begin{align}\label{series1}
	\left|\Delta G_n(\xi, \eta)\right| \leq \frac{M^{n+2}}{(n+1) !}(\xi+\eta)^{n+1}
\end{align}
holds true for all $n\in \mathbb{N}_0$.

Now, we employ the method of induction to prove   \eqref{series1}. At first, since
\begin{align*}
	{\left|G_0(\xi, \eta)\right|}=&\left|\int_0^\eta e^{q(\tau-\eta)} \alpha(\tau) B \text{d} \tau-\int_0^\eta e^{q(\tau-\eta)} c(\tau) \text{d} \tau -\frac{1}{2} \int_0^\eta e^{q(\tau-\eta)}\left(b\left(\frac{\tau}{2}\right)+\lambda_0\left(\frac{\tau}{2}\right)\right) \text{d} \tau\right.\notag\\
	&\left.-\frac{1}{4} \int_\eta^{\xi}\left(b\left(\frac{\tau}{2}\right)+\lambda_0\left(\frac{\tau}{2}\right)\right) \text{d} \tau\right| \\
	\leq& \left(\frac{1}{2}M+\frac{1}{4}M+\frac{1}{4}M\right) \eta+\frac{1}{8}M\left(\xi-\eta\right) \\
	\leq&M \eta+\frac{1}{4}M(1-\eta) \\
	\leq& M ,
\end{align*}
it follows that
\begin{align*}
	\left|\Delta G_0(\xi, \eta)\right|= \left|\Phi_{G_0}\left(\xi,\eta\right)\right|
	\leq& \frac{1}{8} \int_\eta^{\xi} \int_0^\eta   M^2 \text{d} s \text{d} \tau+\frac{1}{4} \int_0^\eta \int_0^\tau M^2 \text{d} s \text{d} \tau +\frac{1}{4}\int_0^\eta \int_0^\tau  M^2 \text{d} s \text{d} \tau \\
	=&\frac{1}{8} M^2 \eta(\xi-\eta)+\frac{1}{8} M^2 \eta^2+\frac{1}{8} M^2 \eta^2 \\
	\leq& M^2(\xi+ \eta),
\end{align*}
which shows that \eqref{series1} holds true for $n{=}0$.

Assuming that \eqref{series1} holds for   $n\in \mathbb{N}$, we need to show that \eqref{series1} also holds for $n+1$. In fact, we have
\begin{align*}
	\left|\Delta G_{n+1}(\xi, \eta)\right| \leq& \frac{1}{8} \int_\eta^{\xi} \int_0^\eta  \frac{M^{n+3}}{(n+1)!}(\tau+ s)^{n+1} \text{d} s \text{d}  \tau +\frac{1}{4} \int_0^\eta \int_0^\tau \frac{M^{n+3}}{(n+1)!}(\tau+s)^{n+1} \text{d}  s \text{d}  \tau \\
	&+\frac{1}{4}\int_0^\eta\int_0^\tau  \frac{M^{n+2}}{(n+1)!}(2\tau)^{n+1} \text{d}  s \text{d}  \tau\\
	\leq& \frac{M^{n+3}}{8(n+2)!}\left(\xi+\eta\right)^{n+3}-\frac{M^{n+3}}{8(n+2)!}(2\eta)^{n+3}-\frac{M^{n+3}}{8(n+2)!}\xi^{n+3}+\frac{M^{n+3}}{8(n+2)!}\eta^{n+3}\notag\\
	&+\frac{M^{n+3}}{8(n+2)!}(2\eta)^{n+3}-\frac{M^{n+3}}{4(n+2)!}\eta^{n+3}
	+\frac{M^{n+2}}{16(n+2)!}(2\eta)^{n+3}\notag\\
	\leq&\frac{M^{n+3}}{8(n+2)!}(\xi+\eta)^{n+3}+\frac{M^{n+2}}{16(n+2)!}(\xi+\eta)^{n+3}\notag\\
	\leq	&\frac{M^{n+3}}{(n+2)!}(\xi+\eta)^{n+2}.
\end{align*}
 Therefore,  by induction, the estimate  \eqref{series1} holds true for all $n\in \mathbb{N}_0$.

Regarding the equivalence of the uniqueness of the solution to the kernel equation \eqref{kernel function} and the uniqueness of the solution to the integral equation \eqref{integral equation}, we  prove that any $C^2$-continuous solution $\overline{G}(\xi,\eta)$ to the equation \eqref{integral equation} can be approximated by the sequence $\{G_n(\xi,\eta)\}$ governed by \eqref{sequence}. More precisely, it suffices to show that
\begin{align}\label{original only}
	|G_n(\xi, \eta)-\overline{G}(\xi ,\eta)| \leq \frac{\delta M^{n}}{n !}(\xi+\eta)^n
\end{align}
holds true for all $n\in \mathbb{N}_0$, where $\delta:=\max_{\eta\in[0,1],\xi\in[\eta,2-\eta]} | \Phi_{\overline{G}}(\xi, \eta) |$. This is because \eqref{original only} ensures that ${G_n(\xi,\eta)}$ converges uniformly to $\overline{G}(\xi ,\eta)$ while the uniqueness of $ \lim_{n\rightarrow\infty}G_n(\xi,\eta)=:G(\xi, \eta)$ leads to $\overline G(\xi, \eta){\equiv }G(\xi ,\eta)$. Therefore, if the solution to \eqref{integral equation} is unique, then the solution to \eqref{kernel function} is also unique.

Now we prove \eqref{original only} by induction. First of all, since $\overline{G}(\xi, \eta)$ is a solution {to} the {equation} \eqref{integral equation},   applying the definition of $\Phi_{\overline{G}}$, we obtain that
\begin{align*}
	\left|G_0(\xi, \eta)-\overline{G}(\xi ,\eta)\right|=|G_0(\xi, \eta)-\left(G_0(\xi, \eta)+\Phi_{\overline{G}}\left(\xi,\eta\right)\right)|
	=\left|-\Phi_{\overline{G}}(\xi, \eta)\right| \leq \delta,
\end{align*}
which indicates that \eqref{original only} holds true for $n=0$.

Assuming that \eqref{original only} holds true for   $n{\in}\mathbb{N}$, we need to show that \eqref{original only} also holds true for $n+1$. In fact, by this assumption and the linearity of $\Phi_{H}$ w.r.t. $H$, we have
\begin{align*}
	 \left|G_{n+1}(\xi, \eta)-\overline{G}(\xi, \eta)\right|
	=&|\Phi _{G_ n}(\xi, \eta)-\Phi_{\overline{G}}(\xi, \eta)|\\
	=& \left|\Phi_{G_ n-\overline{G}}(\xi,\eta)\right| \\
	%\leq& \delta\left|\Phi_{\Delta G_{n-1}}\right| \\
	%\leq&\frac{\delta}{4} \int_\eta^{\xi} \int_0^\eta\left|\left(\Lambda\left(\frac{\tau-s}{2}\right)+\lambda\right) \Phi_{\Delta G_{n-1}}(\tau, s)\right|  \text{d}s\text{d} \tau\\
	%& +\frac{\delta}{2} \int_0^\eta  \int_0^\tau \left|e^{q(\tau-\eta)}\left(\Lambda\left(\frac{\tau-s}{2}\right)+\lambda\right) \Phi_{\Delta G_{n-1}}(\tau, s) \right|\text{d} s \text{d}\tau \\
	%& +\delta\int_0^\eta  \int_0^\tau \left|e^{q(\tau-\eta)}\Phi_{\Delta G_{n-1}}(\tau+s, \tau-s) c(s) \right|\text{d} s \text{d} \tau\notag\\
	\leq& \frac{\delta}{8}  \int_\eta^{\xi} \int_0^\eta \frac{M^{n+1}}{n!}(s+\tau)^n \text{d} s \text{d} \tau + \frac{\delta}{4}  \int_0^{\eta} \int_0^\tau\frac{M^{n+1}}{n!}(s+\tau)^n \text{d} s \text{d} \tau
	\notag\\
&+\frac{\delta}{4} \int_0^{\eta} \int_0^\tau \frac{M^{n+1}}{n!}(2\tau)^n\text{d} s \text{d} \tau\notag\\
	= & \frac{\delta}{8} \int_\eta^{\xi}\left(\frac{M^{n+1}}{(n+1)!}(\eta+\tau)^{n+1}-\frac{M^{n+1}}{(n+1)!} \tau^{n+1}\right) \text{d} \tau
	\notag\\
& +\frac{\delta}{4} \int_0^\eta\left(\frac{M^{n+1}}{(n+1)!}(2 \tau)^{n+1}-\frac{M^{n+1}}{(n+1)!} \tau^{n+1}\right)\text{d} \tau  +\frac{\delta}{8} \int_0^\eta \frac{M^{n+1}}{(n+1)!}(2 \tau)^{n+1} \text{d} \tau \\
	= &  \frac{\delta M^{n+1}}{8(n+2)!}(\eta+\xi)^{n+2}- \frac{\delta M^{n+1}}{8(n+2)!}(2 \eta)^{n+2}-\frac{\delta M^{n+1}}{8(n+2)!} \xi^{n+2} +\frac{\delta M^{n+1}}{8(n+2)!} \eta^{n+2} \notag\\
& + \frac{\delta M^{n+1}}{8(n+2)!}(2 \eta)^{n+2}
	 -\frac{\delta M^{n+1}}{4(n+2)!} \eta^{n+2}
	+ \frac{\delta M^{n+2}}{16(n+2)!}(2 \eta)^{n+2} \\
	\leq & \frac{\delta M^{n+1}}{(n+1)!}(\xi+\eta)^{n+1}.
\end{align*}
Therefore, \eqref{original only}   holds true for all  $n\in \mathbb{N}_0$.
\section{Proof of Proposition~\Rref{equivelent}}\label{system-equivelent}
We prove Proposition~\Rref{equivelent} in two steps.

\textbf{Step 1:}  prove that the  transformation    \eqref{transformation}, which maps $(X,u)$ to $(X,w)$, is invertible in  $C (\mathbb{R}_{>0};  \mathbb{R}^N)  \times    C(\mathbb{R}_{\geq0};C((0,1);\mathbb{R}))$. Note that the mapping from $(X,u)$ to $(X,w)$ is linear, it suffices to prove that it is bijective.  More precisely, for any given  $(X,w)\in C (\mathbb{R}_{>0};  \mathbb{R}^N)  \times    C(\mathbb{R}_{\geq0};C((0,1);\mathbb{R}))$,  it suffices to prove that there exists a unique function  $u\in C(\mathbb{R}_{\geq0};C((0,1);\mathbb{R}))$ such that \eqref{transformation4b} is fulfilled. Furthermore,  letting
\begin{align}\label{def-u0}
	u_0(z,t):=w(z,t)+\alpha(z)X(t)
\end{align}
 and
\begin{align} \label{vv}
 \Psi_{u}(z, t):=\int_0^z k(z, y) u(y, t) \text{d} y,
\end{align}
%the equation \eqref{transformation4b} can be rewritten as
we only need to prove that there exists a unique function  $u\in C(\mathbb{R}_{\geq0};C((0,1);\mathbb{R}))$ such that the following equation
\begin{align}\label{uphi}
	u(z, t)= u_0(z,t)+\Psi_{u}(z, t)
\end{align}
 is fulfilled.

 {First,} we  prove the existence of a solution $u\in C(\mathbb{R}_{\geq0};C((0,1);\mathbb{R}))$  to the equation \eqref{uphi}.  Let
\begin{align}\label{un+1}
	u_{n+1}(z, t):=u_0(z,t)+\Psi_{u_n}(z, t),  {n\in\mathbb{N}_0}  
\end{align}
and
\begin{align*}%\label{zengliang}
	\Delta u_n(z, t):=u_{n+1}(z, t)-u_n(z, t), {n\in\mathbb{N}_0}.
\end{align*}
Then,   we get
\begin{align}\label{detalun}
\Delta u_{n+1}(z, t) =\Psi_{u_{n+1}}(z, t)-\Psi_{u_n}(z, t)=\Psi_{\Delta u_{ n}}(z, t),  {n\in\mathbb{N}_0} 
\end{align}
with
\begin{align}\label{detalu0}
	\Delta u_{0}(z, t)  =\Psi_{u_{0}}(z, t).
\end{align}
 Therefore, \eqref{un+1} can be written as
%				\begin{align*}
	%				& \Delta u_{n+1}(z, t)=u_{n+2}-u_{n+1} \\
	%			& =u_0+\Phi_{u_{ n+1}}-u_0-\Phi_{u_{n}} \\
	%			& =\Phi_{\Delta u_{ n}}(z, t) \\
	%			\end{align*}
\begin{align}\label{seriesun}
	u_{n+1}(z, t)=u_0(z, t)+\sum_{j=0}^n \Delta u_j(z, t).
\end{align}
%According to \eqref{un+1}, if the sequence $u_n(z,t)$ converges uniformly w.r.t. $(z,t)$ as $n\to\infty$, then the limit $ u(z,t):=\lim_{n\to\infty}u_n(z,t)$
	%is a solution of \eqref{uphi}. Furthermore, by \eqref{seriesun}, the convergence of the sequence $\{u_n(z,t)\}$ is equivalent to the convergence of the series $\sum_{j=0}^{\infty}\Delta u_j(z,t).$
	As  in the proof of Proposition~\ref{kernel exitence}(i) (see Appendix~\ref{solution equation}), to prove the existence of a solution to \eqref{uphi}, it suffices to show  that the series $\sum_{j=0}^n \Delta u_j(z, t)$ converges uniformly w.r.t. $(z,t)$. In particular,   we show that
\begin{align}\label{existence-un}
	\left|\Delta u_n(z, t)\right| \leq  \frac{c_0  \overline k^{n+1}}{(n+1)!}z^{n+1}\left(\sup _{z \in (0,1)}|w(z, t)|+|X(t)|\right)
\end{align}
holds true for all $n\in\mathbb{N}_0$, where $c_0:=\max\left\{1,\max_{z\in[0,1]}|\alpha(z) |\right\}$  and $\overline k:=\max _{(z,y) \in \overline{\Omega}}|k(z, y)|$.
Indeed, we prove \eqref{existence-un} by induction. Note that by the definition of $u_0$ (see \eqref{def-u0}), we get
\begin{align}\label{u0leq}
	\left|u_0(z, t)\right|=|w(z, t)+\alpha(z) X(t)|
	%				\leq&  c_0|w(z, t)| \notag\\
	\leq c_0 \left(\sup _{z \in (0,1)}|w(z, t)|+|X(t)|\right).
\end{align}
It follows that
\begin{align*}	
	\left|\Delta u_0(z, t)\right|	=\left|\Psi_{u_0}(z, t)\right| =\left|\int_0^z k(z, y) u_0(y, t) \text{d} y\right|
	%		& \leq  c_0\max _{(z,y) \in \Omega}|k(z, y)| \max _{(z, t) \in \overline{Q}_T }|w(z, t)| z \\
	\leq	 c_0 \overline k z\left(\sup _{z\in(0,1)}|w(z, t)|+|X(t)|\right), \forall t\in\mathbb{R}_{\geq 0}.
	%			& \left|\Delta u_1(z, t)\right|=\left|\dot{q}_{\Delta u_0}(z, t)\right|=\left|\int_0^z k(z, y) \Delta u_0(y, t) d y\right| \\
	%			& \leq M \int_0^z\left|C_0 M \max \right| w(z, t)|\cdot y| d y \\
	%			& =c_0 M^2 \max |w(z, t)| \cdot \frac{1}{2} z^2 \\ .
\end{align*}
Therefore, \eqref{existence-un} holds true for $n=0$.

 Assume that \eqref{existence-un}	 holds true for    $n\in\mathbb{N}$; it then need to show that \eqref{existence-un} holds true for $n+1$. Indeed, for $n+1$, by direct computations, we have
\begin{align*}
	\left|\Delta u_{n+1}(z, t)\right| =	\left|\Psi_{\Delta u_{n+1}}(z, t)\right|=\left|\int_0^z k(z, y) \Delta u_{n+1}(y, t) \text{d} y\right|
	\leq   \frac{c_0  \overline k^{n+2}}{(n+2)!}z^{n+2}\left(\sup _{z \in (0,1)}|w(z, t)|+|X(t)|\right).
\end{align*}
Therefore, {we conclude that \eqref{existence-un} holds true for all $n\in\mathbb{N}_0$.}

{Now, we prove    the  uniqueness  of the solution to \eqref{uphi}.  As in the proof of Proposition~\ref{kernel exitence}(i) (see Appendix~\ref{solution equation}), we only need to prove  that any solution $\overline u \in C(\mathbb{R}_{\geq0};C((0,1);\mathbb{R}))$  to the equation \eqref{uphi} can be approximated by the sequence ${u_n(z,t)}$ governed by \eqref{un+1}. More precisely, we need to show that}
\begin{align}\label{unique-un}
	|u_n(z,t)-\overline u(z,t)|\leq \frac{\delta_0 {\overline{k}}^n}{n!}z^n
\end{align}
holds true for all $n\in\mathbb{N}_0$, where $\delta_0:=\sup_{z\in(0,1)}|\Psi_{\overline u}(z,t)|$.
 
 We once again  employ  the method of induction to prove it.  Indeed, by the definition of $\Psi_{u}$ {(see \eqref{vv})}, we obtain
%\begin{align*}
	$|u_0(z,t)-\overline u(z,t)|\leq \delta_0$,
%\end{align*}
which shows that   \eqref{unique-un} holds true for $n=0$. Supposing that \eqref{unique-un} holds true  for  $n\in\mathbb{N}$, we deduce for  $n+1$ that
\begin{align*}
	|u_{n+1}(z,t)-\overline u(z,t)|=\left|\Psi_{u_n-\overline u}(z,t)\right|
	\leq\left|\int_0^z k(z,y)\frac{\overline k^n\delta_0}{n!}y^n\text{d}y\right|
	\leq\frac{\overline k^{n+1}\delta_0}{(n+1)!}z^{n+1}, \forall t\in\mathbb{R}_{\geq 0}.
\end{align*}
Therefore, \eqref{unique-un} holds true for all $n\in\mathbb{N}_0$. 

{We conclude that  the transformation \eqref{transformation}, which maps $(X,u)$ to $(X,w)$, is invertible in  $C (\mathbb{R}_{>0};  \mathbb{R}^N)  \times    C(\mathbb{R}_{\geq0};C((0,1);\mathbb{R}))$.}

\textbf{Step 2:} {prove that   the inverse   of transformation \eqref{transformation}}, i.e., the mapping from $(X,w)$ to $(X,u)$, is bounded in $ C (\mathbb{R}_{>0};  \mathbb{R}^N)  \times    C(\mathbb{R}_{\geq0};C((0,1);\mathbb{R}))$.
{More specifically, we show that there exists a constant $M_1>0$ such that} \begin{align}\label{x+u}
	|X(t)|+ \sup_{z\in(0,1)}|u(z,t)|\le M_1 \left(\sup_{z\in(0,1)}|w(z,t)| + |X(t)|\right),  \forall t\in\mathbb{R}_{\geq 0}.
\end{align}
Indeed, substituting \eqref{uphi} into \eqref{vv} yields the following integral equation
\begin{align}\label{v-original}
	\Psi_{u}(z,t) = \int_0^z k(z,y)u_0(y,t)\mathrm{d}y + \int_0^z k(z,y)\Psi_{u}(y,t)\mathrm{d}y.
\end{align}
According to the definitions of $u$ (see \eqref{uphi}) and $u_0$ (see \eqref{def-u0}), we obtain
\begin{align}\label{Psiuzt}
	|u(z,t)|\leq|u_0(z,t)|+|\Psi_u(z,t)|
	\leq c_0\left({\sup_{z\in(0,1)}|w(z,t)|}+|X(t)|\right)+|\Psi_u(z,t)|.
\end{align}
 {Therefore, to prove that \eqref{x+u} holds, it is   necessary to estimate $|\Psi_u(z,t)|$ in \eqref{Psiuzt}.    Furthermore, letting}
\begin{align}\label{psiu0}
	{\Psi}_{u_0}(z,t) := \int_0^z k(z,y)u_0(y,t)\mathrm{d}y
\end{align}
and
\begin{align*}
	{\Psi}_{u_n}(z,t) := \int_0^z k(z,y){\Psi}_{u_{n-1}}(y,t)\mathrm{d}y,  n \ge 1,
\end{align*}
it suffices to prove that  
\begin{align}\label{vnhold}
	|{\Psi}_{u_n}(z,t)| \le \frac{c_0{\overline k}^{n+1}}{n!} \left(\sup_{z\in(0,1)}|w(z,t)|+|X(t)|\right)z^{n}, n\in\mathbb{N}_0.
\end{align}
holds true for all   $t\in\mathbb{R}_{\geq 0}$.
 {First, when $n=0$,  by \eqref{psiu0} and \eqref{u0leq}, we obtain}
\begin{align*}
	|{\Psi}_{u_0}(z,t)| \leq
	 c_0\overline k\left(\sup_{z\in(0,1)}|w(z,t)|+|X(t)|\right),  \forall t\in\mathbb{R}_{\geq 0}.
\end{align*}
 Hence, \eqref{vnhold} holds true for $n=0$.
Assuming that \eqref{vnhold} holds true for  $n\in\mathbb{N}$,  by induction,  we need to show that \eqref{vnhold} also holds true for $n+1$. Indeed, for $n+1$, we have
%\begin{align*}
%|v_{n-1}(z,t)| \le \frac{{m_{c_0}}^{n} }{(n-1)!} %\left(\sup_{z\in(0,1)}|w(z,t)|+|X(t)|\right)z^{n-1}.
%\end{align*}
\begin{align*}
	|{\Psi}_{u_{n+1}}(z,t)| \le \overline k \int_0^z |{\Psi}_{u_{n}}(y,t)|\mathrm{d}y
	\le&  \overline k\int_0^z \frac{c_0{\overline k}^{n+1} }{n!}\left(\sup_{z\in(0,1)}|w(z,t)|+|X(t)|\right) y^{n} \mathrm{d}y\notag\\
	=& \frac{c_0{\overline k}^{n+2} }{(n+1)!}\left(\sup_{z\in(0,1)}|w(z,t)|+|X(t)|\right) z^{n+1}.
\end{align*}
  Therefore, \eqref{vnhold} holds true for all $n\in\mathbb{N}_0$. 
  
  Then, we infer that
\begin{align*}
	\sum_{n=0}^\infty  |{\Psi}_{u_n}(z,t)| \le \left(\sup_{z\in(0,1)}|w(z,t)|+|X(t)|\right)  \sum_{n=0}^\infty \frac{c_0{\overline k}^{n+1}}{n!}
	= \left(\sup_{z\in(0,1)}|w(z,t)|+|X(t)|\right) {c_0{\overline k}} e^{\overline k} < +\infty,   \forall t\in\mathbb{R}_{\geq 0}.
\end{align*}
Thus, the series  $\sum_{n=0}^\infty  |{\Psi}_{u_n}(z,t)|$ converges absolutely and uniformly w.r.t.  $z\in(0,1)$, and its sum function   ${\Psi_u}(z,t)$  satisfies
\begin{align}\label{6}
	\sup_{z\in(0,1)}|{\Psi_u}(z,t)| \le {c_0{\overline k}} e^{{\overline k}} \left(\sup_{z\in(0,1)}|w(z,t)|+|X(t)|\right), \forall t\in\mathbb{R}_{\geq 0}.
\end{align}

From \eqref{Psiuzt}, we obtain
\begin{align*}
	\sup_{z\in(0,1)}|u(z,t)|
	\le&  c_0(1+{{\overline k}} e^{{\overline k}}) \left(\sup_{z\in(0,1)}|w(z,t)|+|X(t)|\right), \forall t\in\mathbb{R}_{\geq 0},
\end{align*}
and hence
\begin{align*}
	|X(t)|+ \sup_{z\in(0,1)}|u(z,t)|\le M_1 \left(\sup_{z\in(0,1)}|w(z,t)| + |X(t)|\right),  \forall t\in\mathbb{R}_{\geq 0},
\end{align*}
where $M_1:= c_0(1+{{\overline k}} e^{{\overline k}})+1$  is a constant independent of $t$. {Therefore,}  the inverse of transformation  \eqref{transformation} is bounded in  $C (\mathbb{R}_{>0};  \mathbb{R}^N)  \times    C(\mathbb{R}_{\geq0};C((0,1);\mathbb{R}))$.

%\section{References}
\bibliographystyle{abbrv}
\bibliography{reference}
\end{document}